\documentclass{siamltex}
\usepackage{amsfonts,amsmath,amssymb,stmaryrd}
\usepackage{mathrsfs,mathtools}
\usepackage{enumerate}
\usepackage{xspace} 
\usepackage{hyperref}
\usepackage{esint}
\usepackage{graphicx}
\usepackage{subfigure}

\usepackage{booktabs}
\usepackage{psfrag}
\usepackage{bbold}

\DeclareMathAlphabet{\mathpzc}{OT1}{pzc}{m}{it}

\newcommand{\R}{\mathbb{R}}
\newcommand{\ysf}{\mathsf{y}}
\newcommand{\usf}{\mathsf{u}}
\newcommand{\bsf}{\boldsymbol{\mathsf{b}}}
\newcommand{\fsf}{\mathsf{f}}
\newcommand{\vsf}{\mathsf{v}}
\newcommand{\asf}{\mathsf{a}}
\newcommand{\bbsf}{\mathsf{b}}
\newcommand{\qsf}{\mathsf{q}}
\newcommand{\psf}{\mathsf{p}}
\newcommand{\T}{\mathscr{T}}
\newcommand{\Tr}{\mathbb{T}}
\newcommand{\wsf}{\mathsf{w}}
\newcommand{\norm}[1]{{{|\!|\!|} #1 {|\!|\!|}}}
\newcommand{\V}{\mathbb{V}}
\newcommand{\U}{\mathbb{U}}

\usepackage[usenames,dvipsnames]{color}
\newcommand{\EO}[1]{{\color{black}#1}}
\newcommand{\RR}[1]{{\color{black}#1}}

\newtheorem{remark}[theorem]{Remark}
\numberwithin{equation}{section}

\title{\EO{A} posteriori error estimators for stabilized finite element approximations of an
optimal control problem\thanks{AA is supported by CONICYT through FONDECYT project 1170579. EO is supported by CONICYT through FONDECYT project 3160201. RR is supported by BASAL PFB03 CMM project, Universidad de Chile.}}

\author{Alejandro Allendes\thanks{Departamento de Matem\'atica, Universidad T\'ecnica Federico Santa Mar\'ia, Valpara\'iso, Chile.
\texttt{alejandro.allendes@usm.cl}}
\and
Enrique Ot\'arola\thanks{Departamento de Matem\'atica, Universidad T\'ecnica Federico Santa Mar\'ia, Valpara\'iso, Chile.
\texttt{enrique.otarola@usm.cl}}
\and
Richard Rankin\thanks{Departamento de Matem\'atica, Universidad T\'ecnica Federico Santa Mar\'ia, Valpara\'iso, Chile.
\texttt{richard.rankin@usm.cl}}
}

\date{Draft version of \today.}

\begin{document}

\maketitle

\begin{abstract}
We \RR{derive a} posteriori error estimators for an optimal control problem governed by a convection--reaction--diffusion equation; control constraints are also considered. We consider a family of low--order stabilized finite element methods to approximate the solutions to the state and adjoint equations. \RR{We obtain a fully computable a posteriori error estimator for the optimal control problem.} All the constants that appear in the upper bound for the error are fully specified. Therefore, the proposed estimator can be used as a stopping \RR{criterion} in adaptive algorithms. \RR{We also obtain a robust a posteriori error estimator for when the error is measured in a norm that involves the dual norm of the convective derivative. Numerical} examples, in two and three dimensions, are presented to illustrate the theory.
\end{abstract}



\section{Introduction}
The purpose of this work is to construct and \RR{analyze a} posteriori error estimators for a control--constrained optimal control problem involving a convection--reaction--diffusion equation as state equation. To describe our problem, let $\Omega$ be an open and bounded polytopal domain in $\R^{d}$, $d \in \{2,3\}$, with Lipschitz boundary $\partial \Omega$. Given a desired state $\ysf_{\Omega}\in L^{2}(\Omega)$, we define 
the cost functional
\begin{eqnarray}
\label{functional}
J(\ysf,\usf) = \frac{1}{2} \|\ysf-\ysf_{\Omega}\|_{L^{2}(\Omega)}^{2} + \frac{\vartheta}{2} \|\usf\|_{L^{2}(\Omega)}^{2},
\end{eqnarray}
where $\vartheta > 0$ denotes the so--called regularization parameter. We shall be concerned with the following optimal control problem: Find
\begin{equation}
\label{min}
\min\, J(\ysf,\usf)
\end{equation}
subject to the convection--reaction--diffusion equation
\begin{equation}
\label{state_equation}
-\nu\Delta\ysf + \bsf\cdot\nabla\ysf + \kappa\ysf = \fsf+\usf \quad \textrm{in}~\Omega, \qquad \ysf = 0 \quad \textrm{on}~\partial\Omega,
\end{equation}
and the control constraints
\begin{equation}
\label{control_constraint}
\usf\in\mathbb{U}_{ad}, \qquad \mathbb{U}_{ad}:=\{ \vsf \in L^2(\Omega):~\asf\leq \vsf(\boldsymbol{x}) \leq \bbsf \mbox{ for almost every }\boldsymbol{x}\mbox{ in }\Omega\}.
\end{equation}
Here, the bounds $\asf, \bbsf \in \R$ and are such that $\asf < \bbsf$ \RR{and $\fsf\in L^{2}(\Omega)$. Assumptions on $\nu$, $\bsf$ and $\kappa$} are deferred until later.

The a priori error analysis for standard finite element approximations of problem \eqref{min}--\eqref{control_constraint} has been well established; see \cite{MR756510,MR1993702,MR0686788,MR2122182,MR2843956,troltzsch2009finite} and the references therein. This analysis strongly relies on the error estimates involved in the approximation of \eqref{state_equation}. However, it is well known that applying the standard finite element method to \eqref{state_equation} produces poor results 
when convection--dominated regimes are considered. In order to overcome such a difficulty, in the last few decades a variety of finite element approaches\RR{, such as \emph{stabilized finite element methods},} have been proposed in the literature. In this work, we will focus on low-order conforming stabilized schemes, for which \cite{MR2454024} provides an extensive overview in the subject. In the context of optimal control, the numerical approximation of problem \eqref{min}--\eqref{control_constraint} relies additionally on the discretization of the so--called adjoint equation (see Section 2.12 in \cite{MR2583281}). Since \eqref{min}--\eqref{control_constraint} is intrinsically non--linear and presents a crosswind phenomena, an efficient method for solving such a control problem has to properly treat the oscillatory behaviours that occur when approximating $\bar \ysf$ and its adjoint variable $\bar \psf$ and resolve the boundary layers exhibited by the state and adjoint state. Failure to resolve boundary layers pollutes the numerical solution in the entire domain \cite{MR2595051}. Different stabilized finite element methods have been proposed to solve \eqref{min}--\eqref{control_constraint}; see \cite{MR2302057,MR2486088,MR2178571,MR2647025,weng2015stabilized,MR2463111,MR3246619}.
However, considering only stabilized schemes is not sufficient to efficiently approximate the solution to \eqref{min}--\eqref{control_constraint}; boundary or interior layers and possible geometric singularities need to be resolved. This motivates the design of stabilized adaptive finite element methods.

Adaptive procedures for obtaining finite element solutions are based on the so-called a posteriori error analysis, and it has a solid foundation for elliptic problems; see  \RR{\cite{AObook,BR1978,MR777265,MR2648380,MR3059294,ZZ1987}}. In contrast, the a posteriori error analysis for finite element approximations of optimal control problems has not yet been fully understood. We refer the reader to \cite{MR2434065,MR1887737,MR2373479} for contributions to the theory. An attempt to unify the available results in the literature is presented in \cite{MR3212590}, where the authors derive an important error equivalence that simplifies the a posteriori error analysis to, simply put, provide estimators for the state and adjoint equations. The analysis is based on the energy norm inherited by the state and adjoint equations. Recently, the authors of \cite{MR3485971} provided a general framework that complements the one developed in \cite{MR3212590}, and measures the error in a norm that is motivated by the objective. The analysis relies on the convexity of $\Omega$. Both approaches exploit the first--order optimality conditions to derive a posteriori error estimates. However, the derived error estimates involve several \EO{\textit{unknown}} constants in the analysis, in particular, in the upper bound for the error in terms of the proposed error estimators (see \cite[Theorem 3.2]{MR3212590} and \cite[\EO{Theorem 3.3}]{MR3485971}). \EO{Hence, in real computations, it will be unclear whether over or under estimation of the error has occurred.} In fact, in a practical setting, if the estimator is to be used as a stopping \RR{criterion}, then the constant involved in the upper bound for the error must be fully computable. 
This motivates the design and analysis of fully computable error estimators \cite{MR2652113,MR1648383,MR2114283,MR2253056,MR2423279} for our optimal control problem, which guarantee a genuine upper bound for the error in the sense that the value of the \RR{error estimator} is greater than or equal to the value of the error; see Theorem \ref{th:global_reliability}. 

\RR{One of the main aims} of this work is to develop an a posteriori error estimator with the following features:
\begin{itemize}
\item to be fully computable, in order to have at hand a stopping \RR{criterion} for the adaptive resolution;
\item to be applicable to a wide variety of low--order stabilized methods, allowing different combinations of stabilization terms for the state and adjoint equations.
\end{itemize}

\RR{We follow the a posteriori error analysis from \cite{MR3123245,allendes2015error} in order to obtain a fully computable error estimator for the convection--reaction--diffusion equation \eqref{state_equation}. With this estimator at hand, we provide what we believe is the first fully computable error estimator for the optimal control problem \eqref{min}--\eqref{control_constraint}}. \EO{However, the proposed fully computable estimator is not robust; the constant involved in the lower bound for the errror depends on $\nu$, $\bsf$ and $\kappa$. This motivates another aim of our work:} 
\begin{itemize}
\item \EO{to propose and analyze a robust a posteriori error estimator.}
\end{itemize}
\EO{To accomplish this task we follow \cite{MR3407239,MR2182149,MR3059294} and measure the error in a norm that involves the dual norm of the convective derivative. We comment that the derived robust estimator is not fully computable.}
%

This manuscript is organized as follows: Section \ref{Fully_computable} presents a general framework for constructing a fully computable a posteriori error estimator for an advection--reaction--diffusion problem in the presence of a general stabilized method. Section \ref{control_optimo} presents the finite element discretization of the optimal control problem with its a posteriori error analysis. A fully computable a posteriori error bound is provided for a general family of stabilization schemes. This is one of the highlights of our contribution. \RR{In Section \ref{robustsection}, we present the analysis of an alternative a posteriori error estimator which is robust with respect to a norm that involves the dual norm of the convective derivative.} Finally, Section \ref{numerical} presents numerical examples to illustrate the theory. 
\section{Preliminaries}\label{primeraaa}
We shall use standard notation for Sobolev and Lebesgue spaces, norms, and inner products. For a bounded domain $G\subset \RR{\R^d}$ $(\RR{d=2,3})$: $L^{2}(G)$ denotes the space of square integrable functions over $G$, $H^{1}(G)$ is the usual Sobolev space and $H_{0}^{1}(G)$ denotes the subspace of $H^{1}(G)$ consisting of functions whose trace is zero on $\partial G$. Let $(\cdot,\cdot)_{L^{2}(G)}$ denote the inner product in $L^{2}(G)$. 
The norm of $L^2(G)$ is denoted by $\|\cdot\|_{L^{2}(G)}$. We use bold letters to denote the vector-valued counterparts of 
spaces, e.g. $\mathbf{L}^{2}(G)=L^{2}(G)^d$.

Let $\T = \{ K \}$ be a conforming partition of $\bar \Omega$ into simplices $K$ in the sense of Ciarlet \cite{MR0520174}. We denote by $\Tr$ a collection of conforming and shape regular meshes that are refinements of an initial mesh $\T_0$. For a fixed $\T \in \Tr$, let
\begin{flushleft}
\begin{itemize}
\item $\mathcal{F}$ denote the set of all element edges(2D)/faces(3D); 
\item $\mathcal{F}_{I}\subset \mathcal{F}$ denote the set of interior edges(2D)/faces(3D), $\mathcal{F}_{\partial\Omega}\subset \mathcal{F}$ denote the set of boundary edges(2D)/faces(3D);
\item $\Omega_{n} = \{K\in\T:~ \boldsymbol{x}_{n} \in K \}$, the set of elements for which $\boldsymbol{x}_{n}$ is a vertex;
\item $\mathcal{F}_{n} = \{\gamma\in\mathcal{F}:~ \boldsymbol{x}_{n} \in \gamma \}$ denote the set of element edges(2D)/faces(3D) that have $\boldsymbol{x}_{n}$ as a vertex\RR{.}
\end{itemize}
\end{flushleft}
For an element $K\in\T$, let
\begin{itemize}
\item $\mathbb{P}_{n}(K)$ denote the space of polynomials on $K$ of total degree at most $n$;
\item $\mathcal{F}_{K}\subset \mathcal{F}$ denote the set containing the individual edges(2D)/faces(3D) of $K$;
\item $\Omega_{K}:=\{K'\in\T:~ \mathcal{F}_{K}\cap\mathcal{F}_{K'}\neq
\emptyset\}$;
\item $\mathcal{V}_{K}$ index the set $\{\boldsymbol{x}_{n}\}$ of all the vertices of the element $K$;
\item $|K|$ denote the area/volume of $K$;
\item $h_{K}$ denote the diameter of $K$;
\item $\boldsymbol{n}_{\gamma}^{K}$ denote the unit exterior normal vector to the edge/face $\gamma\in\mathcal{F}_{K}$;
\item $\vsf_{|K}$ denote the restriction of $\vsf$ to the element $K$;
\item $\bar{\vsf}_{K}$ denote the mean value of the function $\vsf$ on $K$, i.e. $\bar{\vsf}_{K}=\tfrac{1}{|K|}(\vsf,1)_{L^{2}(K)}$.
\end{itemize}
For an edge/face $\gamma\in\mathcal{F}$, let:
\begin{itemize}
\item $\mathbb{P}_{n}(\gamma)$ denote the space of polynomials on $\gamma$ of total degree at most $n$;
\item $\Omega_{\gamma}=\{K\in\T:~\gamma\in\mathcal{F}_{K}\}$;
\item $\mathcal{V}_{\gamma}$ index the set $\{\boldsymbol{x}_{n}\}$ of all the vertices of the edge/face $\gamma$;
\item $|\gamma|$ denote the length/area of the edge/face $\gamma$;
\item $h_{\gamma}$ denote the diameter of the edge/face $\gamma$;
\item $\vsf_{|\gamma}$ denote the restriction of $\vsf$ to the edge/face $\gamma$.
\end{itemize}

\RR{For $n\in\mathcal{V}$, we let $\lambda_{n}$ denote the continuous, piecewise linear basis function associated to $\boldsymbol{x}_{n}$, characterized by the conditions $\lambda_{n|K}\in\mathbb{P}_1(K)$ for all $K\in\T$ and $\lambda_n(\boldsymbol{x}_{m})=\delta_{nm}$ for all $m\in\mathcal{V}$, where $\delta_{nm}$ denotes the Kronecker delta}.

For $K\in \T$, we define $\Pi_{K}:L^{2}(K)\rightarrow\mathbb{P}_{1}(K)$
to be the orthogonal projection operator characterized by
\begin{eqnarray}\label{projectionL2K}
\left(f-\Pi_{K}(f),p\right)_{L^{2}(K)}=0\quad\forall p\in\mathbb{P}_{1}(K).
\end{eqnarray}

Finally, in the  manuscript we shall use $C$ to denote any positive constant which is independent of any mesh size and any physical parameter related with the problem.
\section{\EO{A posteriori error analysis for the state equation}}\label{fully_computable}\label{Fully_computable}
In this section\RR{,} we review and extend fully computable a posteriori error estimates for a wide family of low--order stabilized finite element discretizations of \eqref{state_equation}. Before presenting this material, we briefly summarize some results concerning the analysis of problem \eqref{state_equation}.

\subsection{The state equation}
\label{subsec:state_equation}
We consider the following stationary convection--reaction--diffusion problem: Find $\ysf$ such that
\begin{equation}
\label{advection-reaction-diffusion}
-\nu\Delta\ysf + \bsf\cdot\nabla\ysf + \kappa\ysf =  \RR{\qsf} \textrm{  in  } \Omega, \qquad \ysf = 0 \textrm{  on  }\partial\Omega.
\end{equation}
The weak formulation of the previous problems reads: Find $\ysf\in H_{0}^{1}(\Omega)$ such that
\begin{equation}\label{weak_advection-reaction-diffusion}
\mathcal{B}(\ysf,\vsf)=(\RR{\qsf},\vsf)_{L^{2}(\Omega)} \quad \forall \vsf\in H_{0}^{1}(\Omega),
\end{equation}
where, for $\wsf , \vsf \in H_0^1(\Omega)$, the bilinear form $\mathcal{B}$ is defined by
\begin{equation}
\label{forma_B}
\mathcal{B}(\wsf,\vsf):=\nu ( \nabla \wsf,\nabla \vsf)_{\boldsymbol{L}^{2}(G)} + (\bsf \cdot \nabla \wsf + \kappa \wsf,\vsf)_{L^{2}(\Omega)}.
\end{equation}

We assume that the data of problem \eqref{weak_advection-reaction-diffusion} satisfy the following conditions:
\begin{itemize}
\item[\textbf{(A1)}] $\nu$ and $\kappa$ are real and positive constants;
\item[\textbf{(A2)}] $\bsf\in \boldsymbol{W}^{1, \infty}(\Omega)$ and is a solenoidal field, that is, $\mathbf{div}~\bsf=0$;
\item[\textbf{(A3)}] $\RR{\qsf}\in L^{2}(\Omega)$.
\end{itemize}

We define, for $G = \Omega$ or $G\in\T$, and $\vsf \in H ^1(G)$, the norm
\begin{equation}
\label{eq:norm}
\norm{\vsf}_{G}:= \left( \nu\|\nabla \vsf\|_{\boldsymbol{L}^{2}(G)}^{2} + \kappa\|\vsf\|_{L^{2}(G)}^{2} \right)^{1/2} .
\end{equation}
On the basis of \textbf{(A1)}--\textbf{(A3)}, this definition implies that \EO{$\mathcal{B}(\vsf,\vsf) = \norm{\vsf}_{\Omega}^2$} and that $\mathcal{B}(\wsf,\vsf) \leq (1+(\kappa \nu)^{-1/2} \| \bsf\|_{L^{\infty}(\Omega)} ) \norm{\vsf}_{\Omega} \norm{\wsf}_{\Omega} $ for all $\vsf, \wsf \in H_0^1(\Omega)$.  Then, the Lax--Milgram Lemma immediately yields the well--posedness of problem \eqref{weak_advection-reaction-diffusion} \cite{MR2050138,MR2454024}.

To approximate the solution to problem \eqref{weak_advection-reaction-diffusion}, we will consider stabilized finite element methods: Find $\ysf_{\T} \in \mathbb{V}(\T)$ such that
\begin{equation}\label{stabilized_fem}
\mathcal{B}(\ysf_{\T},\vsf_{\T})+\mathcal{S}(\ysf_{\T},\RR{\qsf};\vsf_{\T})=(\RR{\qsf},\vsf_{\T})_{L^{2}(\Omega)} \quad \forall \vsf_{\T}\in \mathbb{V}(\T),
\end{equation}
where $\V(\T)$ denotes the space of continuous piecewise linear functions on $\T$, i.e,
\begin{equation*}
\label{espacio_G}
\mathbb{V}(\T):=
\{
v\in\mathcal{C}^{0}(\overline{\Omega}):~v_{|K}\in\mathbb{P}_{1}(K)~\forall K\in\T~\textrm{and}~v_{|\partial\Omega}=0
\}
\end{equation*}
and $\mathcal{S}$ corresponds to a particular choice of a stabilization term; the election $\mathcal{S} = 0$ corresponds to the standard finite element method without stabilization. We note that $\mathcal{S}$ may contain contributions of the datum $\RR{\qsf}$. In the next subsection, we will be precise about the stabilized terms that are allowed in our analysis. Meanwhile, we will assume that problem \eqref{stabilized_fem} has a unique solution $\ysf_{\T} \in \V(\T)$.

In general, stabilized schemes add mesh-dependent terms to the standard Galerkin formulation of \eqref{advection-reaction-diffusion} with the aim of improving the stability of the numerical method in the regime where the layers are unresolved \cite{MR2454024}. Recently, to improve the accuracy of the schemes, attention has shifted toward the development of a posteriori error estimators, which is the content of the following subsections.
\subsection{Reliability: a fully computable upper bound}\label{subsec:reliability_fully}
In order to construct a fully computable a posteriori error estimator, we follow \cite{MR3123245,allendes2015error}, where the a posteriori error analysis is based on two main ingredients: the construction of equilibrated boundary fluxes and explicit solutions to Neumann-type problems. First, we use \eqref{weak_advection-reaction-diffusion} and integration by parts to arrive at the error equation
\[
\mathcal{B}(\ysf-\ysf_{\T},\vsf)  
= 
\sum_{K\in\T}
\left(
\RR{\qsf}-\bsf\cdot\nabla\ysf_{\T}-\kappa\ysf_{\T}
,\vsf\right)_{L^{2}(K)}
\\
-\sum_{K\in\T}\sum_{\gamma\in\mathcal{F}_K}
\left(
\nu\nabla\ysf_{\T|K}\cdot\boldsymbol{n}_{\gamma}^{K}
,\vsf\right)_{L^{2}(\gamma)}
\]
for all $\vsf\in H_{0}^{1}(\Omega)$. We now introduce boundary fluxes $g_{\gamma,K}\in\mathbb{P}_{1}(\gamma)$ on elements $K\in\T$ and $\gamma\in\mathcal{F}_{K}$, which satisfy
\begin{equation}\label{suma0}
\bullet\mbox{\textbf{Consistency}: }g_{\gamma,K}+g_{\gamma,K'}=0,
\mbox{ if }\gamma\in \mathcal{F}_{K}\cap\mathcal{F}_{K'},~
K,K'\in\T,~K\neq K'.
\end{equation}
We can then incorporate such fluxes into the error equation to see that
\begin{equation}\label{error_equation_2}
 \mathcal{B}(\ysf-\ysf_{\T},\vsf) 
 = 
\sum_{K\in\T}
\left(\left(\mathscr{R}_{K},\vsf\right)_{L^{2}(K)}
+
(\textrm{osc}_{K},\vsf)_{L^{2}(K)}
+
\sum_{\gamma\in\mathcal{F}_{K}}(\mathscr{R}_{\gamma,K},\vsf)_{L^{2}(\gamma)}\right),
\end{equation} 
where, for $K \in \T$ and $\gamma \in \mathcal{F}_{K}$,
\begin{equation}\label{rgk}
\mathscr{R}_{K} := \Pi_{K}(\RR{\qsf})-\Pi_{K}(\bsf\cdot\nabla\ysf_{\T})-\kappa\ysf_{\T|K},\quad
\mathscr{R}_{\gamma,K} := g_{\gamma,K}-\nu\nabla\ysf_{\T|K}\cdot\boldsymbol{n}_{\gamma}^{K},
\end{equation}
\begin{equation}\label{osc}
\textrm{osc}_{K} := \RR{\qsf}-\Pi_{K}(\RR{\qsf})-\bsf\cdot\nabla\ysf_{\T|K}+\Pi_{K}(\bsf\cdot\nabla\ysf_{\T}).
\end{equation}
In addition, for $K \in \T$, we introduce a smooth enough vector function $\boldsymbol{\sigma}_{K}$, such that
\begin{equation}\label{sigmalocal}
\boldsymbol{\sigma}_{K}\cdot\boldsymbol{n}_{\gamma}^{K}  = 
\mathscr{R}_{\gamma,K}\mbox{ on }\gamma, \mbox{ for all }\gamma\in\mathcal{F}_{K},
\end{equation}
and hence, integration by parts yields that
$$
\left(\mathbf{div}~\boldsymbol{\sigma}_{K},\vsf\right)_{L^{2}(K)}
+
\left(\boldsymbol{\sigma}_{K},\nabla\vsf\right)_{\boldsymbol{L}^{2}(K)}
=
\sum_{\gamma\in\mathcal{F}_{K}}(\mathscr{R}_{\gamma,K},\vsf)_{L^{2}(\gamma)}.
$$
Consequently, we can rewrite and bound the error equation as
\begin{align}
\nonumber
\mathcal{B}(\ysf-\ysf_{\T},\vsf) 
& = 
\sum_{K\in\T}\left(
(\mathscr{R}_{K}+\mathbf{div}~\boldsymbol{\sigma}_{K}, \vsf)_{L^{2}(K)}
+
(\boldsymbol{\sigma}_{K},\nabla\vsf)_{\boldsymbol{L}^{2}(K)}
+
(\textrm{osc}_{K},\vsf)_{L^{2}(K)}\right)\\
\label{eq:aux_apost}
& \leq
\left(
\sum_{K\in\T}
\tilde \eta_{K}^{2}
\right)^{1/2}
\norm{\vsf}_{\Omega},
\end{align}
where the error indicator $\tilde \eta_{K}$ is defined by
\begin{equation}\label{indicator}
\tilde \eta_{K}:=\frac{1}{\sqrt{\kappa}}
\|\mathscr{R}_{K}+\mathbf{div}~\boldsymbol{\sigma}_{K}\|_{L^{2}(K)}
+
\frac{1}{\sqrt{\nu}}
\|\boldsymbol{\sigma}_{K}\|_{\boldsymbol{L}^{2}(K)}
+
\mathsf{C}_{\textrm{osc},K}\|\textrm{osc}_{K}\|_{L^{2}(K)}.
\end{equation}
To obtain \eqref{eq:aux_apost}, we have used the Cauchy--Schwarz inequality and  
\[
 | (\textrm{osc}_K, \vsf)_{L^2(K)} | \leq \mathsf{C}_{\mathrm{osc},K} \|\textrm{osc}_{K} \|_{L^2(K)} \norm{\vsf}_{K}, \quad \mathsf{C}_{\textrm{osc},K} = \min\left\{\frac{h_{K}}{\pi\sqrt{\nu}},\frac{1}{\sqrt{\kappa}}\right\}.
\]
The latter holds because, by using \eqref{projectionL2K} and the Poincar\'e inequality \cite{MR2036927,MR0117419}, we have that $|(\textrm{osc}_K, \vsf)_{L^2(K)}| = |(\textrm{osc}_K, \vsf - \bar \vsf_{K})_{L^2(K)} |\leq (h_K / \pi) \| \nabla \vsf \|_{L^2(K)}$ but also $|(\textrm{osc}_K, \vsf)_{L^2(K)}| \leq \kappa^{-1/2}\| \textrm{osc}_K\|_{L^2(K)} \norm{\vsf}_{K}$.

Finally, taking $\vsf=\ysf-\ysf_{\T}$ and using that $\mathcal{B}(\vsf,\vsf)=\norm{\vsf}_{\Omega}^{2}$, we arrive at
\begin{equation}\label{eq:free_upper_bound}
\norm{\ysf-\ysf_{\T}}_{\Omega} \leq \tilde \eta,
\qquad
\tilde \eta:=\left(\sum_{K\in\T}  \tilde{\eta}_{K}^{2} \right)^{1/2}.
\end{equation}
\begin{remark}[fully computable upper bound]
The main advantage of the previous a posteriori error analysis is that it provides an upper bound for the error that is free of any unknown constants. Consequently, the error estimator $\tilde \eta$ can be confidently used as a stopping \EO{criterion} for an adaptive mesh procedure.
\end{remark}

We conclude this subsection by mentioning that the quality of the error estimation depends on the construction of the vector function $\boldsymbol{\sigma}_{K}$ satisfying \eqref{sigmalocal} and the equilibrated boundary fluxes. Before discussing such constructions, we introduce a family of stabilized schemes.
\subsection{Stabilized schemes}\label{stabilizations}
We now describe the stabilized finite element methods that will be the focus of our work. To do this, we write the stabilization term $\mathcal{S}$, based on \cite{allendes2015error}, in terms of local contributions coming from each element, namely,
\begin{eqnarray*}
\mathcal{{S}}(\ysf_{\T},\RR{\qsf};\vsf_{\T}) =
\sum_{K\in \T }\mathcal{{S}}_{K}(\ysf_{\T},\RR{\qsf};\vsf_{\T|K}).
\end{eqnarray*}
The local contributions, $\mathcal{S}_{K}$, for the below mentioned stabilizations, are as follows:
~\\~\\
\underline{\textit{Streamline Upwind Petrov-Galerkin (SUPG)}:} This stabilization technique was introduced by Brooks and Hughes \cite{MR679322}; see also \cite{MR571681,Navert,MR2454024}. The local contributions are
\begin{equation}
\label{eq:SUPG}
\mathcal{{S}}_{K}(\ysf_{\T},\RR{\qsf};\vsf_{\T|K})=\tau_{K}
\left(\bsf\cdot\nabla\ysf_{\T}+\kappa\ysf_{\T}-\RR{\qsf}, \bsf\cdot\nabla \vsf_{\T}
\right)_{L^{2}(K)}.
\end{equation}
\underline{\textit{Galerkin Least Squares (GLS)}}: This method was introduced in \cite{MR1002621}; see also \cite{MR2149365,MR2454024}. The local contributions are
\begin{equation}
 \label{eq:GLS}
\mathcal{{S}}_{K}(\ysf_{\T},\RR{\qsf};\vsf_{\T|K})=\tau_{K}
\left( \bsf\cdot\nabla\ysf_{\T}+\kappa\ysf_{\T}-\RR{\qsf},
\bsf\cdot\nabla \vsf_{\T} + \kappa\vsf_{\T}
\right)_{L^{2}(K)}.
\end{equation}
\underline{\textit{Continuous Interior Penalty (CIP)}}: This stabilization method was proposed by Douglas and Dupont \cite{MR0440955}. Upon defining $\llbracket \bsf\cdot\nabla\ysf_{\T}\rrbracket_{\gamma,K}:=\bsf\cdot\nabla(\ysf_{\T|K}-\ysf_{\T|K^\gamma})$ with $K^\gamma$ being the element that shares $\gamma$ with $K$, the local contributions are
\begin{equation}
\label{eq:CIP}
\mathcal{{S}}_{K}(\ysf_{\T},\RR{\qsf};\vsf_{\T|K}):=\sum_{\gamma\in\mathcal{F}_{K}\cap\mathcal{F}_{I}}\tau_{\gamma}
\left(
\llbracket \bsf\cdot\nabla\ysf_{\T}\rrbracket_{\gamma,K},  \bsf\cdot\nabla\vsf_{\T|K}  \right)_{L^{2}(\gamma)}.
\end{equation}
\underline{\textit{Edge Stabilization (ES)}}: This technique was proposed by Burman and Hansbo in \cite{MR2068903} (see also \cite{MR2192329,MR2454024}). Upon defining $\llbracket \cdot \rrbracket$ to be the usual jump on internal edges/faces and $K^\gamma$ to be the element that shares $\gamma$ with $K$, the local contributions are
\begin{equation}
\label{eq:ES}
\mathcal{{S}}_{K}(\ysf_{\T},\RR{\qsf};\vsf_{\T|K}):=\sum_{\gamma\in\mathcal{F}_{K}\cap\mathcal{F}_{I}}\tau_{\gamma}
\left(
\llbracket \nabla \ysf_{\T} \cdot \boldsymbol{n}_{\gamma} \rrbracket , \nabla \vsf_{\T|K} \cdot\boldsymbol{n}_{\gamma}^{K} (h_{K}^{2}+h_{K^\gamma}^{2})\right)_{L^{2}(\gamma)}.
\end{equation}
In all the previous schemes $\tau_{K}$ and $\tau_{\gamma}$ denote nonnegative stabilization parameters that can vary from one method to another.
\subsection{Construction of the equilibrated fluxes}\label{boundary_fluxes}
\RR{We now describe a procedure for obtaining a set of equilibrated fluxes $\{g_{\gamma,K}: \,\gamma\in\mathcal{F}_K,\,K\in\T\}$ which satisfy \eqref{suma0} and
\begin{itemize}
\item \textbf{First order equilibration}: for all $\lambda\in\mathbb{P}_{1}(K)$ and all $K\in\T$,
\begin{align}\label{eq:first_order}
0&=(\RR{\qsf},\lambda)_{L^{2}(K)}-\mathcal{B}_{K}\left(\ysf_{\T},\lambda\right)+
\sum_{\gamma\in\mathcal{F}_{K}}\left(g_{\gamma,K},\lambda\right)_{L^{2}(\gamma)}
-\mathcal{S}_{K}\left(\ysf_{\T},\RR{\qsf};\lambda\right),
\end{align}
where $\mathcal{B}_{K}\left(\ysf_{\T},\lambda\right)=\nu (\nabla \ysf_{\T} , \nabla \lambda)_{\boldsymbol{L}^2(K)}+(\boldsymbol{b}\cdot\nabla\ysf_{\T}+\kappa\ysf_{\T},\lambda)_{L^2(K)}$.
\end{itemize}
In addition, the equilibrated fluxes satisfy a property, namely \eqref{gmJ}, that will be used to prove that the local error indicators defined by \eqref{indicator} are locally efficient. 

Let
\begin{equation}\label{jumps2}
\langle J\rangle_{\gamma,K}:=
\left\{
\begin{array}{cl}
\tfrac{1}{2}(J_{\gamma,K}-J_{\gamma,K'}) & \textrm{if}~\gamma\in \mathcal{F}_{K}\cap\mathcal{F}_{K'},~K\neq K',\\
J_{\gamma,K} & \textrm{if}~\gamma\in \mathcal{F}_{K}\cap\mathcal{F}_{\partial\Omega},
\end{array}
\right.
\end{equation}
where $J_{\gamma,K}:=\nu \nabla\ysf_{\T|K}\cdot\boldsymbol{n}_{\gamma}^{K}$. For every $i\in\mathcal{V}$, let $\{\xi_{K,i}: \,K\in\Omega_{i}\}$ be a solution to the linear system of equations
\begin{equation}\label{notunique}
\frac{1}{2}\sum_{K'\in \Omega_{K}\cap\Omega_{i}}\left(\xi_{K,i}-\xi_{K',i}\right)+
\sum_{\gamma\in\mathcal{F}_{K}\cap\mathcal{F}_{i}\cap\mathcal{F}_{\partial\Omega}}\xi_{K,i}=
\Delta_{K}\left(\lambda_{i}\right)\quad\forall~K\in\Omega_{i}
\end{equation}
where
\begin{equation*}\label{DeltaK}
\Delta_{K}\left(\lambda_{i}\right)=
\mathcal{B}_{K}(\ysf_{\T},\lambda_{i})-\left(\RR{\qsf},\lambda_{i}\right)_{L^{2}(K)}-
\sum_{\gamma\in\mathcal{F}_{K}}\left(\langle J \rangle_{\gamma,K},\lambda_{i}\right)_{L^{2}(\gamma)}+\mathcal{S}_{K}\left(\ysf_{\T},\RR{\qsf};\lambda_{i|K}\right).
\end{equation*}
In terms of the solutions to these systems of equations, we define
\begin{equation}\label{unique}
\mu_{K,i}^{\gamma}=\left\{
\begin{array}{ll}
\frac{1}{2}\left(\xi_{K,i}-\xi_{K',i}\right)+\left(\langle J \rangle_{\gamma,K},
\lambda_{i}\right)_{L^{2}(\gamma)}
\quad &\textrm{if $\gamma\in\mathcal{F}_{K}\cap\mathcal{F}_{K'}$, $K\ne K'$,}\\
\xi_{K,i}+\left(\langle J \rangle_{\gamma,K},
\lambda_{i}\right)_{L^{2}(\gamma)}
\quad &\textrm{if $\gamma\in\mathcal{F}_{K}\cap\mathcal{F}_{\partial\Omega}$,}\\
\end{array} \right.
\end{equation}
for $i\in\mathcal{V}$, $K\in\Omega_i$ and $\gamma\in\mathcal{F}_i$. Now, the terms $\mathcal{S}_{K}$ are defined in such a way that
\begin{equation*}
\sum_{K\in\Omega_{i}}\Delta_{K}\left(\lambda_{i}\right)=\mathcal{B}(\ysf_{\T},\lambda_{i})-\left(\RR{\qsf},\lambda_{i}\right)_{L^{2}(\Omega)}-
\sum_{K\in\Omega_{i}}\sum_{\gamma\in\mathcal{F}_{K}}\left(\langle J \rangle_{\gamma,K},\lambda_{i}\right)_{L^{2}(\gamma)}+\mathcal{S}\left(\ysf_{\T},\RR{\qsf};\lambda_{i}\right)
\end{equation*}
for $i\in\mathcal{V}$. Hence, we can conclude that
\begin{equation*}
\sum_{K\in\Omega_{i}}\Delta_{K}\left(\lambda_{i}\right)=0\quad \forall\, i\in\mathcal{V}: \boldsymbol{x}_i\notin\partial\Omega
\end{equation*}
upon taking $v_{\T}=\lambda_{i}$ in \eqref{stabilized_fem} and noticing that
\begin{equation*}
\sum_{K\in\Omega_{i}}\sum_{\gamma\in\mathcal{F}_{K}}\left(\langle J \rangle_{\gamma,K},\lambda_{i}\right)_{L^{2}(\gamma)}=0\quad \forall\, i\in\mathcal{V}: \boldsymbol{x}_i\notin\partial\Omega.
\end{equation*}
It then follows from \cite[Lemma 5]{Ainsworthfluxes2007} (and its three dimensional analog which can be proved using the same arguments) that, for $i\in\mathcal{V}$,
\begin{itemize}
\item if $\boldsymbol{x}_i\in\partial\Omega$, then \eqref{notunique} has a unique solution;
\item if $\boldsymbol{x}_i\not\in\partial\Omega$, then solutions to \eqref{notunique} exist and are of the form $\{\xi_{K,i}+c_i,\,K\in\Omega_{i}\}$, where $c_i$ is any constant and $\{\xi_{K,i},\,K\in\Omega_{i}\}$ is any solution to \eqref{notunique}.
\end{itemize}
Consequently, the $\mu_{K,i}^{\gamma}$ are uniquely defined by \eqref{unique}. This is due to the fact that, if $\boldsymbol{x}_i\not\in\partial\Omega$, the solution to \eqref{notunique} only appears in \eqref{unique} as $\xi_{K,i}-\xi_{K',i}$ and so the nonuniqueness cancels out. Hence, for each $i\in\mathcal{V}$, the $\mu_{K,i}^{\gamma}$, for $K\in\Omega_i$ and $\gamma\in\mathcal{F}_i$, can be computed using \eqref{unique} after obtaining a solution to \eqref{notunique}.

For $\gamma\in\mathcal{F}_K$ and $K\in\T$, we define
\begin{equation}\label{fluxdef}
g_{\gamma,K}=
\frac{d}{|\gamma|}\sum_{j\in\mathcal{V}_{\gamma}}\mu_{K,j}^{\gamma}
\left((d+1)\lambda_{j}-1\right)
\end{equation}
which is such that $g_{\gamma,K}\in\mathbb{P}_1(\gamma)$ and
\begin{equation}\label{fluxmomdef}
\left(g_{\gamma,K},\lambda_{i}\right)_{L^{2}(\gamma)}=\mu_{K,i}^{\gamma}\mbox{ for all }i\in\mathcal{V}_\gamma.
\end{equation}

Now, let $\gamma\in \mathcal{F}_{K}\cap\mathcal{F}_{K'},~
K,K'\in\T,~K\neq K'$. From \eqref{unique} and \eqref{jumps2} it can be seen that
\begin{equation*}
\begin{split}
\mu_{K,j}^{\gamma}
&=
\frac{1}{2}\left(\xi_{K,j}-\xi_{K',j}\right)+\left(\langle J \rangle_{\gamma,K},
\lambda_{j}\right)_{L^{2}(\gamma)}
\\
&=
\frac{1}{2}\left(\xi_{K,j}-\xi_{K',j}\right)+\left(\frac{1}{2}(J_{\gamma,K}-J_{\gamma,K'}),
\lambda_{j}\right)_{L^{2}(\gamma)}
\end{split}
\end{equation*}
and, similary,
\begin{equation*}
\begin{split}
\mu_{K',j}^{\gamma}
&=
\frac{1}{2}\left(\xi_{K',j}-\xi_{K,j}\right)+\left(\frac{1}{2}(J_{\gamma,K'}-J_{\gamma,K}),
\lambda_{j}\right)_{L^{2}(\gamma)}.
\end{split}
\end{equation*}
Consequently, by \eqref{fluxdef} we have that
\begin{equation*}
g_{\gamma,K}+g_{\gamma,K'}
=\frac{d}{|\gamma|}\sum_{j\in\mathcal{V}_{\gamma}} \left(\mu_{K,j}^{\gamma}+\mu_{K',j}^{\gamma}\right) \left((d+1)\lambda_{j}-1\right)
=0
\end{equation*}
and hence \eqref{suma0} is satisfied.

Now, let $i\in\mathcal{V}$ and $K\in\Omega_i$. By \eqref{fluxmomdef} and \eqref{unique} we have that
\begin{equation*}
\frac{1}{2}\sum_{K'\in \Omega_{K}\cap\Omega_{i}}\left(\xi_{K,i}-\xi_{K',i}\right)=\sum_{\gamma\in \mathcal{F}_{K}\cap\mathcal{F}_{i}:\gamma\not\in\mathcal{F}_{\partial\Omega}}\left(g_{\gamma,K}-\langle J \rangle_{\gamma,K},
\lambda_{i}\right)_{L^{2}(\gamma)}
\end{equation*}
and
\begin{equation*}
\sum_{\gamma\in\mathcal{F}_{K}\cap\mathcal{F}_{i}\cap\mathcal{F}_{\partial\Omega}}\xi_{K,i}=\sum_{\gamma\in\mathcal{F}_{K}\cap\mathcal{F}_{i}\cap\mathcal{F}_{\partial\Omega}}\left(g_{\gamma,K}-\langle J \rangle_{\gamma,K},
\lambda_{i}\right)_{L^{2}(\gamma)}.
\end{equation*}
Consequently,
\begin{equation*}
\begin{split}
\frac{1}{2}\sum_{K'\in \Omega_{K}\cap\Omega_{i}}\left(\xi_{K,i}-\xi_{K',i}\right)+
\sum_{\gamma\in\mathcal{F}_{K}\cap\mathcal{F}_{i}\cap\mathcal{F}_{\partial\Omega}}\xi_{K,i}
=&\sum_{\gamma\in \mathcal{F}_{K}\cap\mathcal{F}_{i}}\left(g_{\gamma,K}-\langle J \rangle_{\gamma,K},
\lambda_{i}\right)_{L^{2}(\gamma)}
\\
=&\sum_{\gamma\in \mathcal{F}_{K}}\left(g_{\gamma,K}-\langle J \rangle_{\gamma,K},
\lambda_{i}\right)_{L^{2}(\gamma)}
\end{split}
\end{equation*}
since $\lambda_{i|\gamma}=0$ for $\gamma\in\mathcal{F}_{K}\setminus(\mathcal{F}_{K}\cap\mathcal{F}_{i})$. Therefore, from \eqref{notunique} we can see that
\begin{equation*}
0=(\RR{\qsf},\lambda_i)_{L^{2}(K)}-\mathcal{B}_{K}\left(\ysf_{\T},\lambda_i\right)+
\sum_{\gamma\in\mathcal{F}_{K}}\left(g_{\gamma,K},\lambda_i\right)_{L^{2}(\gamma)}
-\mathcal{S}_{K}\left(\ysf_{\T},\RR{\qsf};\lambda_i\right).
\end{equation*}
The fact that \eqref{eq:first_order} is satisfied then follows since $(\cdot,\cdot)_{L^{2}(K)}$, $\mathcal{B}_{K}\left(\cdot,\cdot\right)$, $\left(\cdot,\cdot\right)_{L^{2}(\gamma)}$ and $\mathcal{S}_{K}\left(\cdot,\cdot;\cdot\right)$ are linear in their final arguments and $\{\lambda_{i|K},\,i\in\mathcal{V}_K\}$ is a basis for $\mathbb{P}_1(K)$.

Furthermore, following the arguments in the proof of \cite[Theorem 6.2]{AObook}, yields that
\[
\|g_{\gamma,K}-\langle J \rangle_{\gamma,K}\|_{L^{2}(\gamma)} \leq C \tfrac{1}{\sqrt{|\gamma|}}\sum_{n\in\mathcal{V}_{\gamma}}\sum_{K'\in\Omega_{n}} |\Delta_{K'}(\lambda_n)|
\]
and that
\[
|\Delta_{K'}(\lambda_n)|\leq 
\sqrt{|K'|}\|\mathscr{R}_{K'}\|_{L^{2}(K')}+
\sum_{\gamma'\in\mathcal{F}_{K'}\cap\mathcal{F}_{n}}
\sqrt{|\gamma'|}\|\llbracket J_{\gamma'}\rrbracket\|_{L^{2}(\gamma')} +
|\mathcal{S}_{K'}(\ysf_{\T},\RR{\qsf};\lambda_{n|K'})|.
\]
Consequently, for $K \in\T$ and $\gamma\in\mathcal{F}_{K}$, we have that
\begin{equation}\label{gmJ}
\begin{split}
\left(\tfrac{h_{K}}{\nu}\right)^{1/2} \|g_{\gamma,K}-\langle J\rangle_{\gamma,K}\|_{L^{2}(\gamma)}
\leq C
\sum_{n\in\mathcal{V}_{\gamma}}\sum_{K'\in\Omega_{n}}
\Bigg(
\tfrac{h_{K'}}{\sqrt{\nu}}\|\mathscr{R}_{K'}\|_{L^{2}(K')}
\\
+\sum_{\gamma'\in\mathcal{F}_{K'}\cap\mathcal{F}_{n}}
\left(\tfrac{h_{K'}}{\nu}\right)^{1/2}\|\llbracket J_{\gamma'}\rrbracket\|_{L^{2}(\gamma')} +
\left(\tfrac{h_{K'}^{2-d}}{\nu}\right)^{1/2}
|\mathcal{S}_{K'}(\ysf_{\T},\RR{\qsf};\lambda_{n|K'})|
\Bigg).
\end{split}
\end{equation}}
\subsection{Construction of $\boldsymbol{\sigma}_{K}$}
In order to obtain a fully computable error estimator, we choose $\boldsymbol{\sigma}_{K}\in\mathbb{P}_2(K)\times\mathbb{P}_2(K)$ to be a solution to
\begin{equation}\label{sigmadiv}
\left\{
\begin{array}{rcl}
-\mathbf{div}~\boldsymbol{\sigma}_{K} & = & p_{K} \quad \mbox{in }K,\\
\boldsymbol{\sigma}_{K}\cdot\boldsymbol{n}_{\gamma}^{K} & = & p_{\gamma,K} \quad \mbox{on }\gamma\mbox{ for all }\gamma\in\mathcal{F}_{K},
\end{array}
\right.
\end{equation}
where
\begin{equation}\label{pgk}
p_K = \displaystyle \mathscr{R}_{K}- \overline{\mathscr{R}}_{K}
-\frac{1}{|K|}\sum_{\gamma\in\mathcal{F}_{K}}\left(\mathscr{R}_{\gamma,K},1\right)_{L^{2}(\gamma)}\in\mathbb{P}_1(K), \quad
p_{\gamma,K} = \mathscr{R}_{\gamma,K}\in\mathbb{P}_1(\gamma)
\end{equation}
\RR{where $\mathscr{R}_{K}$ and $\mathscr{R}_{\gamma,K}$ are defined in \eqref{rgk}.} Since the data of problem \eqref{sigmadiv} satisfies a constant equilibration condition, that is,
\[
\begin{split}
\sum_{\gamma\in\mathcal{F}_{K}}\left(p_{\gamma,K},1\right)_{L^{2}(\gamma)}
+\left(p_{K},1\right)_{L^{2}(K)}=0,
\end{split}
\]
then \RR{\cite[Theorem 6.3]{allendes2015error}} provides \RR{an explicit formula for} a solution to \eqref{sigmadiv} that satisfies
\begin{equation}\label{cota_sigmainitial}
\|\boldsymbol{\sigma}_{K}\|_{\boldsymbol{L}^{2}(K)}
\leq C
\left(h_{K}^{1/2}\sum_{\gamma\in\mathcal{F}_{K}}
\|p_{\gamma,K}\|_{L^{2}(\gamma)}
+
h_{K}\left\|p_{K}\right\|_{L^{2}(K)}
\right).
\end{equation}
We note that \eqref{cota_sigmainitial} will also be satisfied by the $\boldsymbol{\sigma}_{K}\in\mathbb{P}_2(K)\times\mathbb{P}_2(K)$ which satisfies \eqref{sigmadiv} and is such that $\|\boldsymbol{\sigma}_{K}\|_{\boldsymbol{L}^{2}(K)}$ is minimized. \RR{Furthermore, once the discrete solution is obtained, the only problems that have to be solved in order to compute $\boldsymbol{\sigma}_{K}$ and $g_{\gamma,K}$ are local problems of size at most $\max \{ \#\Omega_i: i\in\mathcal{V}\}$. Hence, once the number of degrees of freedom is sufficiently large, the cost of obtaining the estimator should be inexpensive when compared with the cost of obtaining the solution to \ref{stabilized_fem}.}
\subsection{Final fully computable upper bound}
Gathering all our findings of the previous sections, allows us to state the following reliability result.
\begin{theorem}[fully computable upper bounds]\label{TH:upper_bound_state_0}
Let $\ysf \in H^1_0(\Omega)$ and $\ysf_{\T} \in \V(\T)$ be the solutions to problems \eqref{weak_advection-reaction-diffusion} and \eqref{stabilized_fem}, respectively. Then, we have the following fully computable upper bound for the energy norm of the error:
\begin{align}\label{eq:final_upper_bound_0}
\norm{\ysf-\ysf_{\T}}_{\Omega} \leq \eta=\left(\sum_{K\in\T}\eta_{K}^{2}\right)^{1/2},
\end{align} 
where the error indicators $\eta_K$ are defined by
\begin{equation}\label{final_reliability}
\eta_{K}:=\frac{1}{\sqrt{\kappa|K|}}
\left|\mathcal{S}_{K}\left(\ysf_{\T},\RR{\qsf};1\right)\right|
+
\frac{1}{\sqrt{\nu}}
\|\boldsymbol{\sigma}_{K}\|_{\boldsymbol{L}^{2}(K)}
+
\mathsf{C}_{\mathrm{osc},K}\|\mathrm{osc}_{K}\|_{L^{2}(K)}.
\end{equation}
\end{theorem}
\begin{proof}
Upon noticing that 
$$
\sum_{\gamma\in\mathcal{F}_{K}}\left(\nabla\ysf_{\T|K}\cdot\boldsymbol{n}_{\gamma}^{K},1\right)_{L^{2}(\gamma)}
=\left(\mathbf{div}~(\nabla\ysf_{\T}),1\right)_{L^{2}(K)}=0
$$
and $\mathscr{R}_{\gamma,K} = g_{\gamma,K}-\nu\nabla\ysf_{\T|K}\cdot\boldsymbol{n}_{\gamma}^{K}$, it follows that 
\begin{equation}\label{eq:aux_1}
\frac{1}{|K|}\sum_{\gamma\in\mathcal{F}_{K}}\left(\mathscr{R}_{\gamma,K},1\right)_{L^{2}(\gamma)}=\frac{1}{|K|}\sum_{\gamma\in\mathcal{F}_{K}}\left(g_{\gamma,K},1\right)_{L^{2}(\gamma)}.
\end{equation}
Moreover, \eqref{projectionL2K}, the equilibration condition \eqref{eq:first_order} and the definition of $\mathscr{R}_{K}$ yield that
\[
\begin{split}
\sum_{\gamma\in\mathcal{F}_{K}}\left(g_{\gamma,K},1\right)_{L^{2}(\gamma)}
=&\mathcal{B}_{K}\left(\ysf_{\T},1\right)-(\RR{\qsf},1)_{L^{2}(K)}+\mathcal{S}_{K}\left(\ysf_{\T},\RR{\qsf};1\right)
\\
=&-(\mathscr{R}_{K},1)_{L^{2}(K)}+\mathcal{S}_{K}\left(\ysf_{\T},\RR{\qsf};1\right).
\end{split}
\]
Combining this with \RR{\eqref{eq:aux_1} implies} that
\begin{equation}
\label{eq:aux}
\frac{1}{|K|}\sum_{\gamma\in\mathcal{F}_{K}}\left(\mathscr{R}_{\gamma,K},1\right)_{L^{2}(\gamma)}=-\overline{\mathscr{R}}_{K}+\frac{1}{|K|}\mathcal{S}_{K}\left(\ysf_{\T},\RR{\qsf};1\right).
\end{equation}
Consequently, we rewrite the norm of $\mathscr{R}_{K}+\mathbf{div}~\boldsymbol{\sigma}_{K}$, as
\begin{equation}\label{RKdiv}
\begin{split}
\|\mathscr{R}_{K}+\mathbf{div}~\boldsymbol{\sigma}_{K}\|_{L^{2}(K)}
=&\left\|\overline{\mathscr{R}}_{K}+\frac{1}{|K|}\sum_{\gamma\in\mathcal{F}_{K}}\left(\mathscr{R}_{\gamma,K},1\right)_{L^{2}(\gamma)}\right\|_{L^{2}(K)}
\\
=&\left\|\frac{1}{|K|}\mathcal{S}_{K}\left(\ysf_{\T},\RR{\qsf};1\right)\right\|_{L^{2}(K)}
=\frac{1}{|K|^{1/2}}\left|\mathcal{S}_{K}\left(\ysf_{\T},\RR{\qsf};1\right)\right|,
\end{split}
\end{equation}
which combined with \eqref{indicator}, \eqref{eq:free_upper_bound}, and \eqref{final_reliability}, yields the result claimed. 
\end{proof}

Applying the above theorem to each of the stabilized methods that we are considering allows us to arrive at the below corollary.

\begin{corollary}\label{CO:upper_bound_state}
Let $\ysf \in H^1_0(\Omega)$ and $\ysf_{\T} \in \V(\T)$ be the solutions to problems \eqref{weak_advection-reaction-diffusion} and \eqref{stabilized_fem}, respectively. Then, we have the following fully computable upper bound for the energy norm of the error:
\begin{equation*}
\label{eq:final_upper_bound_1}
 \norm{\ysf-\ysf_{\T}}_{\Omega} \leq 
\left(\sum_{K\in\T}
\left(
\tfrac{\|\boldsymbol{\sigma}_{K}\|_{\boldsymbol{L}^{2}(K)}}{\sqrt{\nu}}
+
\mathsf{C}_{\mathrm{osc},K}\|\mathrm{osc}_{K}\|_{L^{2}(K)} 
\right)^{2}\right)^{1/2}
\end{equation*}
for SUPG, ES and CIP, and
\begin{equation*}
\label{eq:final_upper_bound_2}
\norm{\ysf-\ysf_{\T}}_{\Omega} \leq 
\left(\sum_{K\in\T}
\left(
\tau_{K}\sqrt{\kappa}
\left\|\overline{\mathscr{R}}_{K}\right\|_{L^{2}(K)}
+
\tfrac{\|\boldsymbol{\sigma}_{K}\|_{\boldsymbol{L}^{2}(K)}}{\sqrt{\nu}}
+
\mathsf{C}_{\mathrm{osc},K}\|\mathrm{osc}_{K}\|_{L^{2}(K)}  
\right)^{2}\right)^{1/2}
\end{equation*}
for GLS.
\end{corollary}
\begin{proof}
The results follow from Theorem \ref{TH:upper_bound_state_0} and an inspection of the stabilization terms. The latter reveals that $\mathcal{S}_{K}\left(\ysf_{\T},\RR{\qsf};1\right)= 0 $ for SUPG, CIP and ES, and that $\left|\mathcal{S}_{K}\left(\ysf_{\T},\RR{\qsf};1\right)\right| = \tau_{K}\left|(-\overline{\mathscr{R}}_{K}, \kappa )_{L^2(K)}\right| = \tau_{K}\kappa\sqrt{|K|}\left\|\overline{\mathscr{R}}_{K}\right\|_{L^{2}(K)}$ for GLS.
\end{proof}
\begin{remark}[extension of the theory]
We are not assuming that $\mathcal{S}_{K}\left(\ysf_{\T},\RR{\qsf};1\right)=0$ as in Assumption 4.3 from \cite{allendes2015error}. This advantage allows us to consider, for example, the GLS scheme. From \eqref{sigmadiv} and \eqref{eq:aux}, it follows that $-\mathbf{div}~\boldsymbol{\sigma}_{K}=\mathscr{R}_{K}$ if $\mathcal{S}_{K}\left(\ysf_{\T},\RR{\qsf};1\right)=0$. If $\mathcal{S}_{K}\left(\ysf_{\T},\RR{\qsf};1\right)\ne 0$ then $-\mathbf{div}~\boldsymbol{\sigma}_{K}\ne\mathscr{R}_{K}$ and the error indicator $\eta_K$, in which the estimator $\eta$ is defined in terms of, has an additional term.
\end{remark}
\subsection{Local efficiency}\label{efficiency}
We now explore the local efficiency properties of the local indicator \eqref{final_reliability}. From \eqref{cota_sigmainitial}\RR{,} we have that
\begin{equation}\label{cota_sigma}
\|\boldsymbol{\sigma}_{K}\|_{\boldsymbol{L}^{2}(K)}
\leq C
\left(h_{K}^{1/2}\sum_{\gamma\in\mathcal{F}_{K}}
\|\mathscr{R}_{\gamma,K}\|_{L^{2}(\gamma)}
+h_{K}\left\|\mathscr{R}_{K}\right\|_{L^{2}(K)}
\right)
\end{equation}
since $p_{\gamma,K}=\mathscr{R}_{\gamma,K}$ and
\[
\begin{split}
h_{K}\left\|p_{K}\right\|_{L^{2}(K)}
&\leq \, h_{K}\left\|\mathscr{R}_{K}-\overline{\mathscr{R}}_{K}\right\|_{L^{2}(K)}+h_{K}{|K|^{1/2}}\sum_{\gamma\in\mathcal{F}_{K}}\tfrac{|\gamma|^{1/2}}{|K|}\left\|\mathscr{R}_{\gamma,K}\right\|_{L^{2}(\gamma)}
\\
&\leq \, h_{K}\left\|\mathscr{R}_{K}\right\|_{L^{2}(K)}+Ch_{K}^{1/2}\sum_{\gamma\in\mathcal{F}_{K}}\left\|\mathscr{R}_{\gamma,K}\right\|_{L^{2}(\gamma)}
\end{split}
\]
because $\left\|\mathscr{R}_{K}-\overline{\mathscr{R}}_{K}\right\|_{L^{2}(K)}\le\left\|\mathscr{R}_{K}\right\|_{L^{2}(K)}$.
Now, we define \begin{eqnarray*}
\llbracket J_{\gamma}\rrbracket:=
\left\{
\begin{array}{cl}
\tfrac{1}{2}(J_{\gamma,K}+J_{\gamma,K'}) & \textrm{if}~\gamma\in \mathcal{F}_{K}\cap\mathcal{F}_{K'},~K\neq K',\\
0 & \textrm{if}~\gamma\in \mathcal{F}_{K}\cap\mathcal{F}_{\partial\Omega},
\end{array}
\right.
\end{eqnarray*}
where $J_{\gamma,K}:=\nu\nabla\ysf_{\T|K}\cdot\boldsymbol{n}_{\gamma}^{K}$. This, combined with \eqref{rgk} and \eqref{jumps2} allows us to state that $\mathscr{R}_{\gamma,K}=g_{\gamma,K}-\langle J\rangle_{\gamma,K}-\llbracket J_{\gamma}\rrbracket$. Then, in view of \eqref{final_reliability}, \eqref{cota_sigma} provides the bound:
\begin{eqnarray}
\eta_{K}&\leq C
\left(
\tfrac{h_{K}}{\sqrt{\nu}}\|\mathscr{R}_{K}\|_{L^{2}(K)}
+
\sum_{\gamma\in\mathcal{F}_{K}}
\left(\tfrac{h_{K}}{\nu}\right)^{1/2}
\left(
\|\llbracket J_{\gamma}\rrbracket\|_{L^{2}(\gamma)}+
\|g_{\gamma,K}-\langle J \rangle_{\gamma,K}\|_{L^{2}(\gamma)}
\right)\right)\nonumber
\\
& \qquad\quad +\mathsf{C}_{\textrm{osc},K}\|\textrm{osc}_{K}\|_{L^{2}(K)}
+\tfrac{1}{\sqrt{\kappa |K|}}|\mathcal{S}_{K}(\ysf_{\T},\RR{\qsf};1)|.\label{eq:bound_for_etaK}
\end{eqnarray}
To prove local efficiency, the terms on the right hand side of \eqref{eq:bound_for_etaK} have to be bounded by the energy norm of the error $\ysf - \ysf_{\T}$ plus data oscillation terms. To do this, we first note that we can rewrite the error equation \eqref{error_equation_2}, for any $\vsf \in H_{0}^{1}(\Omega)$, as
$$
\sum_{K\in\T}
(\mathscr{R}_{K},\vsf)_{L^{2}(K)}
- 2
\sum_{\gamma\in\mathcal{F}_{I}}
(\llbracket J_{\gamma}\rrbracket,\vsf)_{L^{2}(\gamma)}
=
\mathcal{B}(\ysf-\ysf_{\T},\vsf)-\sum_{K\in\T}(\textrm{osc}_{K},\vsf)_{L^{2}(K)}.
$$
Applying standard bubble function arguments \cite{AObook,verfurth1996review} to this error equation yields that
\begin{equation}\label{eq:RK}
\tfrac{h_{K}}{\sqrt{\nu}}\|\mathscr{R}_{K}\|_{L^{2}(K)}
\leq C
\left(
\mathsf{C}_{K}\norm{\ysf-\ysf_{\T}}_{K}+\tfrac{h_{K}}{\sqrt{\nu}}\|\textrm{osc}_{K}\|_{L^{2}(K)}
\right)
\end{equation}
for $K \in\T$, and
\begin{equation}\label{eq:gammaK}
\left(\tfrac{h_{K}}{\nu}\right)^{1/2}\|\llbracket J_{\gamma}\rrbracket\|_{L^{2}(\gamma)}
\leq C
\sum_{K'\in\Omega_{\gamma}}
\left(
\mathsf{C}_{K'}\norm{\ysf-\ysf_{\T}}_{K'}+\tfrac{h_{K'}}{\sqrt{\nu}}\|\textrm{osc}_{K'}\|_{L^{2}(K')}
\right)
\end{equation}
for $K \in\T$ and $\gamma\in\mathcal{F}_{K}$, where
\begin{equation}
\label{eq:C_K}
\mathsf{C}_{K}:=
\max\left\{
1,\tfrac{\|\bsf\|_{\boldsymbol{L}^{\infty}(K)}h_{K}}{\nu},\sqrt{\kappa} \tfrac{h_{K}}{\sqrt{\nu}}
\right\}\RR{.}
\end{equation}

\RR{Now}, from \eqref{eq:bound_for_etaK}, \eqref{eq:RK}, \eqref{eq:gammaK} and \eqref{gmJ} we have that
\begin{align*}
\eta_{K} 
& \leq 	C
\sum_{n\in\mathcal{V}_{K}}
\sum_{K'\in\Omega_{n}}
\Bigg(
\mathsf{C}_{K'}\norm{\ysf-\ysf_{\T}}_{K'}+\tfrac{h_{K'}}{\sqrt{\nu}}\|\textrm{osc}_{K'}\|_{L^{2}(K')} \\
&
+\left(\tfrac{h_{K'}^{2-d}}{\nu}\right)^{1/2}|\mathcal{S}_{K'}(\ysf_{\T},\RR{\qsf};\lambda_{n|K'})|
\Bigg)
+\tfrac{1}{\sqrt{\kappa |K|}}|\mathcal{S}_{K}(\ysf_{\T},\RR{\qsf};1)|.
\end{align*}
Consequently, to obtain the local efficiency of $\eta_{K}$ it remains to control the stabilization term $\mathcal{S}_K$ in the previous inequality. We proceed to examine the local contribution of each method described in Section \ref{stabilizations}, namely, SUPG, GLS, ES, and CIP.
~\\~\\
\underline{\textit{Streamline Upwind Petrov--Galerkin (SUPG):}}  Clearly $\mathcal{{S}}_{K}(\ysf_{\T},\RR{\qsf};1)=0$. Moreover,
\begin{align*}
\left(\tfrac{h_{K}^{2-d}}{\nu}\right)^{1/2}\left|
\mathcal{{S}}_{K}(\ysf_{\T},\RR{\qsf};\lambda_{n|K})\right|
& = \tfrac{h_{K}^{1-d/2}}{\sqrt{\nu}}\tau_{K}
|\left(
\bsf\cdot\nabla\ysf_{\T}+\kappa\ysf_{\T}-\RR{\qsf},
\bsf\cdot\nabla \lambda_{n}
\right)_{L^{2}(K)}|\\
& \leq C
\tau_{K}\tfrac{\|\bsf\|_{\boldsymbol{L}^{\infty}(K)}}{h_{K}}
\tfrac{h_{K}}{\sqrt{\nu}}\left(\|\mathscr{R}_{K}\|_{L^{2}(K)}+\|\textrm{osc}_{K}\|_{L^{2}(K)}\right).
\end{align*}
~\\
\underline{\textit{Galerkin Least Square (GLS):}} In view of \eqref{eq:GLS}, we have that
\[
\tfrac{1}{\sqrt{\kappa |K|}}|\mathcal{S}_{K}(\ysf_{\T},\RR{\qsf},1)| 
 =
\tfrac{1}{\sqrt{\kappa |K|}}\tau_{K}\kappa\sqrt{|K|}\left\|\overline{\mathscr{R}}_{K}\right\|_{L^{2}(K)}
=
\tau_{K}\tfrac{\sqrt{\kappa\nu}}{h_{K}}
\tfrac{h_{K}}{\sqrt{\nu}}\|\mathscr{R}_{K}\|_{L^{2}(K)}
\]
and
\begin{gather*}
 \left(\tfrac{h_{K}^{2-d}}{\nu}\right)^{1/2}\left|
\mathcal{{S}}_{K}(\ysf_{\T},\RR{\qsf};\lambda_{n|K})\right|
 = \tfrac{h_{K}^{1-d/2}}{\sqrt{\nu}}\tau_{K}
|\left(
\bsf\cdot\nabla\ysf_{\T}+\kappa\ysf_{\T}-\RR{\qsf},
\bsf\cdot\nabla \lambda_{n} + \kappa\lambda_{n}
\right)_{L^{2}(K)}|\\
 \leq C
\tau_{K}\max\left\{\tfrac{\|\bsf\|_{\boldsymbol{L}^{\infty}(K)}}{h_{K}},\kappa\right\}
\tfrac{h_{K}}{\sqrt{\nu}}\left(\|\mathscr{R}_{K}\|_{L^{2}(K)}+\|\textrm{osc}_{K}\|_{L^{2}(K)}\right).
\end{gather*}
~\\
\underline{\textit{Edge Stabilization (ES):}} Clearly $\mathcal{{S}}_{K}(\ysf_{\T},\RR{\qsf};1)=0$. Moreover, \eqref{eq:ES} yields
\begin{equation*}
\left(\tfrac{h_{K}^{2-d}}{\nu}\right)^{1/2}\left|\mathcal{{S}}_{K}(\ysf_{\T},\RR{\qsf};\lambda_{n|K})\right|
 \leq C
\tfrac{h_{K}}{\nu}\sum_{\gamma\in\mathcal{F}_{K}\cap\mathcal{F}_{I}}\tau_{\gamma}
\left(\tfrac{h_{K}}{\nu}\right)^{1/2}
\|\llbracket J_{\gamma}\rrbracket\|_{L^{2}(\gamma)}.
\end{equation*}
~\\
\underline{\textit{Continuous Interior Penalty (CIP):}} Clearly $\mathcal{{S}}_{K}(\ysf_{\T},\RR{\qsf};1)=0$. Now, notice that $\bsf=(\boldsymbol{n}_{\gamma}^K\cdot\bsf)\boldsymbol{n}_{\gamma}^K
+(\boldsymbol{t}_{\gamma}\cdot\bsf)\boldsymbol{t}_{\gamma}$
where $\boldsymbol{t}_{\gamma}$ is a unit tangent to $\gamma$. Hence, using the fact that $\bsf\in\boldsymbol{W}_{1}^{\infty}(\Omega)$, allows us to see that $\llbracket \bsf\cdot\nabla\ysf_{\T}\rrbracket_{\gamma,K}=\tfrac{2}{\nu}(\boldsymbol{n}_{\gamma}^K\cdot\bsf)\llbracket J_{\gamma}\rrbracket$ since
\[
\llbracket \bsf\cdot\nabla\ysf_{\T}\rrbracket_{\gamma,K}=
(\boldsymbol{n}_{\gamma}^K\cdot\bsf) \boldsymbol{n}_{\gamma}^K\cdot\nabla(\ysf_{\T|K}-\ysf_{\T|K^\gamma})+(\boldsymbol{t}_{\gamma}\cdot\bsf) \boldsymbol{t}_{\gamma}\cdot\nabla(\ysf_{\T|K}-\ysf_{\T|K^\gamma}).
\]
Thus
$
\|\llbracket \bsf\cdot\nabla\ysf_{\T}\rrbracket_{\gamma,K}\|_{L^{2}(\gamma)}\leq \tfrac{2}{\nu}\|\bsf\|_{\boldsymbol{L}^{\infty}(\gamma)}\|\llbracket J_{\gamma}\rrbracket\|_{L^{2}(\gamma)}.
$
Consequently, \eqref{eq:CIP} yields
\begin{equation*}
\left(\tfrac{h_{K}^{2-d}}{\nu}\right)^{1/2}\left|\mathcal{{S}}_{K}(\ysf_{\T},\RR{\qsf};\lambda_{n|K})\right|
\\
\leq C
\sum_{\gamma\in\mathcal{F}_{K}\cap\mathcal{F}_{I}}
\tfrac{\tau_{\gamma}\|\bsf\|_{\boldsymbol{L}^{\infty}(\gamma)}^2}{\nu h_{K}}\left(\tfrac{h_{K}}{\nu}\right)^{1/2}
\|\llbracket J_{\gamma}\rrbracket\|_{L^{2}(\gamma)}.
\end{equation*}
By combining all of the previous results, we can finally summarize and bound the error indicator $\eta_K$ as described in the following theorem. 

\begin{theorem}
Let $\ysf \in H^1_0(\Omega)$ and $\ysf_{\T} \in \V(\T)$ be the solutions to problems \eqref{weak_advection-reaction-diffusion} and \eqref{stabilized_fem}, respectively. Then, we have the following local lower bound for the energy norm of the error:
\begin{align*}
\eta_{K} 
& \leq 	C
\sum_{n\in\mathcal{V}_{K}}
\sum_{K'\in\Omega_{n}}
\Bigg(
\mathsf{C}_{K'}\norm{\ysf-\ysf_{\T}}_{K'}+
\tfrac{h_{K'}}{\sqrt{\nu}}\|\mathrm{osc}_{K'}\|_{L^{2}(K')}\\
& \qquad\qquad\quad +\mathsf{C}_{\mathcal{S}_{K'}}
\Bigg(
\sum_{K''\in\Omega_{K'}}
\left(
\mathsf{C}_{K''}\norm{\ysf-\ysf_{\T}}_{K''}+
\tfrac{h_{K''}}{\sqrt{\nu}}\|\mathrm{osc}_{K''}\|_{L^{2}(K'')}
\right)
\Bigg)
\Bigg),
\end{align*}
where \EO{$\mathsf{C}_{K}$ is defined as in \eqref{eq:C_K} and}
\[
\mathsf{C}_{\mathcal{S}_{K}}=
\left\{
\begin{array}{l}
\displaystyle{\tau_{K}\frac{\|\bsf\|_{\boldsymbol{L}^{\infty}(K)}}{h_{K}}}\quad  \textrm{for SUPG}, \qquad
\displaystyle{\tau_{K}\max\left\{\frac{\sqrt{\kappa\nu}}{h_{K}},\frac{\|\bsf\|_{\boldsymbol{L}^{\infty}(K)}}{h_{K}},\kappa\right\}}\quad  \textrm{for GLS},\\
\displaystyle{\frac{h_{K}}{\nu}
\max_{\gamma\in\mathcal{F}_{K}\cap\mathcal{F}_{I}}\tau_{\gamma}}\quad  \textrm{for ES}, \qquad
\displaystyle{\max_{\gamma\in\mathcal{F}_{K}\cap\mathcal{F}_{I}}\frac{\tau_{\gamma}\|\bsf\|_{\boldsymbol{L}^{\infty}(\gamma)}^2}{\nu h_{K}}}\quad  \textrm{for CIP}.
\end{array}
\right.
\]
\end{theorem}
\section{Optimal control problem}\label{control_optimo}
In this section\RR{,} we analyze the optimal control problem \eqref{min}--\eqref{control_constraint}. To approximate its solution, we propose a numerical method that is based on the stabilized schemes of Section \ref{stabilizations}. We derive fully computable a posteriori upper bounds for the error and prove local efficiency properties of the proposed error estimators. 

Under the assumptions \textbf{(A1)}--\textbf{(A3)}, the existence and uniqueness of an optimal pair $(\bar \ysf, \bar \usf) \in H^1_0(\Omega) \times \mathbb{U}_{ad}$ satisfying \eqref{min}--\eqref{control_constraint} follows standard arguments \cite{MR2583281}. An equivalent formulation can be obtained by introducing the so--called adjoint variable $\bar{\psf}$. We then say that $(\bar{\ysf},\bar{\psf},\bar{\usf}) $ is optimal if and only if it solves the nonlinear system 
\begin{equation}\label{optimal_system_1}
\left\{
\begin{array}{rcl}
\bar{\ysf}\in H_{0}^{1}(\Omega): & 
\mathcal{B}(\bar{\ysf},\vsf) = (\fsf+\bar{\usf},\vsf)_{L^{2}(\Omega)}, &
\forall~\vsf\in H_{0}^{1}(\Omega),\\
\bar{\psf}\in H_{0}^{1}(\Omega): & 
\mathcal{B}^{*}(\bar{\psf},\wsf) = (\bar{\ysf}-\ysf_{\Omega},\wsf)_{L^{2}(\Omega)}, &
\forall~\wsf\in H_{0}^{1}(\Omega),\\
\bar{\usf}\in {\mathbb{U}_{ad}}: & 
\left(\bar{\psf}+\vartheta\bar{\usf},\usf-\bar{\usf}\right)_{L^{2}(\Omega)}\geq 0
, &  
\forall~\usf\in\mathbb{U}_{ad},
\end{array}
\right.
\end{equation}
which  is necessary and sufficient for optimality. The form $\mathcal{B}$ is given in \eqref{forma_B}, and
\begin{equation}
\label{eq:B_dual}
\mathcal{B}^{*}(\wsf,\vsf):=\nu(\nabla\wsf,\nabla\vsf)_{L^{2}(\Omega)}+
(\kappa\wsf-\boldsymbol{b}\cdot\nabla\wsf,\vsf)_{L^{2}(\Omega)}.
\end{equation}
Finally, we recall the projection formula for the optimal control variable: the variational inequality in \eqref{optimal_system_1} can be equivalently written as (see \cite[Chapter 2]{MR2583281}),\RR{
\begin{equation}\label{eq:u}
\bar{\usf}=\Pi\left(-\frac{1}{\vartheta}\bar{\psf}\right),
\end{equation}}
where\RR{
\begin{equation}\label{projection_formula}
\left(\Pi\left(w\right)\right)(\boldsymbol{x}):=\min\left\{\bbsf,\max\left\{\asf,w(\boldsymbol{x})\right\}\right\}\mbox{ for almost every }\boldsymbol{x}\mbox{ in }\Omega
\end{equation}}
with $\asf$ and $\bbsf$ being the control bounds that define the set $\mathbb{U}_{ad}$ in \eqref{control_constraint}. \RR{We note that this operator is such that, for $K\in\T$,
\begin{equation}\label{Lipschitz}
\left\| \Pi (w_1)-\Pi (w_2)  \right\|_{L^2(K)}\le \left\| w_1-w_2 \right\|_{L^2(K)}
\mbox{ for all }w_1,w_2\in H_{0}^{1}(\Omega).
\end{equation}
}

We now introduce a numerical technique to solve problem \eqref{min}--\eqref{control_constraint} that is based on the discretization of the optimality system \eqref{optimal_system_1}, i.e., we consider the so--called \emph{optimize--then--discretize} approach. The scheme incorporates stabilized terms into the standard Galerkin discretizations of the state and adjoint equations; no a priori relation between the stabilized terms is required. The stabilized scheme reads as follows: Find 
$(\bar{\ysf}_{\T},\bar{\psf}_{\T},\bar{\usf}_{\T})\in \mathbb{V}(\T)\times \mathbb{V}(\T)\times \mathbb{U}_{ad}(\T)$ such that
\begin{equation}\label{optimal_system_discrete_1}
\begin{array}{cl}
\mathcal{B}(\bar{\ysf}_{\T},\vsf_{\T}) + 
\mathcal{S}(\bar{\ysf}_{\T},\fsf+\bar{\usf}_{\T};\vsf_{\T})
= (\fsf+\bar{\usf}_{\T},\vsf_{\T})_{L^{2}(\Omega)}, 
& \forall\vsf_{\T}\in \mathbb{V}(\T),\\
\hspace{-0.3cm}
\mathcal{B}^{*}(\bar{\psf}_{\T},\wsf_{\T}) + \mathcal{S}^{*}(\bar{\psf}_{\T},\bar{\ysf}_{\T}-\ysf_{\Omega};\wsf_{\T})
= (\bar{\ysf}_{\T}-\ysf_{\Omega},\wsf_{\T})_{L^{2}(\Omega)}, 
&  \forall\wsf_{\T}\in \mathbb{V}(\T),\\
\left(\bar{\psf}_{\T}+\vartheta\bar{\usf}_{\T},\usf_{\T}-\bar{\usf}_{\T}\right)_{L^{2}(\Omega)}\geq 0,
& \forall\usf_{\T}\in\mathbb{U}_{ad}(\T),
\end{array}
\end{equation}
where $\mathbb{U}_{ad}(\T) = \{ \usf_{\T} \in L^{\infty}(\Omega):~\usf_{\T|K}\in\mathbb{P}_{0}(K)~\forall K\in\T\}\cap\mathbb{U}_{ad}$, and $\mathcal{S}$ and $\mathcal{S}^*$ are stabilization terms. The stabilization for the state equation $\mathcal{S}$ is defined in Section \ref{stabilizations} for different stabilization methods. The stabilization for the adjoint equation reads:
\begin{eqnarray*}
\mathcal{{S}}^{*}(\bar{\psf}_{\T},\bar{\ysf}_{\T}-\ysf_{\Omega};\wsf_{\T}) =
\sum_{K\in\T}\mathcal{{S}}^{*}_{K}(\bar{\psf}_{\T},\bar{\ysf}_{\T}-\ysf_{\Omega};\wsf_{\T|K}),
\end{eqnarray*}
where the local terms $\mathcal{{S}}^{*}_{K} = \mathcal{{S}}^{*}_{K}(\bar{\psf}_{\T},\bar{\ysf}_{\T}-\ysf_{\Omega};\wsf_{\T|K})$, for the below mentioned stabilizations, are defined as follows:
\begin{equation}\label{eq:SUPGad}
\underline{\textit{SUPG}}:\quad
\mathcal{{S}}^{*}_{K}:=\tau^{*}_{K}
\left(-\bsf\cdot\nabla\bar{\psf}_{\T}+\kappa\bar{\psf}_{\T}-\bar{\ysf}_{\T}+\ysf_{\Omega}, -\bsf\cdot\nabla \wsf_{\T}
\right)_{L^{2}(K)}.
\end{equation}
\begin{equation}\label{eq:GLSad}
\underline{\textit{GLS}}:\quad
\mathcal{{S}}^{*}_{K}:=\tau^{*}_{K}
\left( -\bsf\cdot\nabla\bar{\psf}_{\T}+\kappa\bar{\psf}_{\T}-\bar{\ysf}_{\T}+\ysf_{\Omega},
-\bsf\cdot\nabla \wsf_{\T} + \kappa\wsf_{\T}\right)_{L^{2}(K)}.
\end{equation}
\begin{equation}\label{eq:CIPad}
\underline{\textit{CIP}}:\quad
\mathcal{{S}}^{*}_{K}:=\sum_{\gamma\in\mathcal{F}_{K}\cap\mathcal{F}_{I}}\tau^{*}_{\gamma}
\left(\llbracket \bsf\cdot\nabla\bar{\psf}_{\T}\rrbracket_{\gamma,K},  \bsf\cdot\nabla\wsf_{\T|K}  \right)_{L^{2}(\gamma)}.
\end{equation}
\begin{equation}\label{eq:ESad}
\underline{\textit{ES}}:\quad
\mathcal{{S}}^{*}_{K}
:=\sum_{\gamma\in\mathcal{F}_{K}\cap\mathcal{F}_{I}}\tau^{*}_{\gamma}
\left(\llbracket \nabla \bar{\psf}_{\T} \cdot \boldsymbol{n}_{\gamma} \rrbracket , \nabla \wsf_{\T|K} \cdot\boldsymbol{n}_{\gamma}^{K} (h_{K}^{2}+h_{K^\gamma}^{2})\right)_{L^{2}(\gamma)}.
\end{equation}
In all the above mentioned schemes $\tau^{*}_{K}$ and $\tau^{*}_{\gamma}$ denote nonnegative stabilization parameters that can vary from one method to another. 

\begin{remark}[optimize--then--discretize]
In general, there are two approaches to solve \eqref{min}--\eqref{control_constraint}: optimize--then--discretize and discretize--then--optimize. The first approach
is based on the discretization of the optimality system \eqref{optimal_system_1}. In contrast, the second approach first discretizes the optimal control problem \eqref{min}--\eqref{control_constraint} and then deduces the discrete optimality conditions. In principle, these two approaches do not coincide: they could lead to different discrete problems. If $\mathcal{S}$ and $\mathcal{S}^{*}$ are based on the same scheme and $\mathcal{S}$ is symmetric, then both approaches lead to the same discrete system.
\end{remark}

Before proceeding with the description of our solution technique for \eqref{min}--\eqref{control_constraint}, let us comment on those advocated in the literature. Concerning the a priori theory, to the best of our knowledge, the first work that analyzed a stabilized scheme is \cite{collis2002analysis}. This work considers the SUPG method, elaborates on the fact that the approaches optimize--then--discretize and discretize--then--optimize do not coincide and explores their respective advantages; see \cite{MR2460012} for an improvement on the theory. Later, local projection stabilization (LPS) techniques were proposed in \cite{MR2302057,MR2486088}. These techniques have the advantage that, due to the symmetry of the proposed stabilization term, optimize--then--discretize and discretize--then--optimize coincide. The ES scheme has also been employed to derive a discrete technique that approximates the solution to \eqref{min}--\eqref{control_constraint} \cite{MR2495058,MR2463111}: 
optimize--then--discretize and discretize--then--optimize coincide. We refer the reader to \cite{MR3246619} for a survey that includes other discretization techniques and an extensive list of references.

In contrast to this well--established theory, the a posteriori error analysis for stabilized finite element discretizations of \eqref{min}--\eqref{control_constraint} is not as well developed. We refer the reader to \cite{MR2178571,MR2495058,MR3212590,Nederkoorn,MR2460012,MR2463111} for a posteriori error estimators based on different stabilized schemes: ES scheme, discontinuous Galerkin methods and the Lagrange functional method. A common feature of all of the above--cited references is that the upper bound for the error in terms of the estimator, when it is derived, involves constants that are not known. This motivates the construction of fully computable a posteriori error estimators for the discretization \eqref{optimal_system_discrete_1} of the system \eqref{optimal_system_1}. In contrast to \cite{MR2178571,MR2495058,Nederkoorn,MR2460012,MR2463111}, we also study the efficiency properties of the proposed error indicators.

\subsection{A posteriori error analysis: reliability}
We assume that the discrete and non–linear problem \eqref{optimal_system_discrete_1} has a unique solution. We then construct a posteriori error estimators that are based on three contributions. First, we define the global a posteriori error estimator associated with the optimal control variable
\begin{equation}
 \label{eq:indicator_control}
 \eta_{ct}:=\left( \sum_{K \in \T} \eta_{ct,K}^{2}\right)^{1/2}\mbox{ where }\eta_{ct,K}:= \| \bar \usf_{\T} - \Pi(-\tfrac{1}{\vartheta} \bar \psf_{\T}) \|_{L^2(K)}.
\end{equation}
Here, $\Pi$ denotes the non--linear operator defined in \eqref{projection_formula}.

We now construct the error estimators associated with the state and adjoint optimal variables. To accomplish this task, we define $\hat{\ysf} \in H_0^1(\Omega)$ to be such that
\begin{equation}
\label{eq:y_hat}
\mathcal{B}(\hat{\ysf},\vsf)=(\fsf+\bar{\usf}_{\T},\vsf) \quad  \forall~ \vsf \in H_{0}^{1}(\Omega),
\end{equation}
and $\hat{\psf} \in H_0^1(\Omega)$ to be such that
\begin{equation}
\label{eq:p_hat}
\mathcal{B}^{*} (\hat{\psf},\wsf) = (\bar{\ysf}_{\T}-\ysf_{\Omega},\wsf) \quad  \forall~ \wsf \in H_{0}^{1}(\Omega).
\end{equation}
Performing the a posteriori error analysis presented in Section \ref{Fully_computable}, to bound the error between the solutions of \eqref{eq:y_hat} and the discretization of the state equation in \eqref{optimal_system_discrete_1}, and the error between the solutions of \eqref{eq:p_hat} and the discretization of the adjoint equation from \eqref{optimal_system_discrete_1}, we can conclude that
\begin{equation}\label{final_upper_bound_state_adjoint}
\norm{\hat{\ysf}-\bar{\ysf}_{\T}}_{\Omega}^{2} \leq \eta_{st}^{2}
\quad\textrm{and}\quad
\norm{\hat{\psf}-\bar{\psf}_{\T}}_{\Omega}^{2} \leq \eta_{ad}^{2},
\end{equation}
where, for $\varrho=st$ or $\varrho=ad$, the error estimators $\eta_{st}$ and $\eta_{ad}$ are defined by
\begin{equation}\label{global_estimator_state_adjoint}
\eta_{\varrho}^{2}: = \sum_{K\in\T}\eta_{\varrho,K}^{2}, \quad
\eta_{\varrho,K}:= \frac{1}{\sqrt{\kappa|K|}} \left|\mathcal{S}^{\varrho}_{K}(1)\right|
+
\frac{\|\boldsymbol{\sigma}_{K}^{\varrho}\|_{\boldsymbol{L}^{2}(K)}}{\sqrt{\nu}}
+
\mathsf{C}_{\textrm{osc},K}\|\textrm{osc}_{K}^{\varrho}\|_{L^{2}(K)}. 
\end{equation}
Here, for all $K\in\T$, $\boldsymbol{\sigma}_{K}^{\varrho}\in\mathbb{P}_2(K)\times\mathbb{P}_2(K)$ denotes the solution to 
\begin{equation}\label{sigmadiv_state}
\left\{
\begin{array}{rcl}
-\mathbf{div}~\boldsymbol{\sigma}_{K}^{\varrho} & = & \mathscr{R}_{K}^{\varrho}- \frac{1}{|K|} \left(\mathscr{R}_{K}^{\varrho},1\right)_{L^{2}(K)}- \frac{1}{|K|} \sum_{\gamma\in\mathcal{F}_{K}} (\mathscr{R}_{\gamma,K}^{\varrho},1)_{L^{2}(\gamma)} \quad
\mbox{in }K,\\
\boldsymbol{\sigma}_{K}^{\varrho}\cdot\boldsymbol{n}_{\gamma}^{K} & = & \mathscr{R}_{\gamma,K}^{\varrho} \quad
\mbox{on }\gamma\mbox{ for all }\gamma\in\mathcal{F}_{K},
\end{array}
\right.
\end{equation}
which is such that $\|\boldsymbol{\sigma}_{K}^{\varrho}\|_{\boldsymbol{L}^{2}(K)}$ is minimized, with residuals and oscillation terms defined as
\begin{equation}\label{residuals_osc_state}
\left\{
\begin{array}{rcl}
\mathscr{R}_{K}^{st} & := & \Pi_{K}(\fsf)+\bar{\usf}_{\T|K}-\Pi_{K}(\bsf\cdot\nabla\bar{\ysf}_{\T})-\kappa\bar{\ysf}_{\T|K},
\\ 
\mathscr{R}_{\gamma,K}^{st} & := & g_{\gamma,K}^{st}-\nu \nabla\bar{\ysf}_{\T}\cdot\boldsymbol{n}_{\gamma}^{K},
\\
\textrm{osc}_{K}^{st} & := & \fsf-\Pi_{K}(\fsf)-(\bsf\cdot\nabla\bar{\ysf}_{\T|K}-\Pi_{K}(\bsf\cdot\nabla\bar{\ysf}_{\T})),
\end{array}
\right.
\end{equation}
and
\begin{equation}
\label{residuals_osc_adjoint}
\left\{
\begin{array}{rcl}
\mathscr{R}_{K}^{ad} & := & \bar{\ysf}_{\T|K}-\Pi_{K}(\ysf_{\Omega})+\Pi_{K}(\bsf\cdot\nabla\bar{\psf}_{\T})-\kappa\bar{\psf}_{\T|K},
\\
\mathscr{R}_{\gamma,K}^{ad} & := & g_{\gamma,K}^{ad}-\nu \nabla\bar{\psf}_{\T}\cdot\boldsymbol{n}_{\gamma}^{K},
\\
\textrm{osc}_{K}^{ad} & := & -(\ysf_{\Omega}-\Pi_{K}(\ysf_{\Omega}))+\bsf\cdot\nabla\bar{\psf}_{\T|K}-\Pi_{K}(\bsf\cdot\nabla\bar{\psf}_{\T}).
\end{array}
\right.
\end{equation}
The \RR{equilibrated} boundary fluxes $\{g_{\gamma,K}^{\varrho}\}$ are constructed on the basis of the material presented in Section \ref{boundary_fluxes}. First, they must satisfy the consistency property
\begin{equation}\label{consistency}
g_{\gamma,K}^{\varrho}+g_{\gamma,K'}^{\varrho}=0,\quad
\textrm{if}~\gamma\in \mathcal{F}_{K}\cap\mathcal{F}_{K'},~
K,K'\in\T,~K\neq K'.
\end{equation}
In addition, they must satisfy the first order equilibration condition that is
\begin{equation*}
0=(\fsf+\bar{\usf}_{\T},\lambda)_{L^{2}(K)}-\mathcal{B}_{K}\left(\bar{\ysf}_{\T},\lambda\right)+
\sum_{\gamma\in\mathcal{F}_{K}} ( g_{\gamma,K}^{st},\lambda )_{L^{2}(\gamma)}
-\mathcal{S}_{K}^{st}(\lambda),
\end{equation*}
and
\begin{equation*}
0=(\bar{\ysf}_{\T}-\ysf_{\Omega},\lambda)_{L^{2}(K)}-\mathcal{B}_{K}^{*}\left(\bar{\psf}_{\T},\lambda\right)+
\sum_{\gamma\in\mathcal{F}_{K}}( g_{\gamma,K}^{ad},\lambda )_{L^{2}(\gamma)}
-\mathcal{S}_{K}^{ad}(\lambda),
\end{equation*}
for all $\lambda\in\mathbb{P}_{1}(K)$ and all $K\in\T$, and where $\mathcal{B}_{K}\left(\bar{\ysf}_{\T},\lambda\right)=\nu(\nabla\bar{\ysf}_{\T},\nabla \lambda)_{L^{2}(K)}+
(\boldsymbol{b}\cdot\nabla\bar{\ysf}_{\T}+\kappa\bar{\ysf}_{\T},\lambda)_{L^{2}(K)}$, $\mathcal{B}_{K}^{*}\left(\bar{\psf}_{\T},\lambda\right)=\nu(\nabla\bar{\psf}_{\T},\nabla \lambda)_{L^{2}(K)}+
(\kappa\bar{\psf}_{\T}-\boldsymbol{b}\cdot\nabla\bar{\psf}_{\T},\lambda)_{L^{2}(K)}$ and 
\begin{equation}
\mathcal{S}^{\varrho}_K (\lambda):=
\left\{
\begin{array}{cl}
\mathcal{S}_{K}\left(\bar{\ysf}_{\T},\fsf+\bar{\usf}_{\T};\lambda\right) & \textrm{for}~\varrho=st,\\
\mathcal{S}^*_{K}\left(\bar{\psf}_{\T},\bar{\ysf}_{\T} - \ysf_{\Omega};\lambda\right) & \textrm{for}~\varrho=ad.
\end{array}
\right.
\end{equation}
\RR{Finally, they must satisfy the corresponding analog of \eqref{gmJ}.}

\RR{For $G=\Omega$ or $G\in\T$, we define
\[
  \|(e_{\bar{\ysf}},e_{\bar{\psf}},e_{\bar{\usf}})\|_{G} = \left(\norm{\bar{\ysf} - \bar{\ysf}_{\T}}_{G}^2 
+ \norm{\bar{\psf} - \bar{\psf}_{\T}}_{G}^2
+ \| \bar{\usf} - \bar{\usf}_{\T}\|_{L^2(G)}^2\right)^{1/2}.
\]}
We now present the analysis through which we obtain a fully computable upper bound for the total error for our optimal control problem.
\begin{theorem}[global reliability]\label{th:global_reliability}
Let $(\bar{\ysf},\bar{\psf},\bar{\usf}) \in H_0^1(\Omega) \times H_0^1(\Omega) \times L^2(\Omega)$ be the solution to \eqref{optimal_system_1} and $(\bar{\ysf}_{\T},\bar{\psf}_{\T},\bar{\usf}_{\T}) \in \V(\T) \times \V(\T) \times \U_{\textrm{ad}}(\T)$ its numerical approximation obtained as the solution to \eqref{optimal_system_discrete_1}, then 
\begin{equation}
\label{eq:reliability}
\RR{\|(e_{\bar{\ysf}},e_{\bar{\psf}},e_{\bar{\usf}})\|_{\Omega}}
\leq 
\Upsilon=\left(\sum_{K\in\T}\Upsilon_K^2\right)^{1/2}
\end{equation}
where 
\begin{equation}
\label{eq:indicatorOC}
\Upsilon_K^2:=\mathrm{C}_{st} 
\eta_{st,K}^{2} +
\mathrm{C}_{ad} 
\eta_{ad,K}^{2} +
\mathrm{C}_{ct} 
\eta_{ct,K}^{2},
\end{equation}
with $\eta_{st}^{2}$ and $\eta_{ad}^{2}$ being given in \eqref{global_estimator_state_adjoint}, $\eta_{ct}^{2}$ being defined in \eqref{eq:indicator_control} and
\begin{gather*}
\mathrm{C}_{st} :=
2+\frac{4}{\kappa^2}+\frac{8}{\vartheta^2\kappa^6}\left(\kappa^3+2\kappa^2+4\right)
,\quad
\mathrm{C}_{ad} :=
2+\frac{4}{\vartheta^2\kappa^4}\left(\kappa^3+2\kappa^2+4\right)
,\\
\mathrm{C}_{ct} :=
2+\frac{4}{\kappa}+\frac{8}{\kappa^3}+\frac{8}{\vartheta^2\kappa^7}\left(\kappa^3+2\kappa^2+4\right).
\end{gather*}
\end{theorem}
\begin{proof}
We proceed in four steps.
~\\
\noindent \framebox{Step 1.} The goal of this step is to control the error $\bar{\usf}-\bar{\usf}_{\T}$. We define $\tilde{\usf} = \Pi (-\tfrac{1}{\vartheta}\bar{\psf}_{\T})$, which can be equivalently characterized by
 \begin{equation}
  \label{tildeu}
  (\bar{\psf}_{\T}+\vartheta \tilde{\usf}, \usf - \tilde{\usf})_{L^2(\Omega)} \geq 0 \quad \forall \usf \in \mathbb{U}_{\textrm{ad}}.
 \end{equation}
With this definition at hand, we have that
\begin{equation}
 \label{u-uh-utilde}
 \|  \bar{\usf} -\bar{\usf}_{\T} \|_{L^2(\Omega)}^{2} \leq 2 \left(\|\bar{\usf}-\tilde{\usf}\|_{L^2(\Omega)}^{2}+\|\tilde{\usf}-\bar{\usf}_{\T}\|_{L^2(\Omega)}^{2}\right)=2\|  \bar{\usf} -\tilde{\usf} \|_{L^2(\Omega)}^{2} + 2\eta_{ct}^{2}.
\end{equation}

Let us now focus on the first term on the right hand side of \eqref{u-uh-utilde}. The variational inequality of \eqref{optimal_system_1} with $\usf = \tilde \usf$ and \eqref{tildeu} with $\usf = \bar \usf$ yield that
$$
(\bar{\psf}+\vartheta\bar{\usf},\tilde{\usf}-\bar{\usf})_{L^{2}(\Omega)} \geq 0
\mbox{ and }
(\bar{\psf}_{\T}+\vartheta \tilde{\usf}, \bar{\usf} - \tilde{\usf})_{L^2(\Omega)} \geq 0,
$$
from which it follows that
$$
\vartheta\| \bar{\usf} - \tilde{\usf} \|^2_{L^2(\Omega)} \leq (\bar{\psf} - \bar{\psf}_{\T}, \tilde{\usf} - \bar{\usf})_{L^2(\Omega)}.
$$
To bound the right hand side of the above expression, we let $(\tilde{\ysf},\tilde{\psf})$ be such that
\begin{eqnarray*}
\left\{
\begin{array}{rrcll}
\tilde{\ysf}\in H_{0}^{1}(\Omega): & 
\mathcal{B}(\tilde{\ysf},\vsf) &=& (\fsf+\tilde{\usf},\vsf)  &  \forall~ \vsf \in H_{0}^{1}(\Omega),\\
\tilde{\psf}\in H_{0}^{1}(\Omega): & 
\mathcal{B}^{*}(\tilde{\psf},\wsf) &=& (\tilde{\ysf}-\ysf_{\Omega},\wsf) &  \forall~ \wsf \in H_{0}^{1}(\Omega).
\end{array}
\right.
\end{eqnarray*}
We then have that 
\begin{multline*}
\vartheta\| \bar{\usf} - \tilde{\usf} \|^2_{L^2(\Omega)}
\leq 
(\bar{\psf}  - \tilde{\psf}, \tilde{\usf} - \bar{\usf})_{L^2(\Omega)} + 
(\tilde{\psf}  - \hat{\psf}, \tilde{\usf} - \bar{\usf})_{L^2(\Omega)} +
(\hat{\psf} - \bar{\psf}_{\T}, \tilde{\usf} - \bar{\usf})_{L^2(\Omega)} \\
\leq 
(\bar{\psf}  - \tilde{\psf}, \tilde{\usf} - \bar{\usf})_{L^2(\Omega)} + 
\tfrac{1}{\vartheta}\|\tilde{\psf}  - \hat{\psf}\|_{L^{2}(\Omega)}^{2} +
\tfrac{1}{\vartheta}\|\hat{\psf} - \bar{\psf}_{\T}\|_{L^{2}(\Omega)}^{2} +
\tfrac{\vartheta}{2}
\|\bar{\usf} - \tilde{\usf}\|_{L^2(\Omega)}^{2}
\end{multline*}
upon using Cauchy--Schwarz and Young's inequalities. Hence,
\begin{equation}\label{eq:usf-tildeusf_2}
\| \bar{\usf} - \tilde{\usf} \|^2_{L^2(\Omega)} \leq 
\tfrac{2}{\vartheta}(\bar{\psf}  - \tilde{\psf}, \tilde{\usf} - \bar{\usf})_{L^2(\Omega)} + 
\tfrac{2}{\vartheta^2}\|\tilde{\psf}  - \hat{\psf}\|_{L^{2}(\Omega)}^{2} +
\tfrac{2}{\vartheta^2}\|\hat{\psf} - \bar{\psf}_{\T}\|_{L^{2}(\Omega)}^{2}.
\end{equation}
We now proceed to bound the term $(\bar{\psf}  - \tilde{\psf}, \tilde{\usf} - \bar{\usf})_{L^2(\Omega)}$. To do this, we note that $\tilde{\ysf}-\bar{\ysf}$ is such that $\mathcal{B}(\tilde{\ysf}-\bar{\ysf},\vsf)=(\tilde{\usf}-\bar{\usf},\vsf)_{L^{2}(\Omega)}$ for all $\vsf \in H_{0}^{1}(\Omega)$ and that $\bar{\psf}-\tilde{\psf}$ solves $\mathcal{B}^{*}(\bar{\psf}-\tilde{\psf},\wsf)=(\bar{\ysf}-\tilde{\ysf},\wsf)$ for all $\wsf \in H_{0}^{1}(\Omega)$. Hence,
\[
 (\bar{\psf}  - \tilde{\psf}, \tilde{\usf} - \bar{\usf})_{L^2(\Omega)} = \mathcal{B}(\tilde{\ysf}-\bar{\ysf},\bar{\psf}  - \tilde{\psf}) = \mathcal{B}^{*}(\bar{\psf}-\tilde{\psf},\tilde{\ysf}-\bar{\ysf})_{L^{2}(\Omega)} = -\|\bar{\ysf}-\tilde{\ysf}\|_{L^{2}(\Omega)}^{2}\leq 0.
\]
This, in conjunction with \eqref{eq:usf-tildeusf_2} yields
\begin{equation}\label{eq:usf-tildeusf_3}
\| \bar{\usf} - \tilde{\usf} \|^2_{L^2(\Omega)} 
\leq 
\tfrac{2}{\vartheta^2}\| \tilde{\psf}  - \hat{\psf}\|_{L^{2}(\Omega)}^{2} +
\tfrac{2}{\vartheta^2}\| \hat{\psf} - \bar{\psf}_{\T}\|_{L^{2}(\Omega)}^{2}
\leq
\tfrac{2}{\vartheta^2}\| \tilde{\psf}  - \hat{\psf}\|_{L^{2}(\Omega)}^{2} + \tfrac{2}{\vartheta^2\kappa}\eta_{ad}^{2}
\end{equation}
since $\| \hat{\psf} - \bar{\psf}_{\T}\|_{L^{2}(\Omega)}^{2}\le\tfrac{1}{\kappa}\norm{\hat{\psf} - \bar{\psf}_{\T}}^2\le\tfrac{1}{\kappa}\eta_{ad}^{2}$ because of \eqref{final_upper_bound_state_adjoint}. To bound $\| \tilde{\psf}  - \hat{\psf}\|_{L^{2}(\Omega)}^{2}$ we first note that $\mathcal{B}^{*}(\tilde{\psf}  - \hat{\psf},\wsf)=(\tilde{\ysf}-\bar{\ysf}_{\T},\wsf)_{L^{2}(\Omega)}$ for all $\wsf\in H_{0}^{1}(\Omega)$. So, taking $\wsf=\tilde{\psf}  - \hat{\psf}$ and using the fact that $\boldsymbol{b}$ is a solenoidal field allows us to conclude that
$$
\kappa\|\tilde{\psf}  - \hat{\psf}\|_{L^{2}(\Omega)}^{2} 
\leq 
\mathcal{B}^{*}(\tilde{\psf}  - \hat{\psf},\tilde{\psf}  - \hat{\psf})=(\tilde{\ysf}-\bar{\ysf}_{\T},\tilde{\psf}  - \hat{\psf})_{L^{2}(\Omega)}=(\tilde{\ysf}-\hat{\ysf}+\hat{\ysf}-\bar{\ysf}_{\T},\tilde{\psf}  - \hat{\psf})_{L^{2}(\Omega)}
$$
and hence,
$$
\|\tilde{\psf}  - \hat{\psf}\|_{L^{2}(\Omega)}^{2} 
\leq 
\tfrac{1}{\kappa}\left(\|\tilde{\ysf}-\hat{\ysf}\|_{L^{2}(\Omega)}+\|\hat{\ysf}-\bar{\ysf}_{\T}\|_{L^{2}(\Omega)}\right)
\|\tilde{\psf}  - \hat{\psf}\|_{L^{2}(\Omega)}.
$$
Consequently,
$$
\|\tilde{\psf}  - \hat{\psf}\|_{L^{2}(\Omega)}^{2} 
\leq 
\tfrac{2}{\kappa^2}\left(\|\tilde{\ysf}-\hat{\ysf}\|_{L^{2}(\Omega)}^2+\|\hat{\ysf}-\bar{\ysf}_{\T}\|_{L^{2}(\Omega)}^2\right)
\leq
\tfrac{2}{\kappa^2}\|\tilde{\ysf}-\hat{\ysf}\|_{L^{2}(\Omega)}^2+\tfrac{2}{\kappa^3}\eta_{st}^{2},
$$
since $\| \hat{\ysf} - \bar{\ysf}_{\T}\|_{L^{2}(\Omega)}^{2}\le\tfrac{1}{\kappa}\norm{\hat{\ysf} - \bar{\ysf}_{\T}}^2\le\tfrac{1}{\kappa}\eta_{st}^{2}$. This, in conjunction with \eqref{eq:usf-tildeusf_3}, yields
\begin{equation}\label{eq:usf-tildeusf_4}
\| \bar{\usf} - \tilde{\usf} \|^2_{L^2(\Omega)} 
\leq
\tfrac{4}{\vartheta^2\kappa^2}\|\tilde{\ysf}-\hat{\ysf}\|_{L^{2}(\Omega)}^2
+\tfrac{4}{\vartheta^2\kappa^3}\eta_{st}^{2}
+\tfrac{2}{\vartheta^2\kappa}\eta_{ad}^{2}.
\end{equation}
To bound $\| \tilde{\ysf}  - \hat{\ysf}\|_{L^{2}(\Omega)}^{2}$ we first note that $\mathcal{B}(\tilde{\ysf}  - \hat{\ysf},\vsf)=(\tilde{\usf}-\bar{\usf}_{\T},\vsf)_{L^{2}(\Omega)}$ for all $\vsf\in H_{0}^{1}(\Omega)$. So, taking $\vsf=\tilde{\ysf}  - \hat{\ysf}$ 
allows us to conclude that
$$
\kappa\|\tilde{\ysf}  - \hat{\ysf}\|_{L^{2}(\Omega)}^{2} 
\leq 
\mathcal{B}(\tilde{\ysf}  - \hat{\ysf},\tilde{\ysf}  - \hat{\ysf})=(\tilde{\usf}-\bar{\usf}_{\T},\tilde{\ysf}  - \hat{\ysf})_{L^{2}(\Omega)}
\leq 
\|\tilde{\usf}-\bar{\usf}_{\T}\|_{L^{2}(\Omega)}\|\tilde{\ysf}  - \hat{\ysf}\|_{L^{2}(\Omega)}
$$
and hence, 
$
\|\tilde{\ysf}  - \hat{\ysf}\|_{L^{2}(\Omega)}^{2} 
\leq
\tfrac{1}{\kappa^2}\|\tilde{\usf}-\bar{\usf}_{\T}\|_{L^{2}(\Omega)}^2
=
\tfrac{1}{\kappa^2}\eta_{ct}^{2}.
$
Combining this with \eqref{eq:usf-tildeusf_4} implies that
$$
\| \bar{\usf} - \tilde{\usf} \|^2_{L^2(\Omega)} 
\leq
\tfrac{4}{\vartheta^2\kappa^4}\eta_{ct}^{2}
+\tfrac{4}{\vartheta^2\kappa^3}\eta_{st}^{2}
+\tfrac{2}{\vartheta^2\kappa}\eta_{ad}^{2}
$$
which together with \eqref{u-uh-utilde}, allows us to conclude that
\begin{equation}\label{error_u}
\| \bar{\usf} - \bar{\usf}_{\T} \|^2_{L^2(\Omega)} 
\leq
\tfrac{8}{\vartheta^2\kappa^3}\eta_{st}^{2}
+\tfrac{4}{\vartheta^2\kappa}\eta_{ad}^{2}
+\left(2+\tfrac{8}{\vartheta^2\kappa^4}\right)\eta_{ct}^{2}.
\end{equation}
\noindent \framebox{Step 2.} The goal of this step is to control the error $\bar{\ysf}-\bar{\ysf}_{\T}$. Now, we have that 
\begin{equation}\label{step2-1}
\norm{\bar{\ysf}-\bar{\ysf}_{\T}}_{\Omega}^{2}\leq 2 (\norm{\bar{\ysf}-\hat{\ysf}}_{\Omega}^{2}+\norm{\hat{\ysf}-\bar{\ysf}_{\T}}_{\Omega}^{2})\le 
2\norm{\bar{\ysf}-\hat{\ysf}}_{\Omega}^{2}+2\eta_{st}^{2}.
\end{equation}
Moreover, $\norm{\bar{\ysf}-\hat{\ysf}}_{\Omega}^{2}\leq \tfrac{1}{\kappa}
\|\bar{\usf}-\bar{\usf}_{\T}\|_{L^{2}(\Omega)}^{2}$, since
$$
\norm{\bar{\ysf}-\hat{\ysf}}_{\Omega}^{2}=\mathcal{B}(\bar{\ysf}-\hat{\ysf},\bar{\ysf}-\hat{\ysf})=
(\bar{\usf}-\bar{\usf}_{\T},\bar{\ysf}-\hat{\ysf})_{L^{2}(\Omega)}\leq \tfrac{1}{\sqrt{\kappa}}
\|\bar{\usf}-\bar{\usf}_{\T}\|_{L^{2}(\Omega)}
\norm{\bar{\ysf}-\hat{\ysf}}_{\Omega}.
$$
Therefore, upon combining this with \eqref{step2-1} and \eqref{error_u}, we can conclude that
\begin{equation}\label{error_y}
\norm{\bar{\ysf}-\bar{\ysf}_{\T}}_{\Omega}^{2}
\leq
\left(2+\tfrac{16}{\vartheta^2\kappa^4}\right)\eta_{st}^{2}
+\tfrac{8}{\vartheta^2\kappa^2}\eta_{ad}^{2}
+\left(\tfrac{4}{\kappa}+\tfrac{16}{\vartheta^2\kappa^5}\right)\eta_{ct}^{2}.
\end{equation}
\noindent \framebox{Step 3.} The goal of this step is to control the error $\bar{\psf}-\bar{\psf}_{\T}$. Now, we have that 
\begin{equation}\label{step3-1}
\norm{\bar{\psf}-\bar{\psf}_{\T}}_{\Omega}^{2}\leq 2 (\norm{\bar{\psf}-\hat{\psf}}_{\Omega}^{2}+\norm{\hat{\psf}-\bar{\psf}_{\T}}_{\Omega}^{2})\le 
2\norm{\bar{\psf}-\hat{\psf}}_{\Omega}^{2}+2\eta_{ad}^{2}.
\end{equation}
Moreover, $\norm{\bar{\psf}-\hat{\psf}}_{\Omega}^{2}\leq \tfrac{1}{\kappa^2}
\norm{\bar{\ysf}-\bar \ysf_{\T}}_{\Omega}^{2}$, since
$$
\norm{\bar{\psf}-\hat{\psf}}_{\Omega}^{2}=\mathcal{B}^{*}(\bar{\psf}-\hat{\psf},\bar{\psf}-\hat{\psf})=
(\bar{\ysf}-\bar{\ysf}_{\T},\bar{\psf}-\hat{\psf})_{L^{2}(\Omega)}\leq \tfrac{1}{\kappa}
\norm{\bar{\ysf}-\bar \ysf_{\T}}_{\Omega}
\norm{\bar{\psf}-\hat{\psf}}_{\Omega}.
$$
Therefore, upon combining this with \eqref{step3-1} and \eqref{error_y}, we can conclude that
\begin{equation}\label{error_p}
\norm{\bar{\psf}-\bar{\psf}_{\T}}_{\Omega}^{2}
\leq
\left(\tfrac{4}{\kappa^2}+\tfrac{32}{\vartheta^2\kappa^6}\right)\eta_{st}^{2}
+\left(2+\tfrac{16}{\vartheta^2\kappa^4}\right)\eta_{ad}^{2}
+\left(\tfrac{8}{\kappa^3}+\tfrac{32}{\vartheta^2\kappa^7}\right)\eta_{ct}^{2}.
\end{equation}
\noindent \framebox{Step 4.} The result claimed follows upon gathering \eqref{error_u}, \eqref{error_y} and \eqref{error_p}.
\end{proof}
\begin{remark}[fully computable a posteriori upper bound]
The bound \eqref{eq:reliability} is a genuine upper bound in the sense that the value of the estimator exceeds the value of the true error regardless of the coarseness of the mesh or the nature of the data of the problem. All constants appearing in the bound are fully specified.
\end{remark}
\begin{remark}[Poisson problem]
If we set $\nu=1$, $\bsf = 0$ and $\kappa =0$, the state and adjoint equations become the Poisson problem. By examining the proof of Theorem \ref{th:global_reliability}, it can be seen that if we apply the Poincar\'e inequality $\|\vsf\|_{L^{2}(\Omega)}^2\le C_P\|\nabla \vsf\|_{L^{2}(\Omega)}^2$ for all $\vsf \in H^{1}_0 (\Omega)$, instead of $\|\vsf\|_{L^{2}(\Omega)}^2 \leq \kappa^{-1} \norm{\vsf}_{\Omega}^{2}$, then \eqref{eq:reliability} will hold if the constants $\mathrm{C}_{st}$, $\mathrm{C}_{ad}$ and $\mathrm{C}_{ct}$ in \eqref{eq:indicator_control} are taken to be
\begin{gather*}
\mathrm{C}_{st} =
2+4C_P^2+\frac{8}{\vartheta^2}\left(C_P^3+2C_P^4+4C_P^6\right)
,\quad
\mathrm{C}_{ad} =
2+\frac{4}{\vartheta^2}\left(C_P+2C_P^2+4C_P^4\right)
,\\
\mathrm{C}_{ct} =
2+4C_P+8C_P^3+\frac{8}{\vartheta^2}\left(C_P^4+2C_P^5+4C_P^7\right).
\end{gather*}
The resulting bound is fully computable because an upper bound for $C_{P}$ can be found, for instance, in \cite{MR853915}. 
\end{remark}
\subsection{Error estimator: efficiency}
\label{subsec:efficiency}
We first write the error equation for the state variable and its approximation, for all $\vsf \in H_{0}^{1}(\Omega)$, as
\begin{align}\nonumber
&\sum_{K\in\T}
(\mathscr{R}_{K}^{st},\vsf)_{L^{2}(K)}
-
2\sum_{\gamma\in\mathcal{F}_{I}}
(\llbracket J_{\gamma}^{st}\rrbracket,\vsf)_{L^{2}(\gamma)}\\\label{eq:jump}
&~~=
\mathcal{B}(\bar{\ysf}-\bar{\ysf}_{\T},\vsf)-(\bar{\usf}-\bar{\usf}_{\T},\vsf)_{L^{2}(\Omega)}-\sum_{K\in\T}(\textrm{osc}_{K}^{st},\vsf)_{L^{2}(K)},
\end{align}
with $\llbracket J_{\gamma}^{st}\rrbracket:=\tfrac{1}{2}(J_{\gamma,K}^{st}+J_{\gamma,K'}^{st})$ if $\gamma\in \mathcal{F}_{K}\cap\mathcal{F}_{K'},~K\neq K'$ where $J_{\gamma,K}^{st}:=\nu\nabla\bar{\ysf}_{\T|K}\cdot \boldsymbol{n}_{\gamma}^{K}$. Bubble function arguments then lead to
$$
\tfrac{h_{K}}{\sqrt{\nu}}\|\mathscr{R}_{K}^{st}\|_{L^{2}(K)}
\leq C
\left(
\mathsf{C}_{K}\norm{\bar{\ysf}-\bar{\ysf}_{\T}}_{K}+\tfrac{h_{K}}{\sqrt{\nu}}\left(\|\textrm{osc}_{K}^{st}\|_{L^{2}(K)}
+\|\bar{\usf}-\bar{\usf}_{\T}\|_{L^{2}(K)}\right)\right),
$$
and 
\begin{align*}
\left(\tfrac{h_{K}}{\nu}\right)^{1/2}\|\llbracket J_{\gamma}^{st}\rrbracket\|_{L^{2}(\gamma)}
& \leq C
\sum_{K'\in\Omega_{\gamma}}
\Big(
\mathsf{C}_{K'}\norm{\bar{\ysf}-\bar{\ysf}_{\T}}_{K'}  \\
& \qquad\qquad\qquad + \tfrac{h_{K'}}{\sqrt{\nu}}\left(\|\textrm{osc}_{K'}^{st}\|_{L^{2}(K')}+\|\bar{\usf}-\bar{\usf}_{\T}\|_{L^{2}(K')}
\right)
\Big).
\end{align*}
Now, we write the error equation for the adjoint variable and its approximation as
\begin{align*}
&\sum_{K\in\T}
(\mathscr{R}_{K}^{ad},\vsf)_{L^{2}(K)}
-
2\sum_{\gamma\in\mathcal{F}_{I}}
(\llbracket J_{\gamma}^{ad}\rrbracket,\vsf)_{L^{2}(\gamma)}\\
&~~=
\mathcal{B}^{*}(\bar{\psf}-\bar{\psf}_{\T},\vsf)-(\bar{\ysf}-\bar{\ysf}_{\T},\vsf)_{L^{2}(\Omega)}-\sum_{K\in\T}(\textrm{osc}_{K}^{ad},\vsf)_{L^{2}(K)},
\end{align*}
for all $\vsf \in H_{0}^{1}(\Omega)$, where, with $J_{\gamma,K}^{ad}:=\nu\nabla\bar{\psf}_{\T|K}\cdot\hat{\boldsymbol{n}}_{\gamma}^{K}$, $\llbracket J_{\gamma}^{ad}\rrbracket$ is defined analogously to $\llbracket J_{\gamma}^{st}\rrbracket$. By again using bubble function arguments, we can establish that 
$$
\tfrac{h_{K}}{\sqrt{\nu}}\|\mathscr{R}_{K}^{ad}\|_{L^{2}(K)}
\leq C
\left(
\mathsf{C}_{K}\norm{\bar{\psf}-\bar{\psf}_{\T}}_{K}+\tfrac{h_{K}}{\sqrt{\nu}}\left(\|\textrm{osc}_{K}^{ad}\|_{L^{2}(K)}
+\tfrac{1}{\sqrt{\kappa}}\norm{\bar{\ysf}-\bar{\ysf}_{\T}}_{(K)}\right)\right),
$$
and 
\begin{align*}
\left(\tfrac{h_{K}}{\nu}\right)^{1/2}\|\llbracket J_{\gamma}^{ad}\rrbracket\|_{L^{2}(\gamma)}
& \leq C
\sum_{K'\in\Omega_{\gamma}}
\Bigg(
\mathsf{C}_{K'}\norm{\bar{\psf}-\bar{\psf}_{\T}}_{K'}  \\
& \qquad\qquad\qquad + \tfrac{h_{K'}}{\sqrt{\nu}}\left(\|\textrm{osc}_{K'}^{ad}\|_{L^{2}(K')}+
\tfrac{1}{\sqrt{\kappa}}\norm{\bar{\ysf}-\bar{\ysf}_{\T}}_{(K')}
\right)
\Bigg).
\end{align*}
Moreover, an application of the triangle inequality, \RR{\eqref{eq:u} and \eqref{Lipschitz}}, yield that
\begin{align*}
\eta_{ct\RR{,K}}=\left\| \bar{\usf}_{\T} - \Pi (-\tfrac{1}{\vartheta}\bar{\psf}_{\T}) \right\|_{L^2(K)} 
& \leq 
\left\| \bar{\usf}- \bar{\usf}_{\T}  \right\|_{L^2(K)} + 
\left\| \Pi (-\tfrac{1}{\vartheta}\bar{\psf})-\Pi (-\tfrac{1}{\vartheta}\bar{\psf}_{\T})  \right\|_{L^2(K)} \\
& \leq 
\left\| \bar{\usf}- \bar{\usf}_{\T}  \right\|_{L^2(K)} + 
\tfrac{1}{\vartheta\sqrt{\kappa}}
\norm{\bar{\psf}-\bar{\psf}_{\T}}_{K}
\end{align*}
\RR{and hence
\begin{equation}\label{ct_efficiency}
\eta_{ct,K}^2\leq 
2\left(\left\| \bar{\usf}- \bar{\usf}_{\T}  \right\|_{L^2(K)}^2 + 
\tfrac{1}{\vartheta^2\kappa}
\norm{\bar{\psf}-\bar{\psf}_{\T}}_{K}^2\right).
\end{equation}}

Gathering all of the previous results with \eqref{global_estimator_state_adjoint} and following the analysis presented in Section \ref{efficiency}, allows us to conclude the following result.
\begin{theorem}[local efficiency]\label{theorem:local efficiency}
Let $(\bar{\ysf},\bar{\psf},\bar{\usf})$ be the solution to \eqref{optimal_system_1} and $(\bar{\ysf}_{\T},\bar{\psf}_{\T},$ $\bar{\usf}_{\T})$ be the solution to \eqref{optimal_system_discrete_1}. In addition, let the stabilization parameters be such that, for all $K\in \T$ and for both the state and adjoint equations, $\mathsf{C}_{\mathcal{S}_{K}}$ can be bounded by a constant which is independent of the size of the elements in the mesh. Then, 
\begin{align*}
\Upsilon_{K}^2 \leq &\mathsf{C}_{ef} \sum_{K'\in D_{K}}
\left(
\RR{\|(e_{\bar{\ysf}},e_{\bar{\psf}},e_{\bar{\usf}})\|_{K'}^2} 
+ h_{K'}\left(\|\mathrm{osc}_{K'}^{st}\|_{L^{2}(K')}^2+\|\mathrm{osc}_{K'}^{ad}\|_{L^{2}(K')}^2  \right)
\right),
\end{align*}
where the constant $\mathsf{C}_{ef}$ depends on the physical parameters in \eqref{optimal_system_1} but is independent of the size of the elements in the mesh and
$D_{K}=\{K'\in\T:~\mathcal{F}_{K'}\cap\mathcal{F}_{K''}\neq\emptyset,~K''\in\hat{\Omega}_{K}\}$ with $\hat{\Omega}_{K}=\{K'\in\T:~K'\cap K \neq\emptyset\}$.
\end{theorem}
\section{Robust a posteriori error estimation}\label{robustsection}
\RR{
In this section we derive and analyze robust a posteriori error estimates for the approximation of the optimalilty system \eqref{optimal_system_1} \RR{obtained} using the stabilized scheme \eqref{optimal_system_discrete_1}. We immediately comment that by robustness we mean that the constants involved in the upper and lower bounds for the errror are independent of the diffusion parameter $\nu$ and the vector field $\mathbf{b}$. Regarding \RR{the} stabilization, in this section we \RR{only} consider the following techniques: the streamline upwind Petrov--Galerkin method (SUPG) and the continuous interior penalty method (CIP).

The analysis is based on the results by Verf{\"u}rth \cite{MR2182149,MR3059294} and Tobiska and Verf{\"u}rth  \cite{MR3407239}. As in these works, we measure the error in a norm that adds, to the energy norm, the dual norm of the convective derivative. We also refer the reader to \cite{MR3123245,MR3394441,MR2601287,MR3029040,MR1855830,MR2353943} for different approaches.

\subsection{The state \RR{and adjoint equations}}

We assume that \EO{the} data of \EO{the} problem \eqref{weak_advection-reaction-diffusion} satisfy, in addition to the assumptions \textbf{(A1)}--\textbf{(A3)} of section \ref{subsec:state_equation}, the following assumption:
\begin{itemize}
\item[\textbf{(B1)}] $0<  \nu \ll 1$.
\end{itemize}
This assumption emphasize\RR{s} that, in this section, we are interested in the convection--dominated regime.

The presented analysis hinges on an appropriate choice of norm. We define 
\begin{equation}\label{robust-norm}
\|\vsf\|_{\textsf{R}}:=\norm{\vsf}_{\Omega}+\|\RR{\bsf\cdot\nabla}\vsf\|_{*}\quad\forall\vsf\in H_{0}^{1}(\Omega),
\end{equation}
where $\norm{\vsf}_{\Omega}$ is defined as in \eqref{eq:norm}, and 
\begin{equation}\label{dual-norm}
\|\RR{\bsf\cdot\nabla}\vsf\|_{*} = \sup_{\phi\in H_{0}^{1}(\Omega)\setminus\{0\}}\frac{\langle \RR{\bsf\cdot\nabla}\vsf,\phi\rangle}{\norm{\phi}_{\Omega}}.
\end{equation}
The term $\langle \cdot, \cdot \rangle$ denotes the duality pairing between $H_0^1(\Omega)$ and $H^{-1}(\Omega)$.
In this \RR{setting, we} have that
\begin{equation}\label{inf_sup_robust}
\sup_{\vsf\in H_{0}^{1}(\Omega)\setminus\{0\}}\frac{\mathcal{B}(\wsf,\vsf)}{\norm{\vsf}_{\Omega}}
\geq \frac{1}{3}\|\wsf\|_{\RR{\textsf{R}}}\quad\forall\wsf\in H_{0}^{1}(\Omega).
\end{equation}
We refer the reader to \cite{MR2182149}[Lemma 3.1] and \cite{MR3059294}[Proposition 4.17] for \RR{more} details. \RR{We also have that
\begin{equation}\label{inf_sup_robust*}
\sup_{\vsf\in H_{0}^{1}(\Omega)\setminus\{0\}}\frac{\mathcal{B}^*(\wsf,\vsf)}{\norm{\vsf}_{\Omega}}
\geq \frac{1}{3}\|\wsf\|_{\RR{\textsf{R}}}\quad\forall\wsf\in H_{0}^{1}(\Omega).
\end{equation}
Moreover, we note that
\begin{equation}\label{L2Rbound}
\| \vsf\|_{L^{2}(\Omega)}^{2} \leq \tfrac{1}{\kappa}\norm{\vsf}_{\Omega}^{2} \leq \|\vsf\|_{\textsf{R}}^{2}\quad\forall\vsf\in H_{0}^{1}(\Omega).
\end{equation}
}

We follow \cite[section 2]{MR3407239} and define, for $\varrho = st$ or $\varrho = ad$, the error estimator
\begin{equation}
\label{eq:estimator_robust_st_ad}
\mathcal{E}_{\varrho}^{2}:=\sum_{K\in\T}\mathcal{E}_{\varrho,K}^{2},
\end{equation}
where the local error indicators are given by
\[
\mathcal{E}_{\varrho,K}^{2}:=\hslash_{K}^{2}\|\mathcal{R}_{K}^{\varrho}\|_{L^{2}(K)}^{2} + 
\sum_{\gamma\in\mathcal{F}_{K}\cap\mathcal{F}_{I}}\nu^{-1/2}\hslash_{\gamma}\|\llbracket J_{\gamma}^{\varrho}\rrbracket\|_{L^{2}(\gamma)}^{2}.
\]
Here, $\hslash_{\omega} = \min\left\{h_{\omega} \varepsilon^{-1/2},\kappa^{-1/2}\right\}$ for $\omega=K$ or $\omega=\gamma$, $\mathcal{R}_{K}^{\varrho}$ is defined as in \eqref{residuals_osc_state} or \eqref{residuals_osc_adjoint}, and the jump term $\llbracket J_{\gamma}^{\varrho}\rrbracket$ is defined as in section \ref{subsec:efficiency}. We also define, again for $\varrho = st$ or $\varrho = ad$, the global oscillation term
\begin{equation}
\label{eq:oscillation_global_robust}
  \|\mathrm{osc}^{\varrho} \|_{L^{2}(\Omega)} = \left(\sum_{K \in \T} \RR{\hslash_{K}^{2}} \|\mathrm{osc}_{K}^{\varrho} \|_{L^{2}(K)}^2 \right)^{\frac{1}{2}},
\end{equation}
where $\mathrm{osc}_{K}^{\varrho}$ is defined as in \eqref{residuals_osc_state} or \eqref{residuals_osc_adjoint}.

\EO{As stabilization techniques we will consider the SUPG and CIP methods; see \eqref{eq:SUPG}\RR{, \eqref{eq:CIP}, \eqref{eq:SUPGad} and \eqref{eq:CIPad}}. We assume that their stabilization parameters are such that (\cite{MR3407239}[ineq. (2.7) and (2.10)]):
\RR{\begin{equation}
\label{eq:assump_stab_parametersSUPG}
\| \bsf \|_{L^{\infty}(K)} \tau_K \leq C h_K \mbox{ and } \| \bsf \|_{L^{\infty}(K)} \tau_K^* \leq C h_K \mbox{ } \forall K \in \T
\end{equation}
and
\begin{equation}
\label{eq:assump_stab_parametersCIP}
\tau_{\gamma} \leq C h_{\gamma}^2 \mbox{ and } \tau_{\gamma}^* \leq C h_{\gamma}^2 \mbox{ } \forall \gamma \in \mathcal{F}_{I},
\end{equation}}
where $C$ is independent of the \RR{size of the elements in the mesh} and the physical \RR{parameters}. We also define $\displaystyle\Theta(\xi)^{2}=\sum_{K\in\T}\Theta_{K}(\xi)^{2}$, where
\begin{equation*}
 \Theta_{K}(\xi)=
 \left\{\begin{array}{ll}
 \hslash_{K}\|(\bsf-\Pi_{K}(\bsf))\cdot\nabla \xi\|_{\boldsymbol{L}^{2}(K)} +
\hslash_{K}h_{K}\| \nabla \bsf\|_{\boldsymbol{L}^{\infty}(K)}\|\nabla\xi\|_{\boldsymbol{L}^{2}(K)}, & \mbox{CIP},
\\
0, & \mbox{SUPG}.
\end{array}\right.
\end{equation*}

The \RR{assumptions \eqref{eq:assump_stab_parametersSUPG} and \eqref{eq:assump_stab_parametersCIP}} on the stabilization parameters guarantee that the so--called consistency error can be bounded robustly by the residual estimator \eqref{eq:estimator_robust_st_ad} and the oscillation term \eqref{eq:oscillation_global_robust}; see \cite{MR3407239}[Lemma 2.3 and Lemma 2.6]. This is a key result in the derivation of the main result of} \cite[Theorem 2.8]{MR3407239} (see also \cite[Theorem 4.1]{MR2182149}). Let $\hat \ysf$ be the solution to problem \eqref{eq:y_hat} and $\bar \ysf_{\T}$ be the solution to the first variational equation of the discrete optimality system \eqref{optimal_system_discrete_1}. If \RR{\eqref{eq:assump_stab_parametersSUPG} and \eqref{eq:assump_stab_parametersCIP} hold}, then
\begin{equation}\label{eq:robust_st}
\tfrac{1}{\RR{D_{\ysf}}} \mathcal{E}_{\textsf{st}} - \|\mathrm{osc}^{st} \|_{L^{2}(\Omega)}   \leq \|\hat{\ysf}-\bar{\ysf}_{\T}\|_{\textsf{R}} \leq \RR{C_{\ysf}}\left( \mathcal{E}_{\textsf{st}}^2 +   \|\mathrm{osc}^{st}\|^2_{L^{2}(\Omega)} +
\EO{\Theta(\bar{\ysf}_{\T})^{2}} \right)^{\frac{1}{2}}.
\end{equation}
We have a similar result for the numerical approximation of the adjoint equation. Let $\hat \psf$ be the solution to problem \eqref{eq:p_hat} and \RR{$\bar \psf_{\T}$} be the solution to the second variational equation of the discrete optimality system \eqref{optimal_system_discrete_1}. If \RR{\eqref{eq:assump_stab_parametersSUPG} and \eqref{eq:assump_stab_parametersCIP} hold}, then
\begin{equation}\label{eq:robust_ad}
\tfrac{1}{\RR{D_{\psf}}} \mathcal{E}_{\textsf{ad}} -   \|\mathrm{osc}^{ad}\|_{L^{2}(\Omega)} \leq \|\hat{\psf}-\bar{\psf}_{\T}\|_{\textsf{R}} \leq \RR{C_{\psf}} \left( \mathcal{E}^2_{\textsf{ad}} +   \|\mathrm{osc}^{ad}\|^2_{L^{2}(\Omega)} +
\EO{\Theta(\bar{\psf}_{\T})^{2}}\right)^{\frac{1}{2}}.
\end{equation}
\RR{We immediately comment that \RR{the estimates \eqref{eq:robust_st} and \eqref{eq:robust_ad}} are robust in the sense that the constants $C_{\ysf}$, $C_{\psf}$, $D_{\ysf}$ and $D_{\psf}$ are independent of $\nu$ and $\mathbf{b}$.}


\subsection{The optimal control problem}

We now derive a posteriori error estimates for the discretization \RR{of the optimal control problem \eqref{min}--\eqref{control_constraint} proposed in section 
\ref{control_optimo}}. \RR{We now present a modification of the analysis elaborated in the proof of Theorem \ref{th:global_reliability} in order to obtain an estimator, for the error $\|\bar{\ysf} - \bar{\ysf}_{\T}\|_{\textsf{R}}^2 
 + \|\bar{\psf} - \bar{\psf}_{\T}\|_{\textsf{R}}^2
 + \| \bar{\usf} - \bar{\usf}_{\T}\|_{L^2(\Omega)}^2$, that is robust with repect to $\nu$ and $\mathbf{b}$ in the sense that the constants involved in the upper and lower bounds for the error are independent of $\nu$ and $\mathbf{b}$.}

\begin{theorem}[global reliability]
Let $(\bar{\ysf},\bar{\psf},\bar{\usf}) \in H_0^1(\Omega) \times H_0^1(\Omega) \times L^2(\Omega)$ be the solution to \eqref{optimal_system_1} and $(\bar{\ysf}_{\T},\bar{\psf}_{\T},\bar{\usf}_{\T}) \in \V(\T) \times \V(\T) \times \U_{\textrm{ad}}(\T)$ be its numerical approximation obtained as the solution to \eqref{optimal_system_discrete_1}\RR{. If the stabilization parameters are such that \RR{\eqref{eq:assump_stab_parametersSUPG} and \eqref{eq:assump_stab_parametersCIP} hold}, then
\begin{equation}
 \label{eq:robust_reliability}
 \|\bar{\ysf} - \bar{\ysf}_{\T}\|_{\textsf{R}}^2 
 + \|\bar{\psf} - \bar{\psf}_{\T}\|_{\textsf{R}}^2
 + \| \bar{\usf} - \bar{\usf}_{\T}\|_{L^2(\Omega)}^2
 \leq 
 \textsc{C}_1\Upsilon_{\textsf{R}}^2
 \end{equation}
 where
\begin{equation}\label{eq:robust_estimator}
\Upsilon_{\textsf{R}}^2=
\mathcal{E}_{\textsf{st}}^2 +\mathcal{E}^2_{\textsf{ad}}+\eta_{ct}^2
+\|\mathrm{osc}^{st}\|^2_{L^{2}(\Omega)}+\|\mathrm{osc}^{ad}\|^2_{L^{2}(\Omega)}
+\Theta(\bar{\ysf}_{\T})^{2}+\Theta(\bar{\psf}_{\T})^{2}
\end{equation}
and the positive constant $\textsc{C}_1$ is independent of the size of the elements in the mesh, $\nu$ and $\mathbf{b}$.}
\end{theorem}

\begin{proof}
\noindent \framebox{Step 1.} We begin the proof by recalling the estimate \eqref{u-uh-utilde}:
$$
\|  \bar{\usf} -\bar{\usf}_{\T} \|_{L^2(\Omega)}^{2} 
\leq 
2\|  \bar{\usf} -\tilde{\usf} \|_{L^2(\Omega)}^{2} + 2\eta_{ct}^{2}, \quad \tilde{\usf} = \Pi (-\tfrac{1}{\vartheta}\bar{\psf}_{\T}),
$$
where $\eta_{ct}$ is defined as in \eqref{eq:indicator_control}. We control the first term on the right hand side of the previous expression \RR{using the intermediate step in} \eqref{eq:usf-tildeusf_3}. We thus have that
\begin{align*}
\| \bar{\usf} - \tilde{\usf} \|^2_{L^2(\Omega)} & \leq 
\tfrac{2}{\vartheta^2} \| \tilde{\psf}  - \hat{\psf}\|_{L^{2}(\Omega)}^{2} + 
\tfrac{2}{\vartheta^2\kappa}\| \hat{\psf} - \bar{\psf}_{\T}\|_{\textsf{R}}^{2} 
\\
& \leq
\tfrac{2}{\vartheta^2}\| \tilde{\psf}  - \hat{\psf}\|_{L^{2}(\Omega)}^{2} + 
\tfrac{2 \RR{C_{\psf}^{2}}}{\vartheta^2\kappa}\left( \mathcal{E}^2_{\textsf{ad}} +   \|\mathrm{osc}^{ad}\|^2_{L^{2}(\Omega)} \RR{+\Theta(\bar{\psf}_{\T})^{2}}\right),
\end{align*}
\RR{upon using \eqref{L2Rbound} and} \eqref{eq:robust_ad}. We now proceed to control the term  $\| \tilde{\psf}  - \hat{\psf}\|_{L^{2}(\Omega)}$. To accomplish this task, we invoke \RR{\eqref{L2Rbound} and \eqref{inf_sup_robust*} to} obtain that
 \begin{align*}
\|\tilde{\psf}  - \hat{\psf}\|_{L^{2}(\Omega)}
 \leq \frac{1}{\sqrt{\kappa}}
\|\tilde{\psf}  - \hat{\psf}\|_{\textsf{R}} & 
\leq \frac{3}{\sqrt{\kappa}} \sup_{\vsf\in H_{0}^{1}(\Omega)\setminus\{0\}}
\frac{\mathcal{B}^{*}(\tilde{\psf}  - \hat{\psf},\vsf)}{\norm{\vsf}_{\Omega}}.
\end{align*}
We now utilize the problem that $\tilde{\psf}  - \hat{\psf}$ solves \RR{and \eqref{L2Rbound}} to arrive at
\begin{equation}
\label{eq:paso_vaca}
\|\tilde{\psf}  - \hat{\psf}\|_{L^{2}(\Omega)} \leq \frac{3}{\sqrt{\kappa}} \sup_{\vsf\in H_{0}^{1}(\Omega)\setminus\{0\}}
\frac{(\tilde{\ysf}-\bar{\ysf}_{\T},\vsf)_{L^{2}(\Omega)}}{\norm{\vsf}_{\Omega}} 
\leq \frac{3}{\kappa^{3/2}}\RR{\|\tilde{\ysf}-\bar{\ysf}_{\T}\|_{\textsf{R}}}.
\end{equation}
To control the right hand side of the previous expression we invoke the triangle inequality and the a posteriori error estimate \eqref{eq:robust_st}. This yields \RR{that}
\begin{align*}
\|\tilde{\psf}  - \hat{\psf}\|_{L^{2}(\Omega)}^{2} 
 \leq  \tfrac{9}{\kappa^3} \|\tilde{\ysf}-\bar{\ysf}_{\T}\|_{\textsf{R}}^2 \leq \tfrac{18}{\kappa^3} \|\tilde{\ysf}-\hat{\ysf}\|_{\textsf{R}}^2 + \tfrac{18\RR{C_{\ysf}^2}}{\kappa^3}\left( \mathcal{E}_{\textsf{st}}^2 +   \|\mathrm{osc}^{st}\|^2_{L^{2}(\Omega)} \RR{+\Theta(\bar{\ysf}_{\T})^{2}}\right).
\end{align*}
The term $  \|\tilde{\ysf}-\hat{\ysf}\|_{\textsf{R}} $ is controlled using similar arguments to the ones that lead to \eqref{eq:paso_vaca}:
\begin{align*}
\|\tilde{\ysf}  - \hat{\ysf}\|_{\textsf{R}}  
\leq 3
\sup_{\vsf\in H_{0}^{1}(\Omega)\setminus\{0\}}
\frac{\mathcal{B}(\tilde{\ysf}  - \hat{\ysf},\vsf)}{\norm{\vsf}_{\Omega}} 
\leq \frac{3}{\sqrt{\kappa}}
\|\tilde{\usf}-\bar{\usf}_{\T}\|_{L^{2}(\Omega)} \leq \frac{3}{\sqrt{\kappa}} \eta_{ct},
\end{align*}
where $\eta_{ct}$ is the \RR{estimator} asocciated with the control variable that is defined in \eqref{eq:indicator_control}. By gathering our previous findings, we conclude that
\begin{equation}
\label{eq:robust_control}
\begin{split}
\|\bar{\usf}-\bar{\usf}_{\T}\|_{L^{2}(\Omega)}^{2}
\leq 
&\left(2+\tfrac{648}{\vartheta^{2}\kappa^{4}}\right)\eta_{ct}^{2} +
\tfrac{4\RR{C_{\psf}^{2}}}{\vartheta^{2}\kappa}\left( \mathcal{E}_{ad}^{2} +\| \mathrm{osc}^{ad} \|^2_{L^2(\Omega)} \RR{+\Theta(\bar{\psf}_{\T})^{2}}\right)
\\
 &+ \tfrac{72\RR{C_{\ysf}^{2}}}{\vartheta^{2} \kappa^3} \left( \mathcal{E}_{\textsf{st}}^{2} + \| \mathrm{osc}^{st} \|^2_{L^2(\Omega)} \RR{+\Theta(\bar{\ysf}_{\T})^{2}}\right).
\end{split}
\end{equation}
~\\
\noindent \framebox{Step 2.}  The goal of this step is to control the error $\| \bar{\ysf}-\bar{\ysf}_{\T}\|_{\textsf{R}}$. In fact, using \RR{standard inequalities} and \eqref{eq:robust_st}, we arrive at 
\[
\|\bar{\ysf}-\bar{\ysf}_{\T}\|_{\textsf{R}}^{2}
\leq 
2\|\bar{\ysf}-\hat{\ysf}\|_{\textsf{R}}^{2}+2\RR{C_{\ysf}^{2}}\left( \mathcal{E}_{\textsf{st}}^2 +   \|\mathrm{osc}^{st}\|^2_{L^{2}(\Omega)} \RR{+\Theta(\bar{\ysf}_{\T})^{2}}\right).
\]
\RR{By using} \eqref{inf_sup_robust} and the problem that $\bar{\ysf}-\hat{\ysf}$ solves, we obtain that
\begin{align*}
\|\bar{\ysf}-\hat{\ysf}\|_{\textsf{R}}
 \leq 3
\sup_{\vsf\in H_{0}^{1}(\Omega)\setminus\{0\}}
\frac{\mathcal{B}(\bar{\ysf}-\hat{\ysf},\vsf)}{\norm{\vsf}_{\Omega}} 
& \leq  \frac{3}{\sqrt{\kappa}}
\|\bar{\usf}-\bar{\usf}_{\T}\|_{L^{2}(\Omega)}.
\end{align*}
Therefore, invoking \eqref{eq:robust_control} we arrive at
\begin{equation}
\label{eq:robust_state}
\begin{split}
\|\bar{\ysf}-\bar{\ysf}_{\T}\|_{\textsf{R}}^{2}  \leq &
\frac{18}{\kappa}
\left(2+\tfrac{648}{\vartheta^{2}\kappa^{4}}\right)\eta_{ct}^{2} +
\tfrac{72\RR{C_{\psf}^{2}}}{\vartheta^{2}\kappa^{2}}\left( \mathcal{E}_{ad}^{2} +\| \mathrm{osc}^{ad} \|^2_{L^2(\Omega)} \RR{+\Theta(\bar{\psf}_{\T})^{2}}\right)
\\ 
& + \left(2+\tfrac{1296}{\vartheta^{2} \kappa^4}\right) \RR{C_{\ysf}^{2}} \left( \mathcal{E}_{\textsf{st}}^{2} + \| \mathrm{osc}^{st} \|^2_{L^2(\Omega)} \RR{+\Theta(\bar{\ysf}_{\T})^{2}}\right).
\end{split}
\end{equation}
~\\
\noindent \framebox{Step 3.} In this step we \RR{bound $\| \bar{\psf}-\bar{\psf}_{\T} \|_{\textsf{R}}$}. We invoke similar arguments to the ones elaborated in step 2 and conclude that 
\[
\|\bar{\psf}-\bar{\psf}_{\T}\|_{\textsf{R}}^{2}
\leq 2
(\|\bar{\psf}-\hat{\psf}\|_{\textsf{R}}^{2}+\|\hat{\psf}-\bar{\psf}_{\T}\|_{\textsf{R}}^{2})
\leq 
2\|\bar{\psf}-\hat{\psf}\|_{\textsf{R}}^{2}+2\RR{C_{\psf}^{2}\left( \mathcal{E}_{ad}^{2} +\| \mathrm{osc}^{ad} \|^2_{L^2(\Omega)} +\Theta(\bar{\psf}_{\T})^{2}\right)}
\]
and that
\begin{align*}
\|\bar{\psf}-\hat{\psf}\|_{\textsf{R}}
 \leq 3
\sup_{\vsf\in H_{0}^{1}(\Omega)\setminus\{0\}}
\frac{\mathcal{B}^{*}(\bar{\psf}-\hat{\psf},\vsf)}{\norm{\vsf}_{\Omega}} 
\leq  \frac{3}{\kappa} \|\bar{\ysf}-\bar{\ysf}_{\T}\|_{\textsf{R}}.
\end{align*}
We thus use these estimates, \RR{and \eqref{eq:robust_state}, to} arrive at
\begin{equation}
\label{eq:robust_adjoint}
\begin{split}
\|\bar{\psf}-\bar{\psf}_{\T}\|_{\textsf{R}}^{2} \leq & 
\RR{\tfrac{324}{\kappa^3} 
\left(2+\tfrac{648}{\vartheta^{2}\kappa^{4}}\right)}\eta_{ct}^{2}
\\
&+\RR{\left(2 + \tfrac{1296}{\vartheta^{2}\kappa^{4}}\right) C_{\psf}^{2}} \left( \mathcal{E}_{ad}^{2} +\| \mathrm{osc}^{ad} \|^2_{L^2(\Omega)} \RR{+\Theta(\bar{\psf}_{\T})^{2}}\right)
\\ 
&+ \tfrac{18}{\kappa^\RR{2}}\left(2+\tfrac{1296}{\vartheta^{2} \kappa^4}\right) \RR{C_{\ysf}^{2}} \left( \mathcal{E}_{\textsf{st}}^{2} + \| \mathrm{osc}^{st} \|^2_{L^2(\Omega)} \RR{+\Theta(\bar{\ysf}_{\T})^{2}}\right) .
\end{split}
\end{equation}
~\\
\noindent \framebox{Step 4.} Finally, \RR{combining} the estimates \eqref{eq:robust_control}, \eqref{eq:robust_state}, and \eqref{eq:robust_adjoint} allows us to arrive at \RR{\eqref{eq:robust_reliability}.}
\end{proof}

\RR{We now provide an efficiency analysis.

\begin{theorem}[global efficiency]
Let $(\bar{\ysf},\bar{\psf},\bar{\usf}) \in H_0^1(\Omega) \times H_0^1(\Omega) \times L^2(\Omega)$ be the solution to \eqref{optimal_system_1} and $(\bar{\ysf}_{\T},\bar{\psf}_{\T},\bar{\usf}_{\T}) \in \V(\T) \times \V(\T) \times \U_{\textrm{ad}}(\T)$ be its numerical approximation obtained as the solution to \eqref{optimal_system_discrete_1}. If the stabilization parameters are such that \RR{\eqref{eq:assump_stab_parametersSUPG} and \eqref{eq:assump_stab_parametersCIP} hold}, then
\begin{equation}
 \label{eq:robust_efficiency}
 \begin{split}
 \mathcal{E}_{\textsf{st}}^2 +\mathcal{E}^2_{\textsf{ad}}+\eta_{ct}^2
 \leq &
 \textsc{C}_2\Big(
 \|\bar{\ysf} - \bar{\ysf}_{\T}\|_{\textsf{R}}^2 
 + \|\bar{\psf} - \bar{\psf}_{\T}\|_{\textsf{R}}^2
 + \| \bar{\usf} - \bar{\usf}_{\T}\|_{L^2(\Omega)}^2
\\
&+\|\mathrm{osc}^{st}\|^2_{L^{2}(\Omega)}+\|\mathrm{osc}^{ad}\|^2_{L^{2}(\Omega)}
\Big)
\end{split}
 \end{equation}
where the positive constant $\textsc{C}_2$ is independent of the size of the elements in the mesh, $\nu$ and $\mathbf{b}$.
\end{theorem}

\begin{proof}
From \eqref{eq:robust_st} and \eqref{eq:robust_ad} we have that
\[
\mathcal{E}_{\textsf{st}}
\leq 
D_{\ysf}\left(\|\hat{\ysf}-\bar{\ysf}_{\T}\|_{\textsf{R}}+\|\mathrm{osc}^{st} \|_{L^{2}(\Omega)}\right)
\leq
D_{\ysf}\left(\|\hat{\ysf}-\bar{\ysf}\|_{\textsf{R}}+\|\bar{\ysf}-\bar{\ysf}_{\T}\|_{\textsf{R}}+\|\mathrm{osc}^{st} \|_{L^{2}(\Omega)}\right)
\]
and
\[
\mathcal{E}_{\textsf{ad}}
\leq 
D_{\psf}\left(\|\hat{\psf}-\bar{\psf}_{\T}\|_{\textsf{R}}+\|\mathrm{osc}^{ad} \|_{L^{2}(\Omega)}\right)
\leq
D_{\psf}\left(\|\hat{\psf}-\bar{\psf}\|_{\textsf{R}}+\|\bar{\psf}-\bar{\psf}_{\T}\|_{\textsf{R}}+\|\mathrm{osc}^{ad} \|_{L^{2}(\Omega)}\right).
\]
Moreover, \eqref{inf_sup_robust} and \eqref{inf_sup_robust*} imply that
\[
\|\hat{\ysf}-\bar{\ysf}\|_{\textsf{R}}\le 3\sup_{\vsf\in H_{0}^{1}(\Omega)\setminus\{0\}}\frac{\mathcal{B}(\hat{\ysf}-\bar{\ysf},\vsf)}{\norm{\vsf}_{\Omega}}\mbox{ and }
\|\hat{\psf}-\bar{\psf}\|_{\textsf{R}}\le 3\sup_{\vsf\in H_{0}^{1}(\Omega)\setminus\{0\}}\frac{\mathcal{B}^*(\hat{\psf}-\bar{\psf},\vsf)}{\norm{\vsf}_{\Omega}}.
\]
Now, \eqref{optimal_system_1} and \eqref{eq:y_hat} yield that
\[
\mathcal{B}(\hat{\ysf}-\bar{\ysf},\vsf)=(\bar{\usf}_\T-\bar{\usf},\vsf)_{L^2(\Omega)}
\le \|  \bar{\usf} -\bar{\usf}_{\T} \|_{L^2(\Omega)}\| \vsf \|_{L^2(\Omega)}
\le \frac{1}{\sqrt{\kappa}}\|  \bar{\usf} -\bar{\usf}_{\T} \|_{L^2(\Omega)}\norm{\vsf}_{\Omega}
\]
and \eqref{optimal_system_1} and \eqref{eq:p_hat} yield that
\[
\mathcal{B}^*(\hat{\psf}-\bar{\psf},\vsf)=(\bar{\ysf}_\T-\bar{\ysf},\vsf)_{L^2(\Omega)}
\le \|  \bar{\ysf} -\bar{\ysf}_{\T} \|_{L^2(\Omega)}\| \vsf \|_{L^2(\Omega)}
\le \frac{1}{\kappa}\|  \bar{\ysf} -\bar{\ysf}_{\T} \|_{\textsf{R}}\norm{\vsf}_{\Omega}.
\]
Hence
\begin{equation}\label{eq:stRE}
\mathcal{E}_{\textsf{st}}^2
\leq
3D_{\ysf}\left(\frac{9}{\kappa}\|  \bar{\usf} -\bar{\usf}_{\T} \|_{L^2(\Omega)}^2+\|\bar{\ysf}-\bar{\ysf}_{\T}\|_{\textsf{R}}^2+\|\mathrm{osc}^{st} \|_{L^{2}(\Omega)}^2\right)
\end{equation}
and
\begin{equation}\label{eq:adRE}
\mathcal{E}_{\textsf{ad}}^2
\leq
3D_{\psf}\left(\frac{9}{\kappa^2}\|  \bar{\ysf} -\bar{\ysf}_{\T} \|_{\textsf{R}}^2+\|\bar{\psf}-\bar{\psf}_{\T}\|_{\textsf{R}}^2+\|\mathrm{osc}^{ad} \|_{L^{2}(\Omega)}^2\right).
\end{equation}
In addition, \eqref{ct_efficiency} leads to
\begin{equation}\label{eq:ctRE}
\eta_{ct}^2\leq 
2\left(\left\| \bar{\usf}- \bar{\usf}_{\T}  \right\|_{L^2(\Omega)}^2 + 
\frac{2}{\vartheta^2\kappa}
\norm{\bar{\psf}-\bar{\psf}_{\T}}_{\textsf{R}}^2\right).
\end{equation}
The claimed result then follows upon combining \eqref{eq:stRE}, \eqref{eq:adRE} and \eqref{eq:ctRE}.
\end{proof}

\begin{remark}[robusteness]
We remark that the a posteriori error estimator $\Upsilon_{\textsf{R}}$ is robust in the sense that the constants that appear in \eqref{eq:robust_reliability} and \eqref{eq:robust_efficiency} are independent of $\nu$ and $\mathbf{b}$. The estimator is not robust with respect to $\kappa$ or $\vartheta$. The dependence of the constant in \eqref{eq:robust_reliability} on $\kappa$ and $\vartheta$ can be seen from \eqref{eq:robust_control}, \eqref{eq:robust_state}, and \eqref{eq:robust_adjoint}. Similarly, the dependence of the constant in \eqref{eq:robust_efficiency} on $\kappa$ and $\vartheta$ can be seen from \eqref{eq:stRE}, \eqref{eq:adRE} and \eqref{eq:ctRE}.
\end{remark}
}
}

\section{Numerical examples}\label{numerical}
In this section we \RR{show numerical} examples that illustrate the performance of the error estimator. \RR{We wrote a code in \texttt{C++} that implements the procedure described in \textbf{Algorithm 1}}. \RR{The integrals involving the data $\ysf_\Omega$ and $\fsf$ were computed using quadrature formulas which are exact for polynomials of degree $\mathsf{N}$. We show results for $\mathsf{N}\in\{4,19\}$ for $d=2$ and $\mathsf{N}\in\{4,14\}$ for $d=3$.  The error $\|(e_{\bar{\ysf}},e_{\bar{\psf}},e_{\bar{\usf}})\|_{\Omega}$ was computed using a quadrature formula which is exact for polynomials of degree $19$ for $d=2$ and $14$ for $d=3$.} \RR{The global} linear systems were solved using the multifrontal massively parallel sparse direct solver (MUMPS) \cite{MR1856597,MR2202663}. \RR{In order to construct exact solutions we fix the optimal state and adjoint state, and compute the optimal control and data $\ysf_\Omega$ and $\fsf$ using \eqref{projection_formula} and \eqref{optimal_system_1}.}

\begin{table}[!htbp]
\begin{flushleft}
\scalebox{0.7}
{
\begin{tabular}{l l} 
\multicolumn{2}{l}{\textbf{Algorithm 1:  Adaptive Primal-Dual Active Set Algorithm.}} 
\vspace{0.15cm}\\
\toprule
\multicolumn{2}{l}{\textbf{Input:} A mesh $\T$ and data $\lambda$, $\asf$, $\bbsf$, $\nu$, $\bsf$, $\kappa$, $\ysf_{\Omega}$ and $\fsf$.}
\\
\textbf{1:}    &  Compute $(\bar{\ysf}_\T,\bar{\psf}_\T,\bar{\usf}_\T)\in \V(\T)\times \V(\T)\times \mathbb{U}_{ad}(\T)$ that solves \eqref{optimal_system_discrete_1} using the active set strategy of \cite[\S 2.12.4]{MR2583281}.
\\
\textbf{2:}    &  Compute the local error indicator $\Upsilon_K$ given in \eqref{eq:indicatorOC} for each $K \in \T$ and the error estimator $\Upsilon$ given in \eqref{eq:reliability}.
 \\
\textbf{3:}    & Mark an element $K \in \T$ for refinement if $\Upsilon_K^{2}\ge\Upsilon^{2}/\#\T$.
\\
\textbf{4:}    & Refine the mesh $\T$ using a longest edge bisection algorithm and return to step \textbf{1}.
\\
\bottomrule
\end{tabular}}
\vspace{-0.3cm}
\end{flushleft}
\end{table}

\textbf{\RR{Example 1}:} We set $d=2$, $\asf=-1$, $\bbsf=-0.1$, $\lambda=1$, $\nu=10^{-3}$, $\bsf=(1,0)$ and $\kappa=1$. The exact optimal state and adjoint are given by, taking $\varsigma:=x_{2}(1-x_{2})$,
\begin{gather*}
\bar{\ysf}(x_{1},x_{2})=\varsigma\left(x_{1}+\frac{e^{\tfrac{x_{1}-1}{\nu}}-e^{\tfrac{-1}{\nu}}}{e^{\tfrac{-1}{\nu}}-1}\right),~
\bar{\psf}(x_{1},x_{2})=\varsigma\left(1-x_{1}+\frac{e^{\tfrac{-x_{1}}{\nu}}-e^{\tfrac{-1}{\nu}}}{e^{\tfrac{-1}{\nu}}-1}\right).
\end{gather*}  
\textbf{\RR{Example 2}:} We set $d=3$, $\asf=-0.01$, $\bbsf=0.01$, $\lambda=1$, $\nu=0.01$, $\bsf=(3,2,1)$ and $\kappa=10$. The exact optimal state and adjoint are given by
\begin{gather*}
\bar{\ysf}(x_{1},x_{2},x_{3})=\prod_{i=1}^{3}x_{i}(1-x_{i}),\quad
\bar{\psf}(x_{1},x_{2},y_{3})=\bar{\ysf}(x_{1},x_{2},y_{3})\textrm{tan}^{-1}\left(\frac{x_{1}-0.5}{\nu}\right).
\end{gather*}
In Figures \RR{\ref{FigE1} and \ref{FigE2}}, we present the performance of the adaptive procedure by \RR{showing the} total error $\|(e_{\bar{\ysf}},e_{\bar{\psf}},e_{\bar{\usf}})\|_{\Omega}$ and error estimator $\Upsilon$, as well as effectivity indices $\Upsilon/\|(e_{\bar{\ysf}},e_{\bar{\psf}},e_{\bar{\usf}})\|_{\Omega}$, for different combinations of stabilizations for state--adjoint equations. We use the following notation: SUPG--GLS corresponds to using a SUPG stabilization for the state equation and a GLS stabilization for the adjoint equation; SUPG--SUPG, SUPG--CIP and SUPG--ES are defined analogously. We took the stabilization parameters $\tau^{*}_{K}=\tau_{K}$ where
$$
\tau_{K}=\left\{
\begin{array}{cl}
\tfrac{h_{K}}{2\|\bsf\|_{\boldsymbol{L}^{\infty}(K)}} & \textrm{if}~Pe_{K}>1,\\
\tfrac{h_{K}^{2}}{12\nu} & \textrm{if}~Pe_{K}\leq 1,
\end{array}
\right.
\quad\textrm{with}\quad
Pe_{K}:=\frac{\|\bsf\|_{\boldsymbol{L}^{\infty}(K)}h_{K}}{2\nu},
$$
for SUPG and GLS, $\tau^{*}_{\gamma}=1/24$ for ES and $\tau^{*}_{\gamma}=h_\gamma^2/12$ for CIP. \RR{The total number of degrees of freedom $\mathsf{Ndof} = 2 \dim(\V(\T)) + \#\T$. In Figures \ref{FigE1} and \ref{FigE2}, we observe that, once the mesh has been sufficiently refined, the experimental rates of convergence for the error are optimal. Computationally, we observe that the estimator is never less than the error, the effectivity index never goes below $1$; on the final meshes takes the numerical value stated in the plots. In Figures \ref{Ex1Meshes} and \ref{Ex2Meshes} we observe that the refinement is being concentrated around the boundary and interior layers, even though different values of $\mathsf{N}$ resulted in different adaptively refined meshes.}

\psfrag{example N - total error concha tu madre}{\Huge Example 1 - $\|(e_{\bar{\ysf}},e_{\bar{\psf}},e_{\bar{\usf}})\|_{\Omega}$}
\psfrag{example N - total estimator concha tu madre}{\Huge Example 1 - Estimator $\Upsilon$}
\psfrag{example N - effectivity index concha tu madre}{\Huge Example 1 - $\Upsilon/\|(e_{\bar{\ysf}},e_{\bar{\psf}},e_{\bar{\usf}})\|_{\Omega}$}
\psfrag{SUPG - ES}{\Large SUPG-ES}
\psfrag{SUPG - CIP}{\Large SUPG-CIP}
\psfrag{SUPG - GLS}{\Large SUPG-GLS}
\psfrag{SUPG - SUPG}{\Large SUPG-SUPG}
\psfrag{optimal rate}{\Large Optimal rate}
\psfrag{Ndofs}{\Large Ndof}
\begin{figure}[!htbp]
\begin{center}
\scalebox{0.3}{\includegraphics{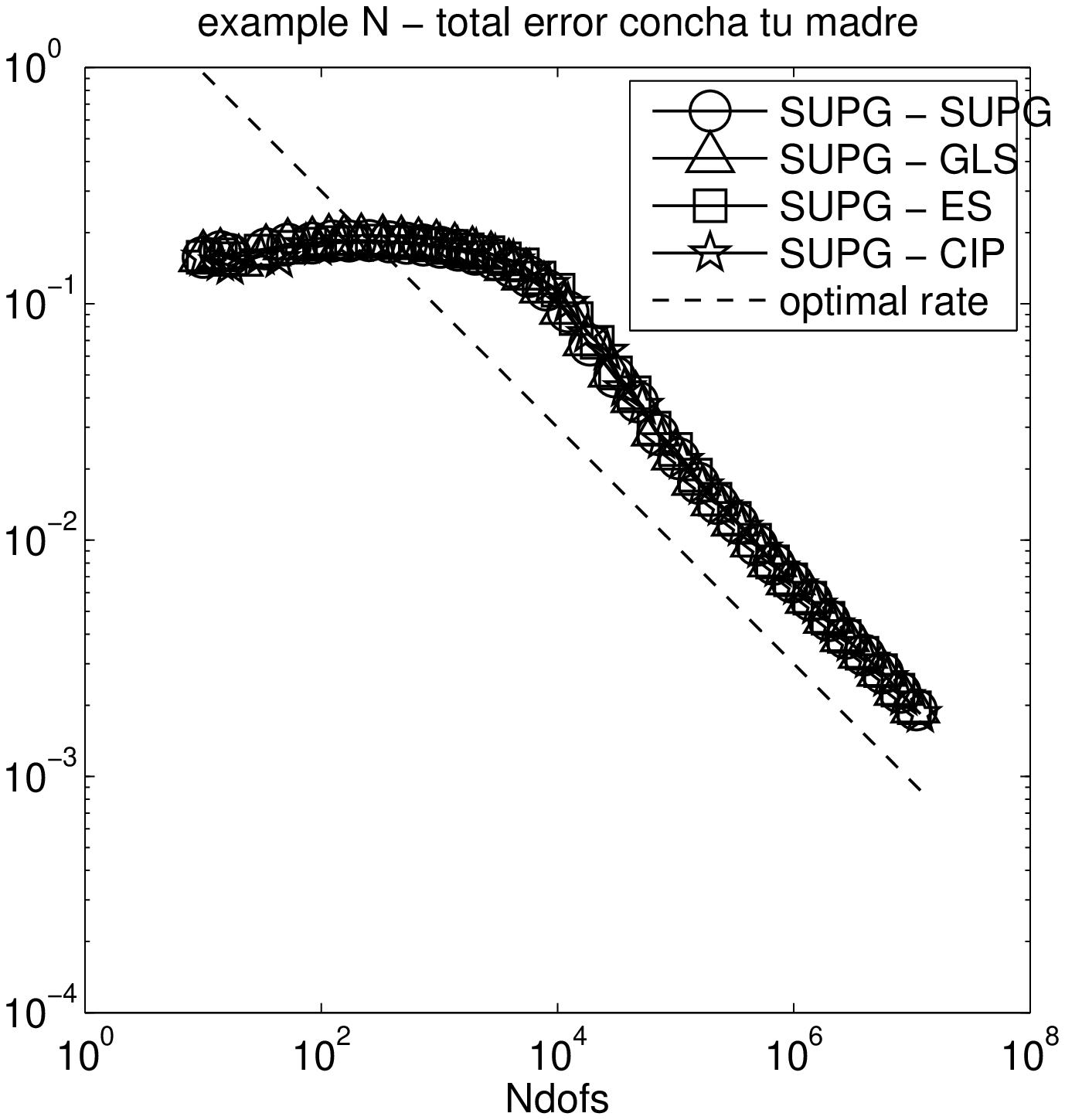}}
\scalebox{0.3}{\includegraphics{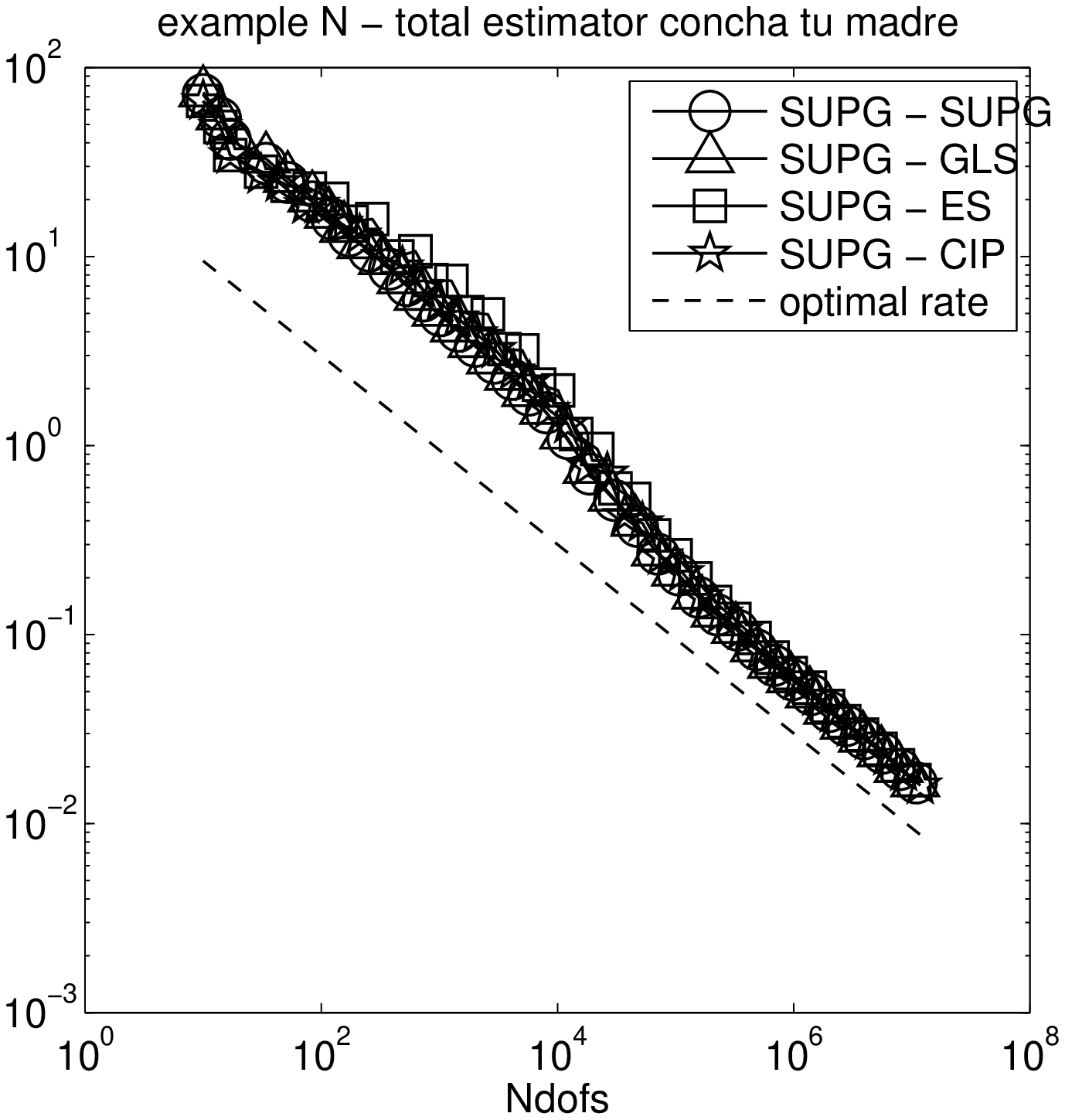}}
\scalebox{0.3}{\includegraphics{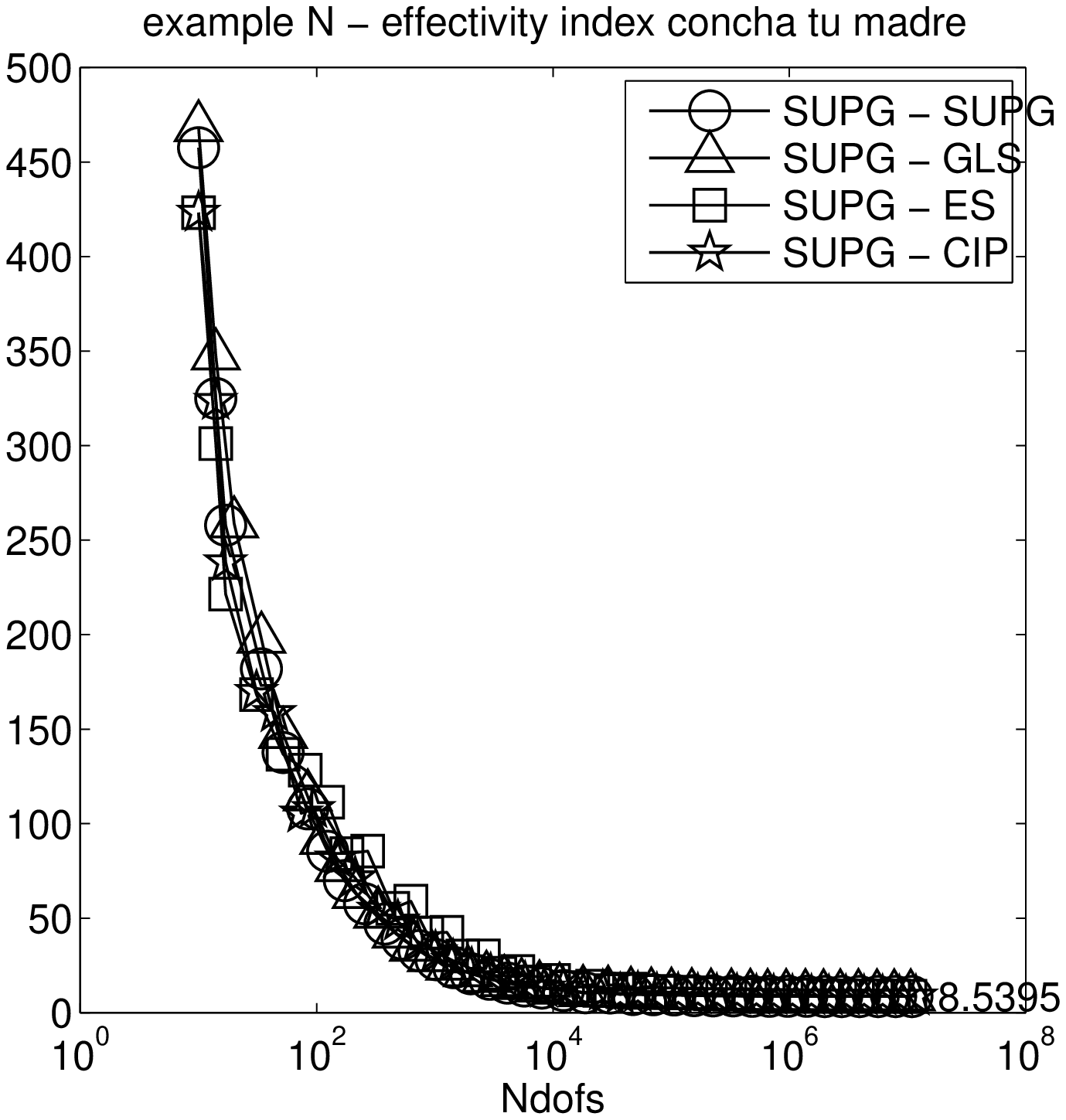}}
\scalebox{0.3}{\includegraphics{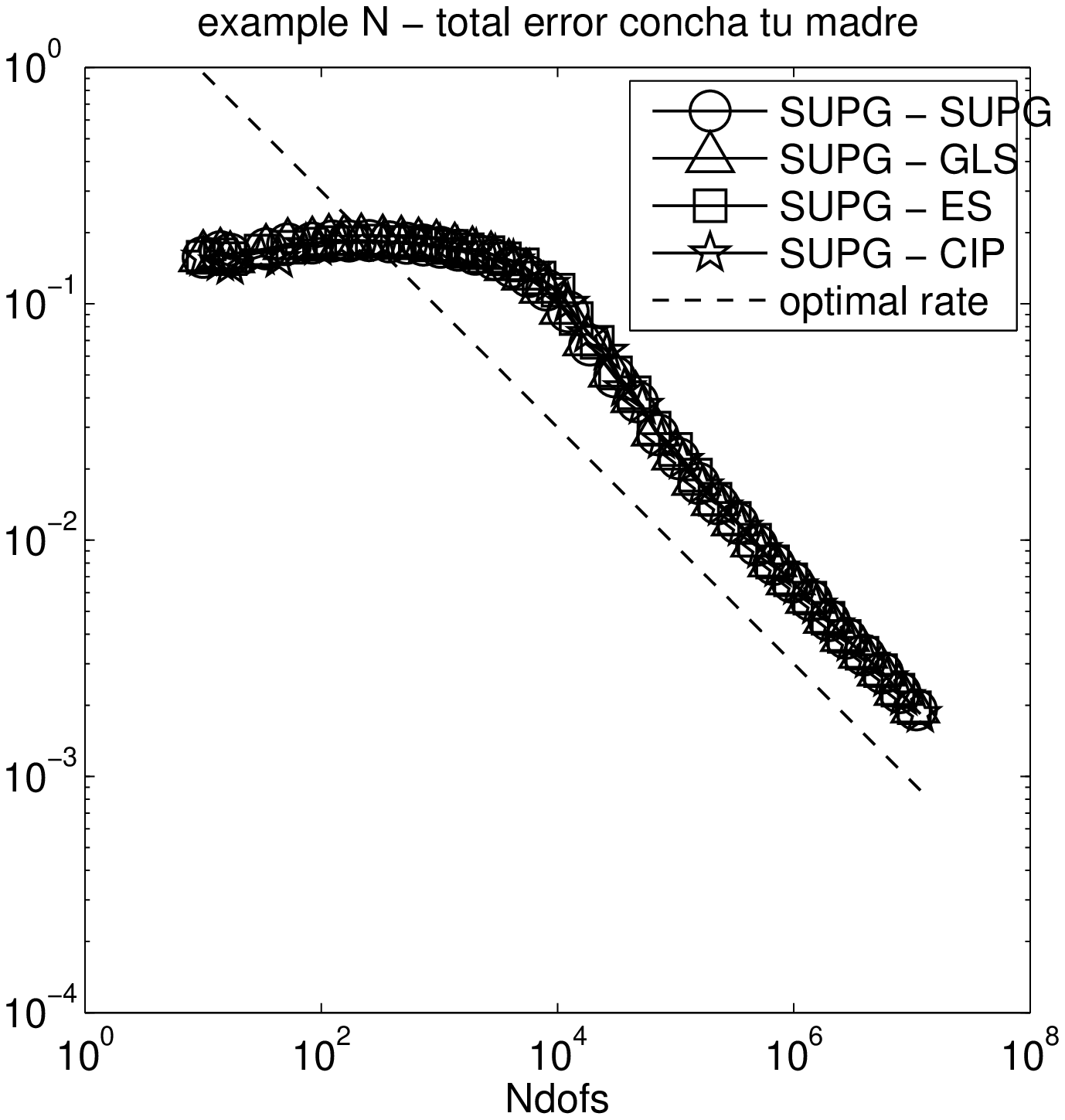}}
\scalebox{0.3}{\includegraphics{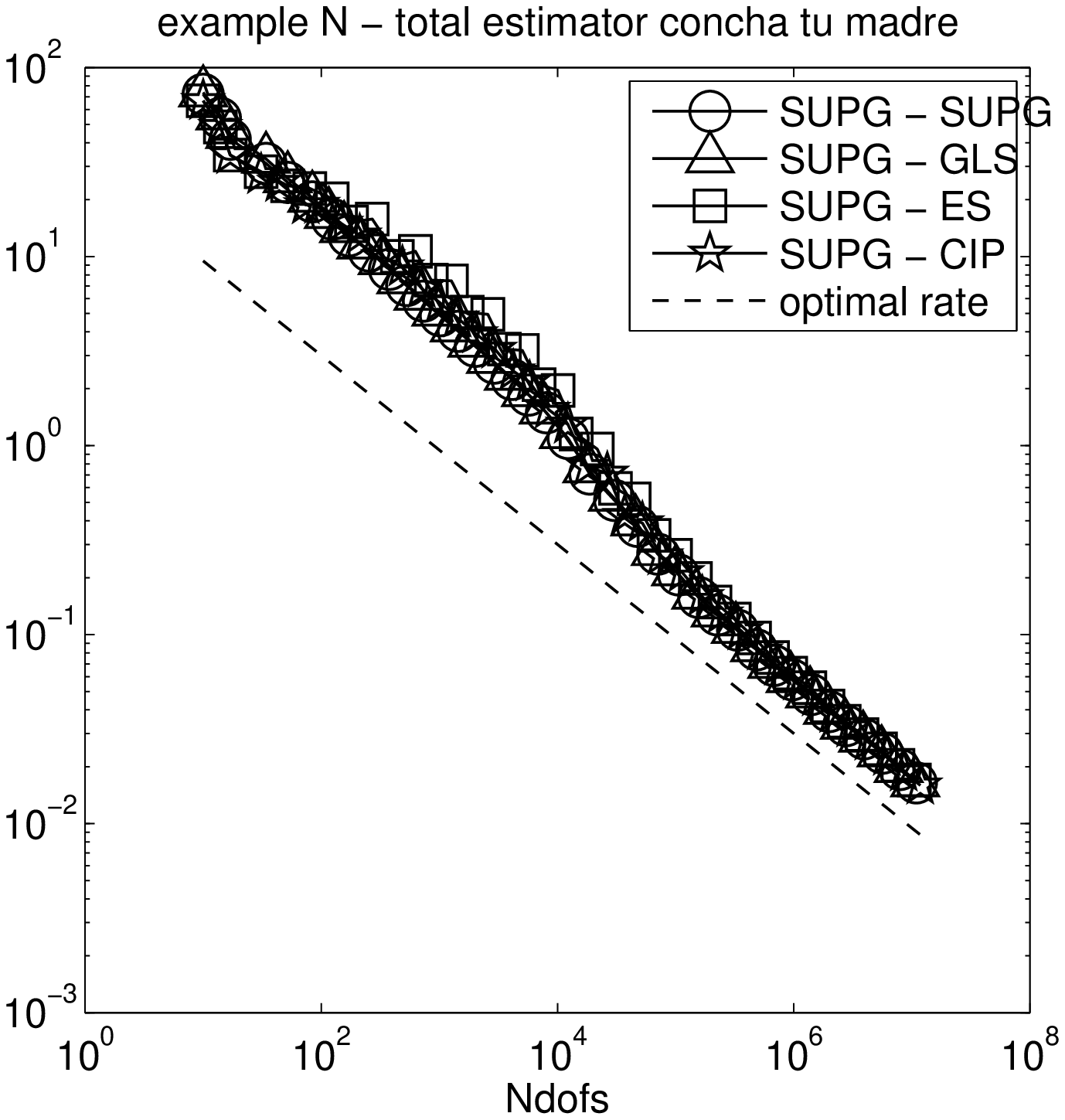}}
\scalebox{0.3}{\includegraphics{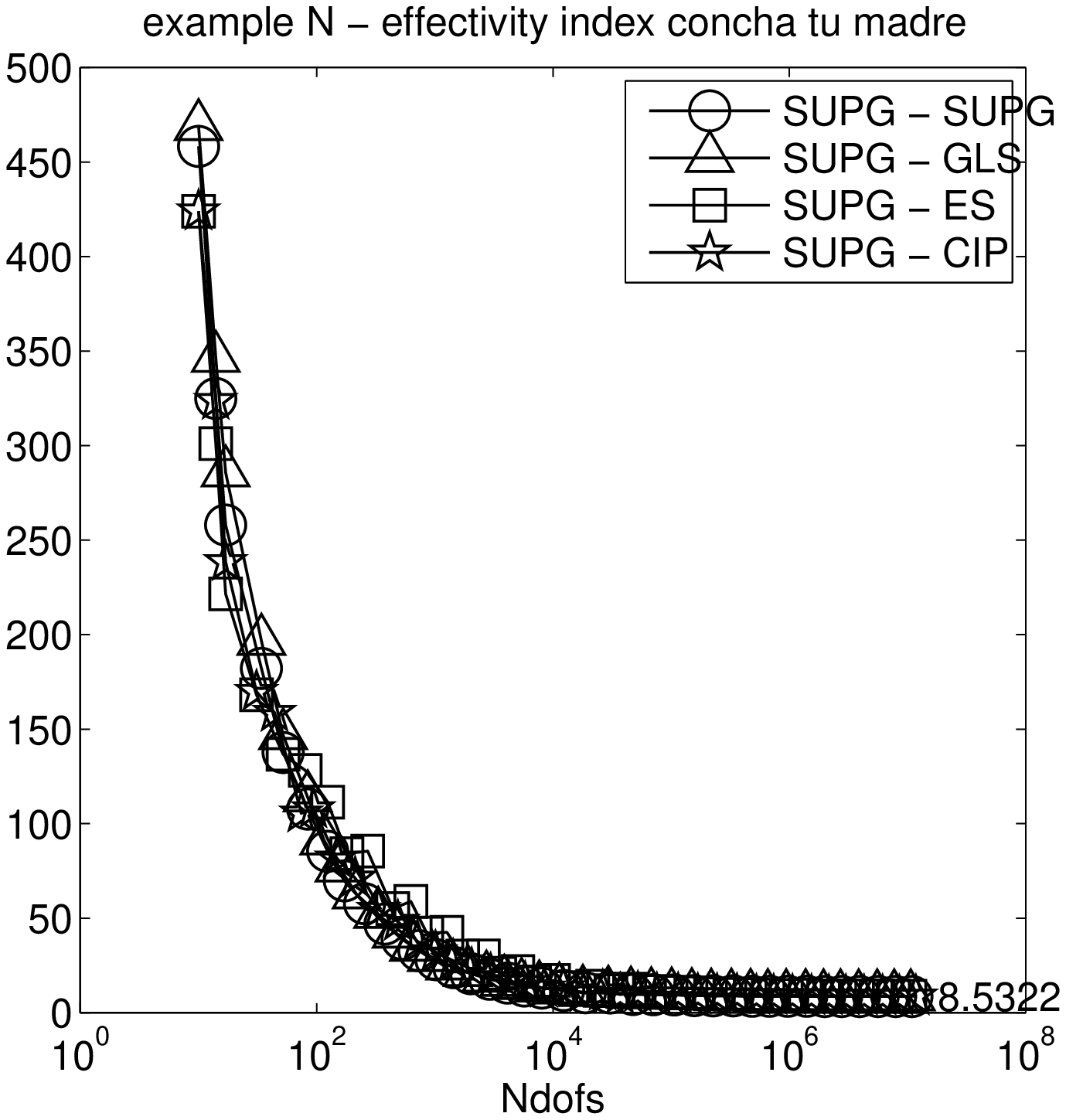}}
\end{center}
\caption{Example 1: The error $\|(e_{\bar{\ysf}},e_{\bar{\psf}},e_{\bar{\usf}})\|_{\Omega}$, estimator $\Upsilon$ and effectivity indices $\Upsilon/\|(e_{\bar{\ysf}},e_{\bar{\psf}},e_{\bar{\usf}})\|_{\Omega}$ obtained with $\mathsf{N}=19$ (top) and $\mathsf{N}=4$ (bottom).}
\label{FigE1}
\end{figure}

\begin{figure}[!htbp]
\begin{center}
\scalebox{0.3}{\includegraphics{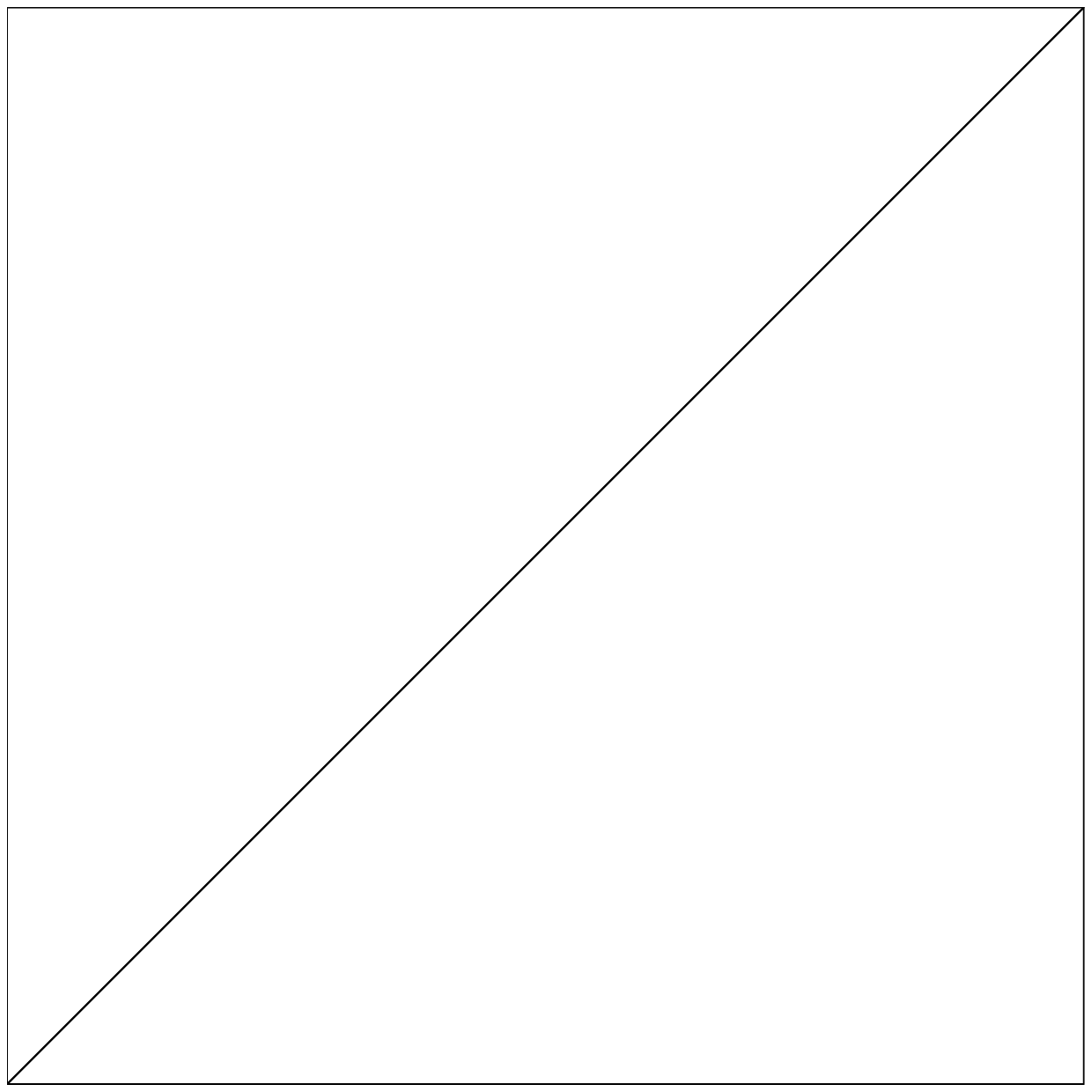}}
\scalebox{0.3}{\includegraphics{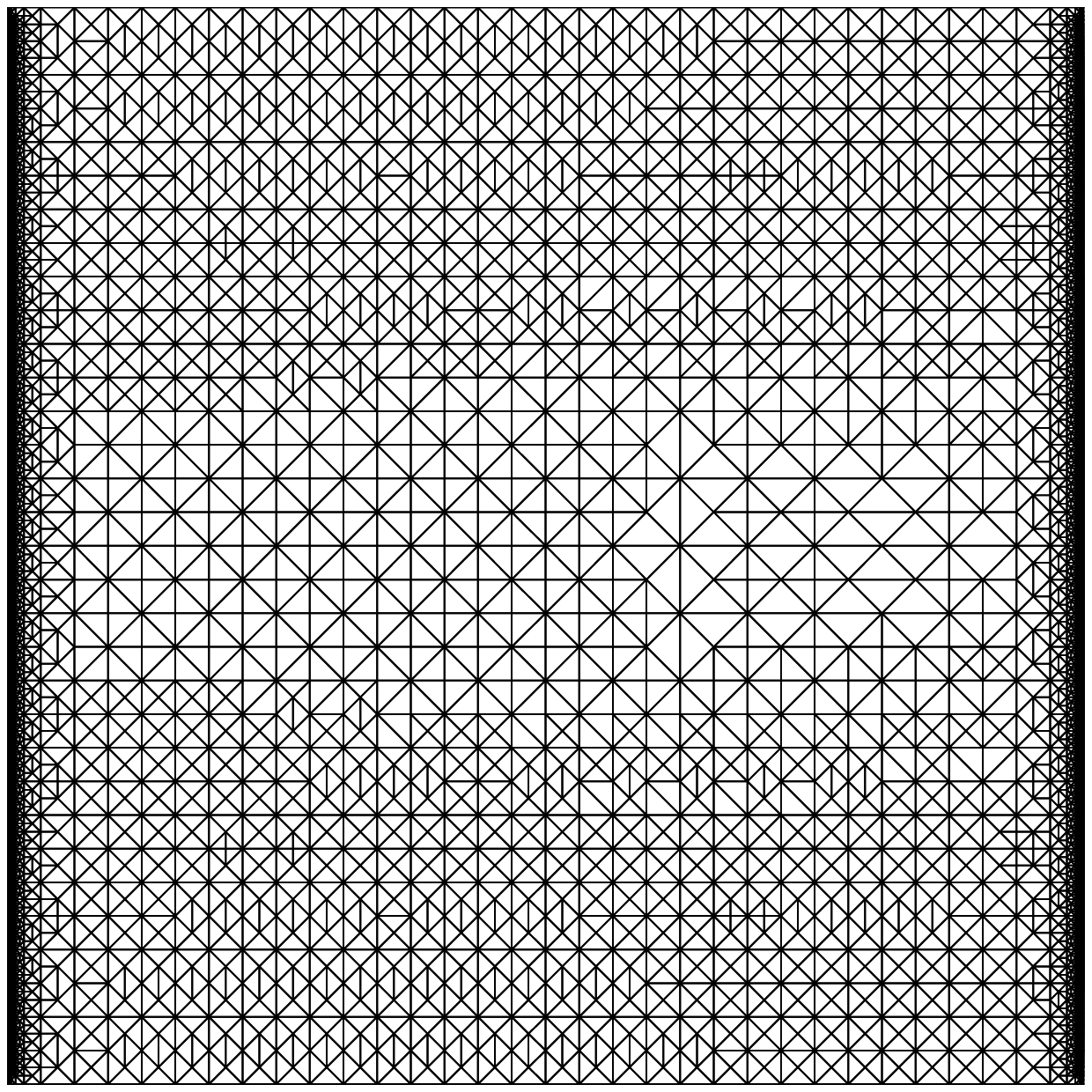}}
\scalebox{0.3}{\includegraphics{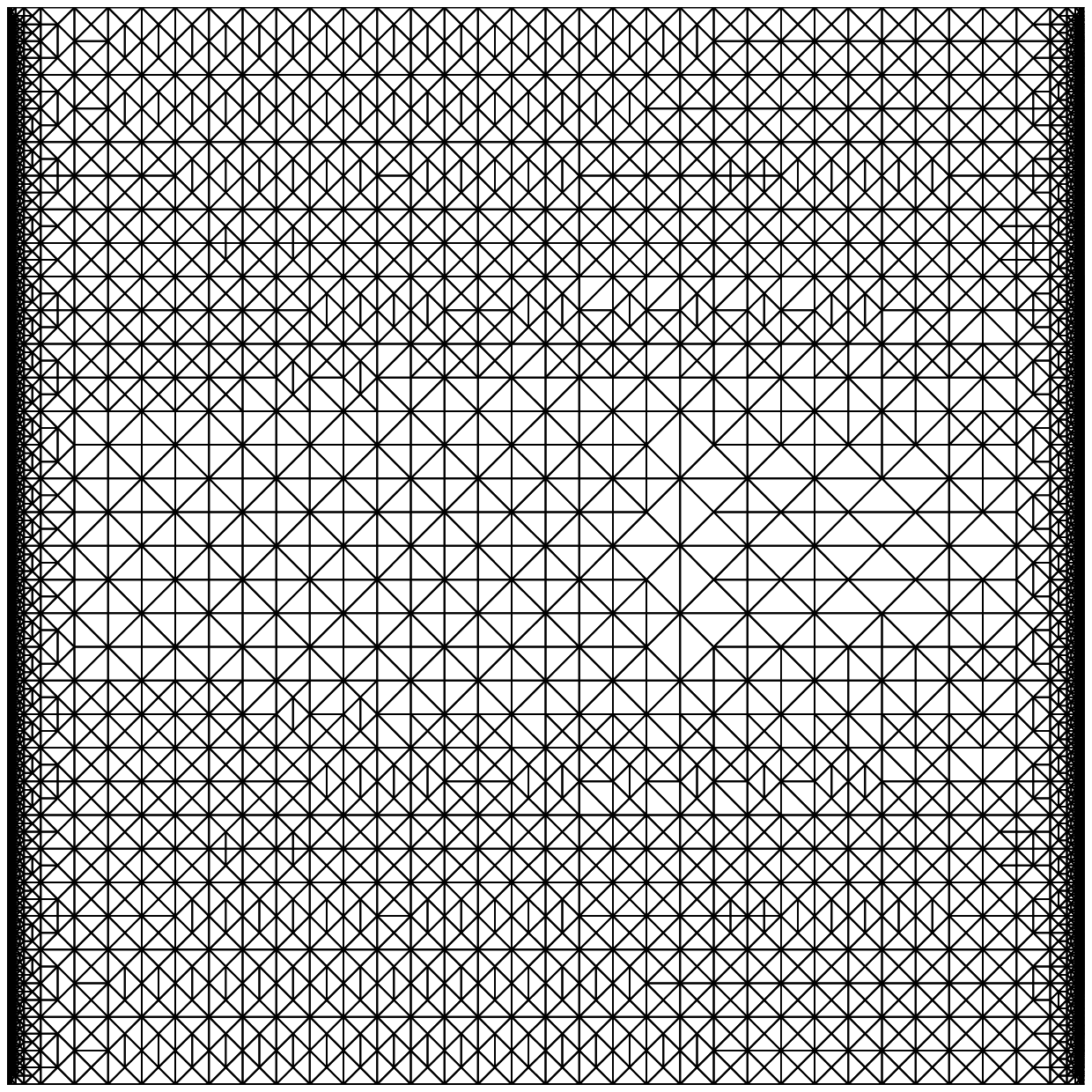}}
\end{center}
\caption{Example 1: The initial mesh (left) and the 25th adaptively refined meshes obtained with $\mathsf{N}=19$ (middle, $\mathsf{Ndof}=156476$) and $\mathsf{N}=4$ (right, $\mathsf{Ndof}=156508$).}
\label{Ex1Meshes}
\end{figure}

\psfrag{example N - total error concha tu madre}{\Huge Example 2 - $\|(e_{\bar{\ysf}},e_{\bar{\psf}},e_{\bar{\usf}})\|_{\Omega}$}
\psfrag{example N - total estimator concha tu madre}{\Huge Example 2 - Estimator $\Upsilon$}
\psfrag{example N - effectivity index concha tu madre}{\Huge Example 2 - $\Upsilon/\|(e_{\bar{\ysf}},e_{\bar{\psf}},e_{\bar{\usf}})\|_{\Omega}$}
\psfrag{SUPG - ES}{\Large SUPG-ES}
\psfrag{SUPG - CIP}{\Large SUPG-CIP}
\psfrag{SUPG - GLS}{\Large SUPG-GLS}
\psfrag{SUPG - SUPG}{\Large SUPG-SUPG}
\psfrag{optimal rate}{\Large Optimal rate}
\psfrag{Ndofs}{\Large Ndof}
\begin{figure}[!htbp]
\begin{center}
\scalebox{0.3}{\includegraphics{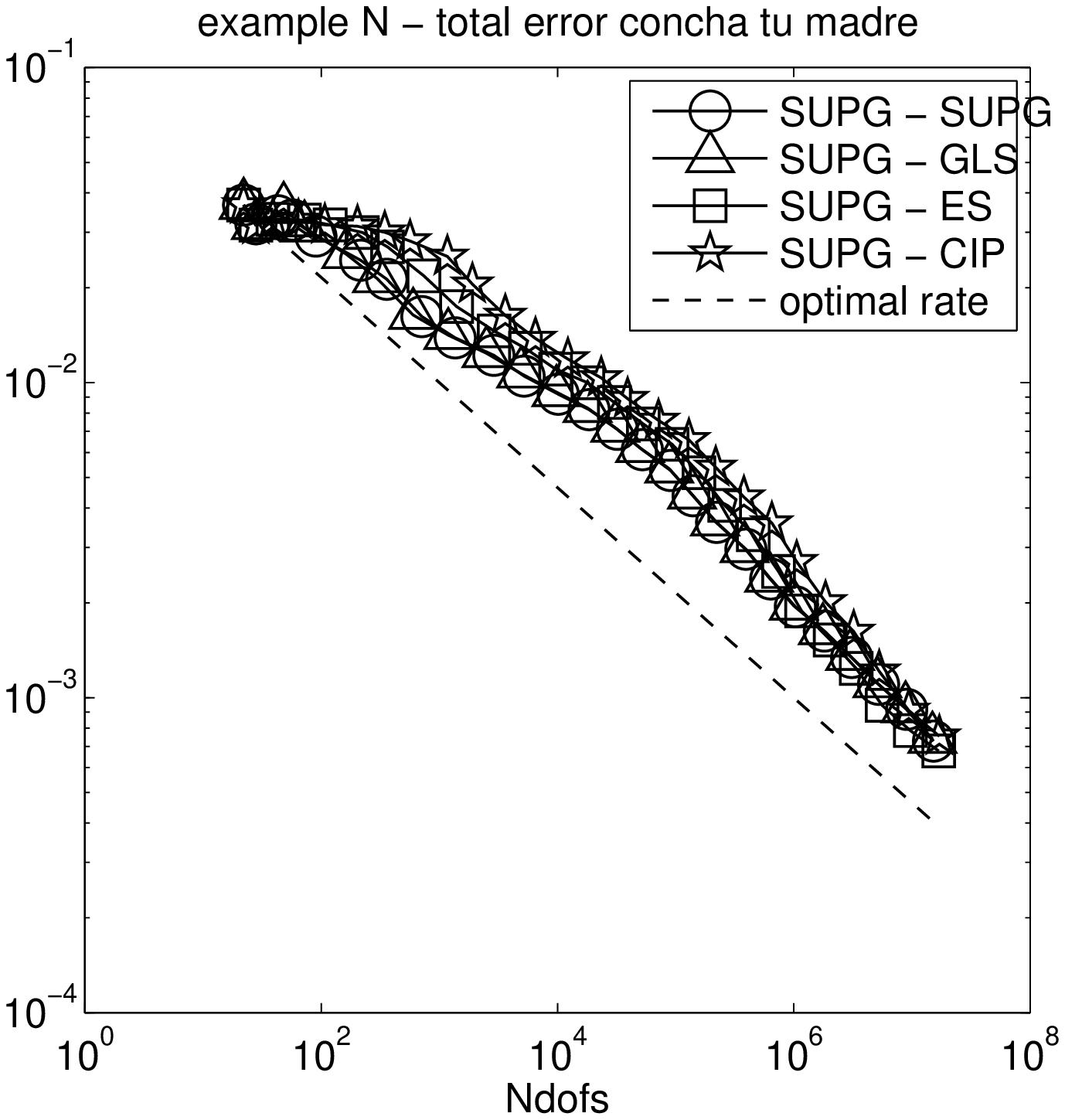}}
\scalebox{0.3}{\includegraphics{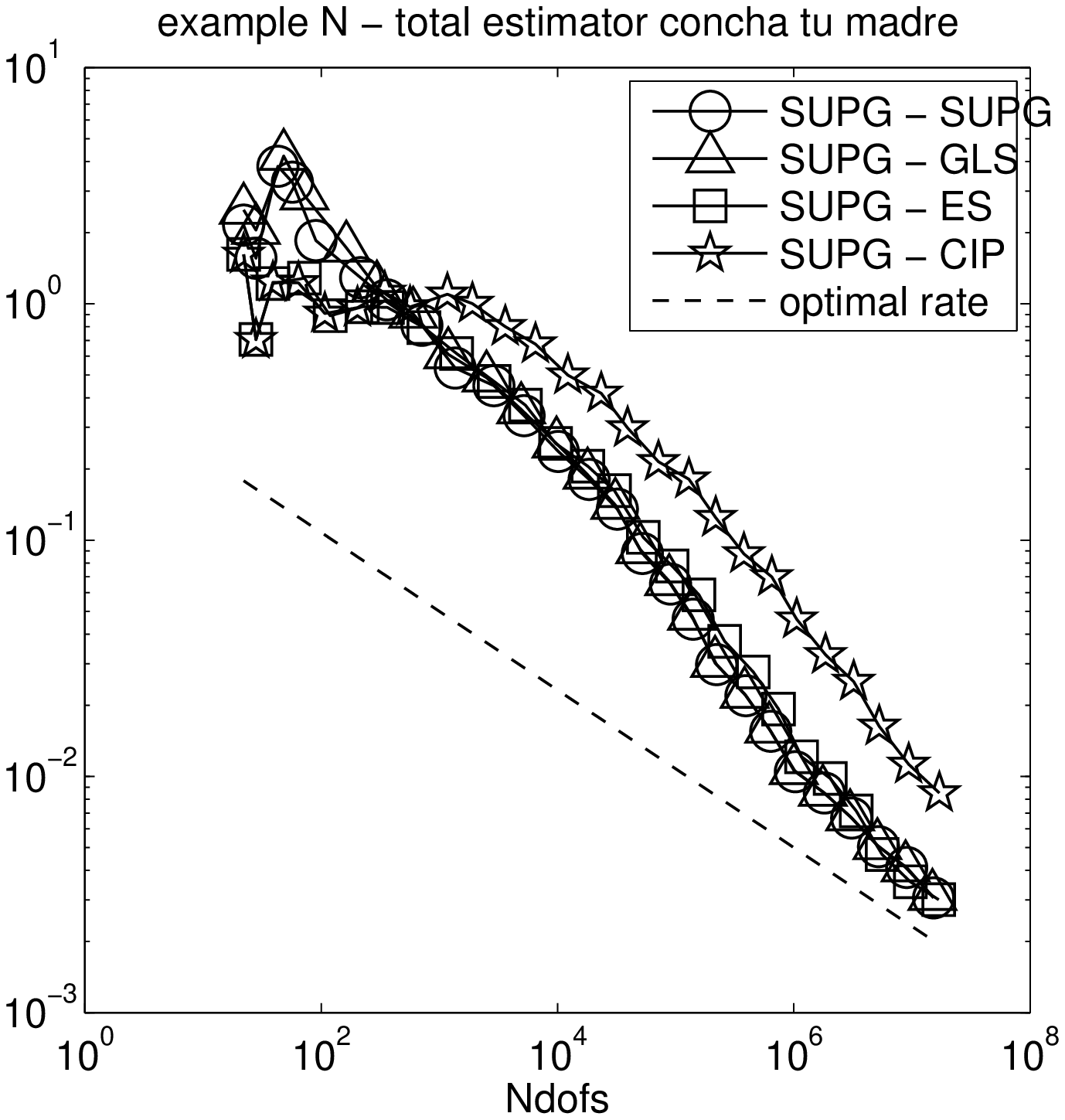}}
\scalebox{0.3}{\includegraphics{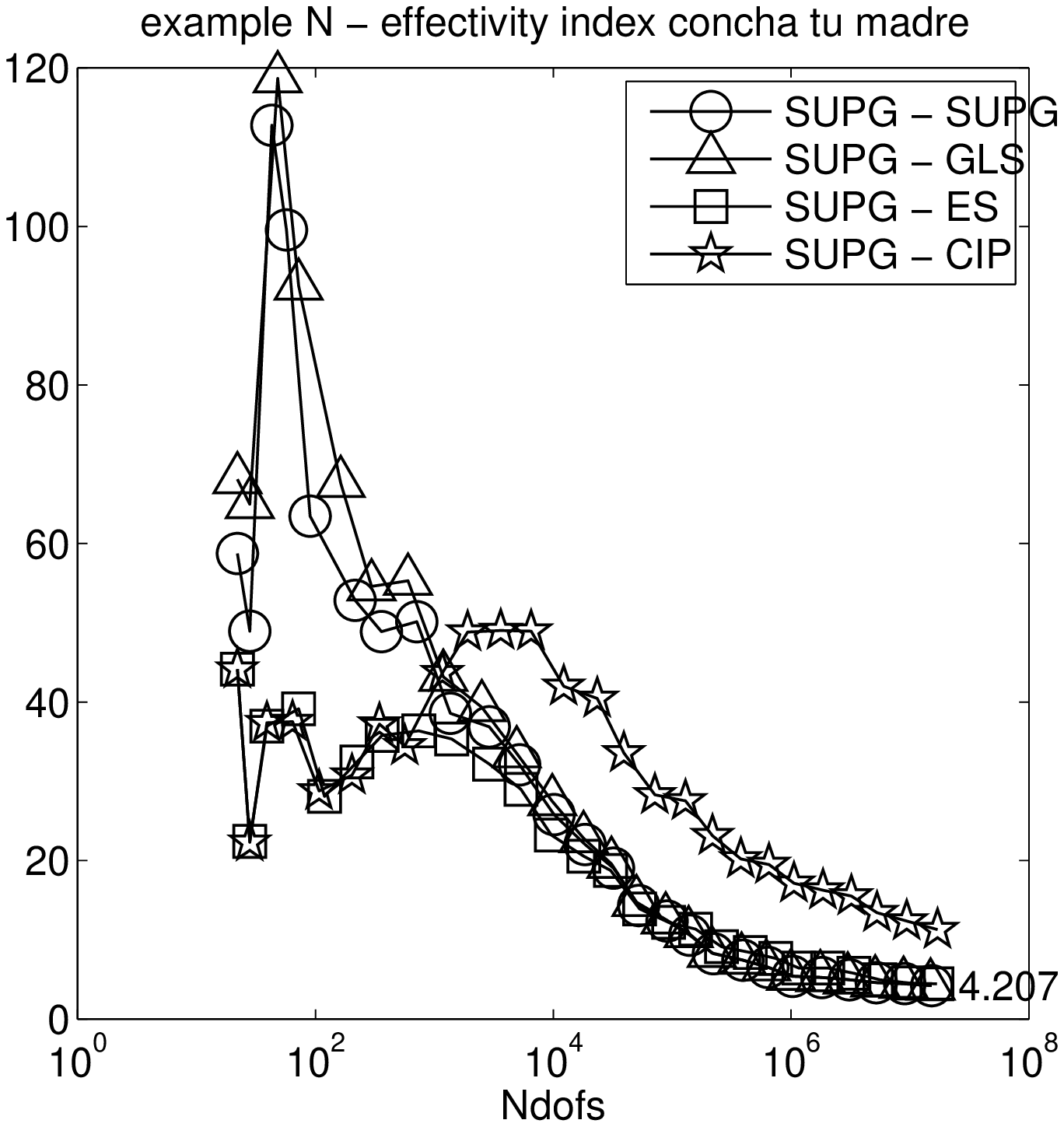}}
\scalebox{0.3}{\includegraphics{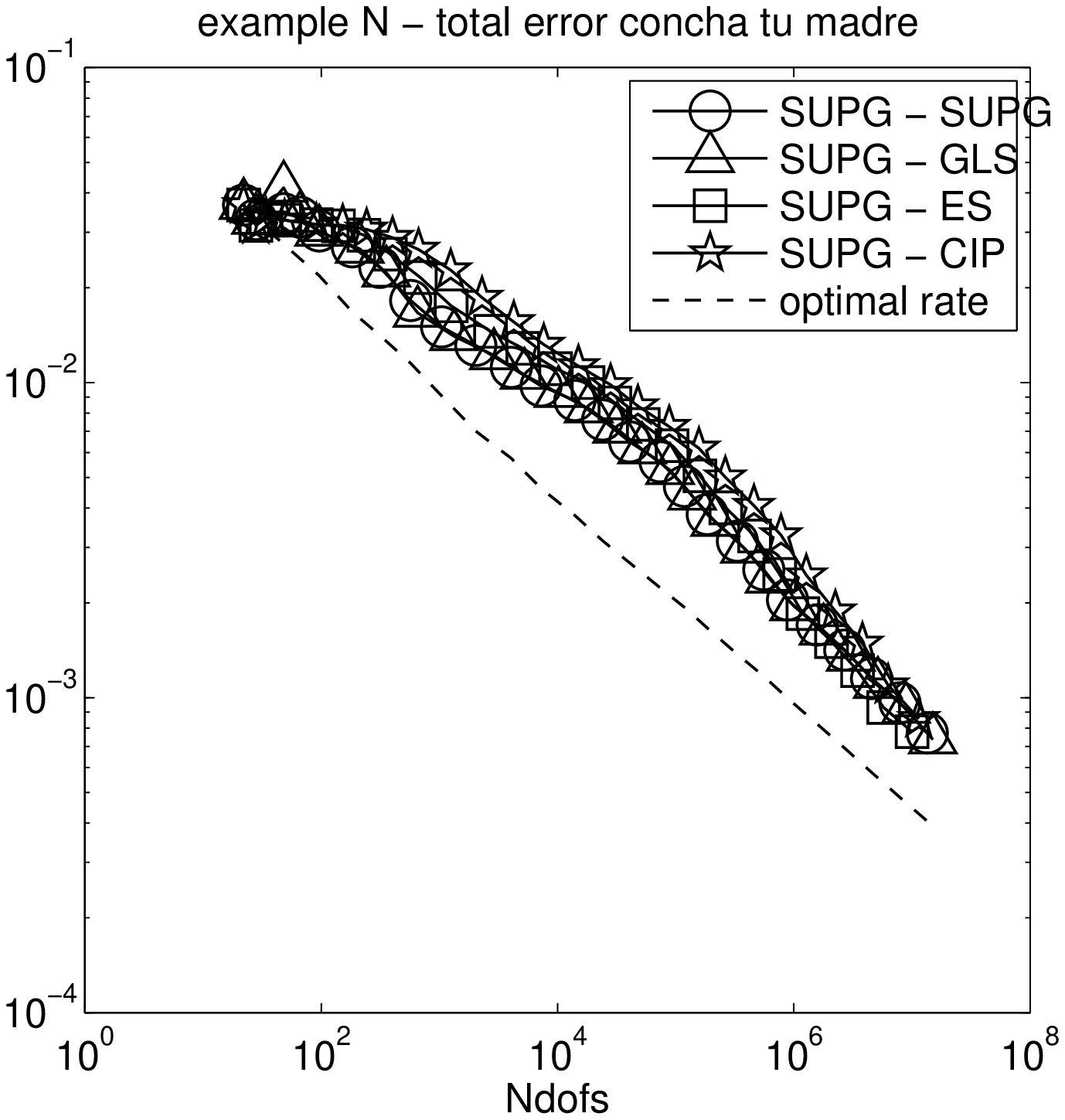}}
\scalebox{0.3}{\includegraphics{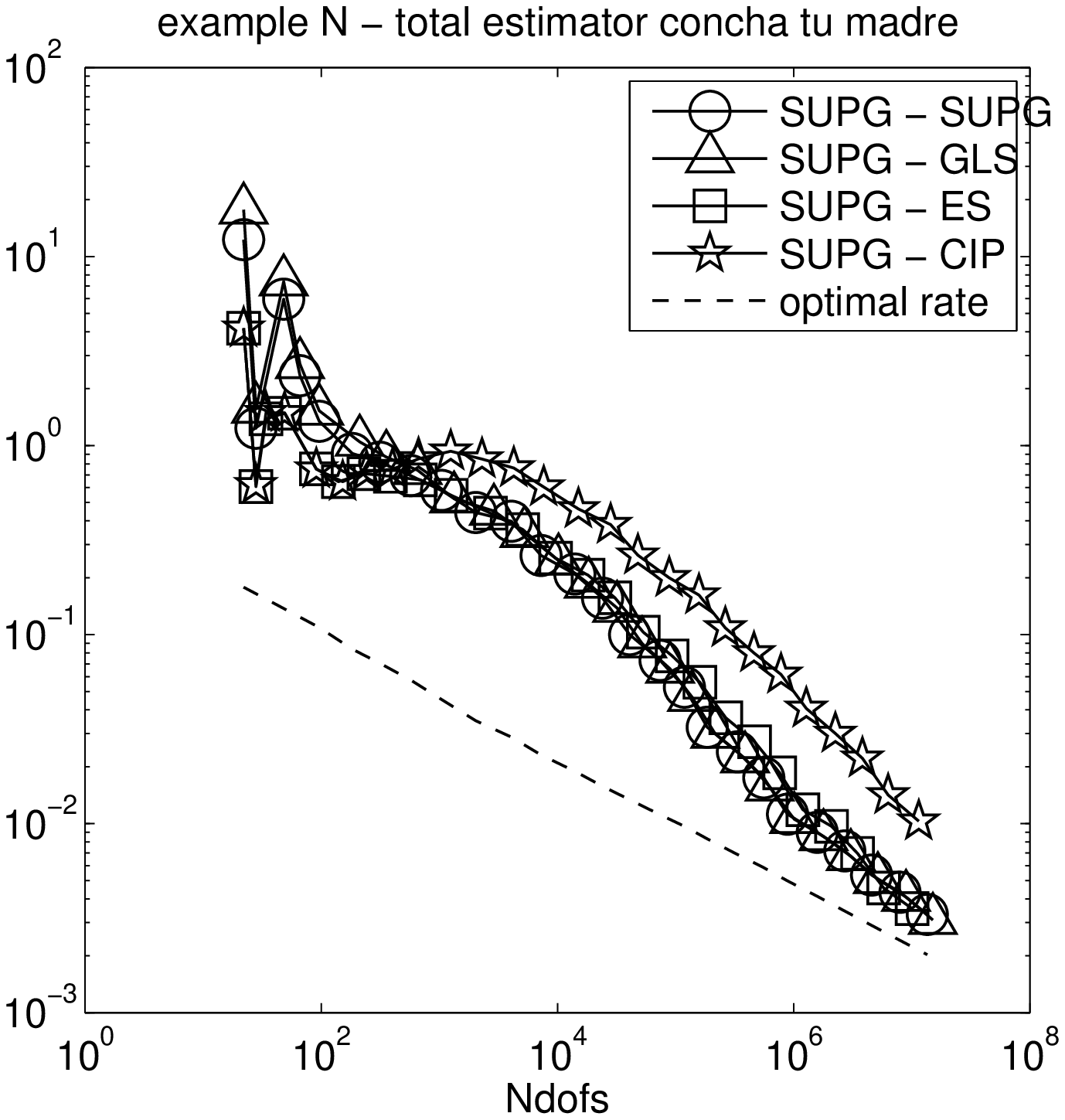}}
\scalebox{0.3}{\includegraphics{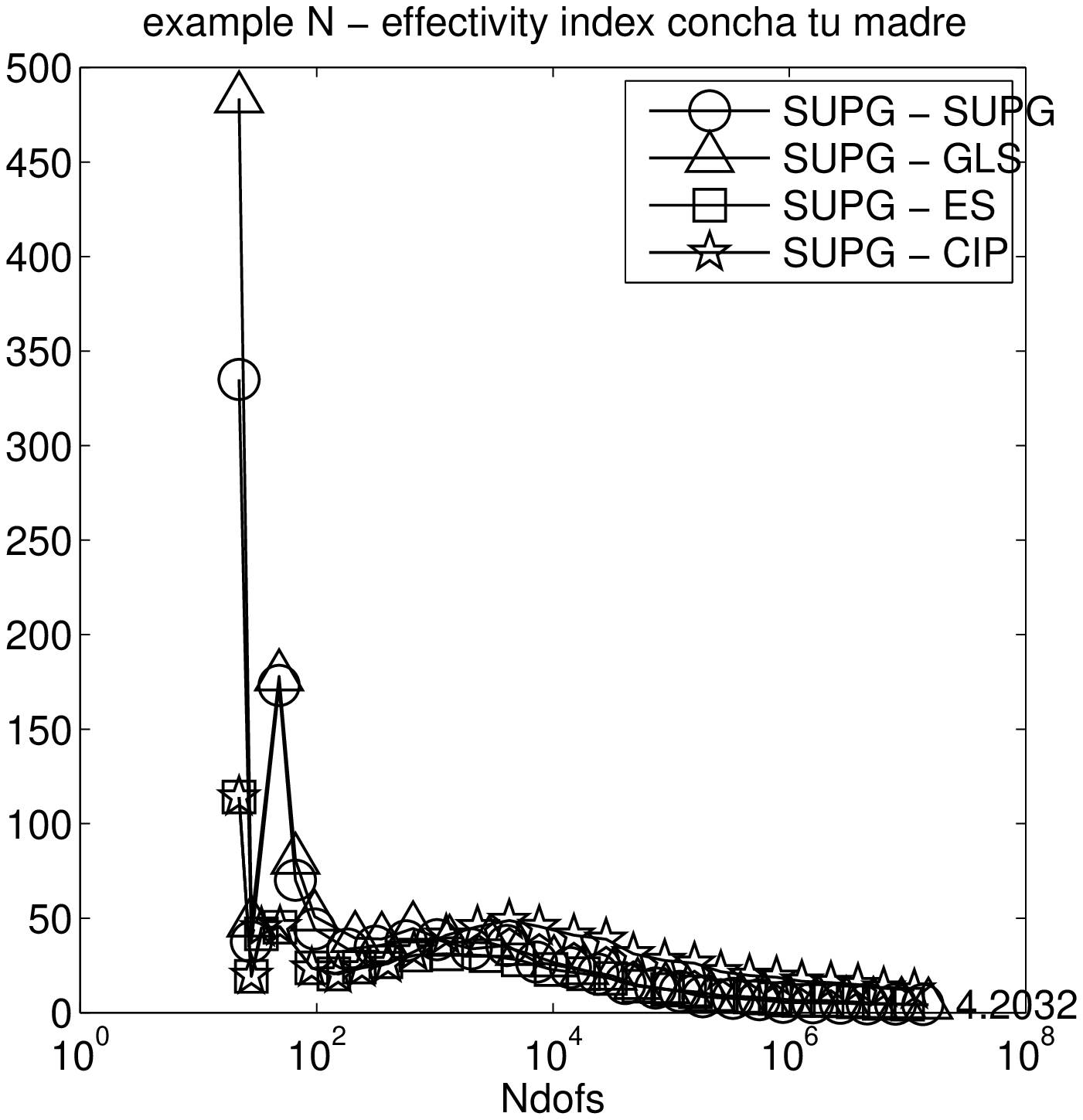}}
\end{center}
\caption{Example 2: The error $\|(e_{\bar{\ysf}},e_{\bar{\psf}},e_{\bar{\usf}})\|_{\Omega}$, estimator $\Upsilon$ and effectivity indices $\Upsilon/\|(e_{\bar{\ysf}},e_{\bar{\psf}},e_{\bar{\usf}})\|_{\Omega}$ obtained with $\mathsf{N}=14$ (top) and $\mathsf{N}=4$ (bottom).}
\label{FigE2}
\end{figure}

\begin{figure}[!htbp]
\begin{center}
\scalebox{0.25}{\includegraphics{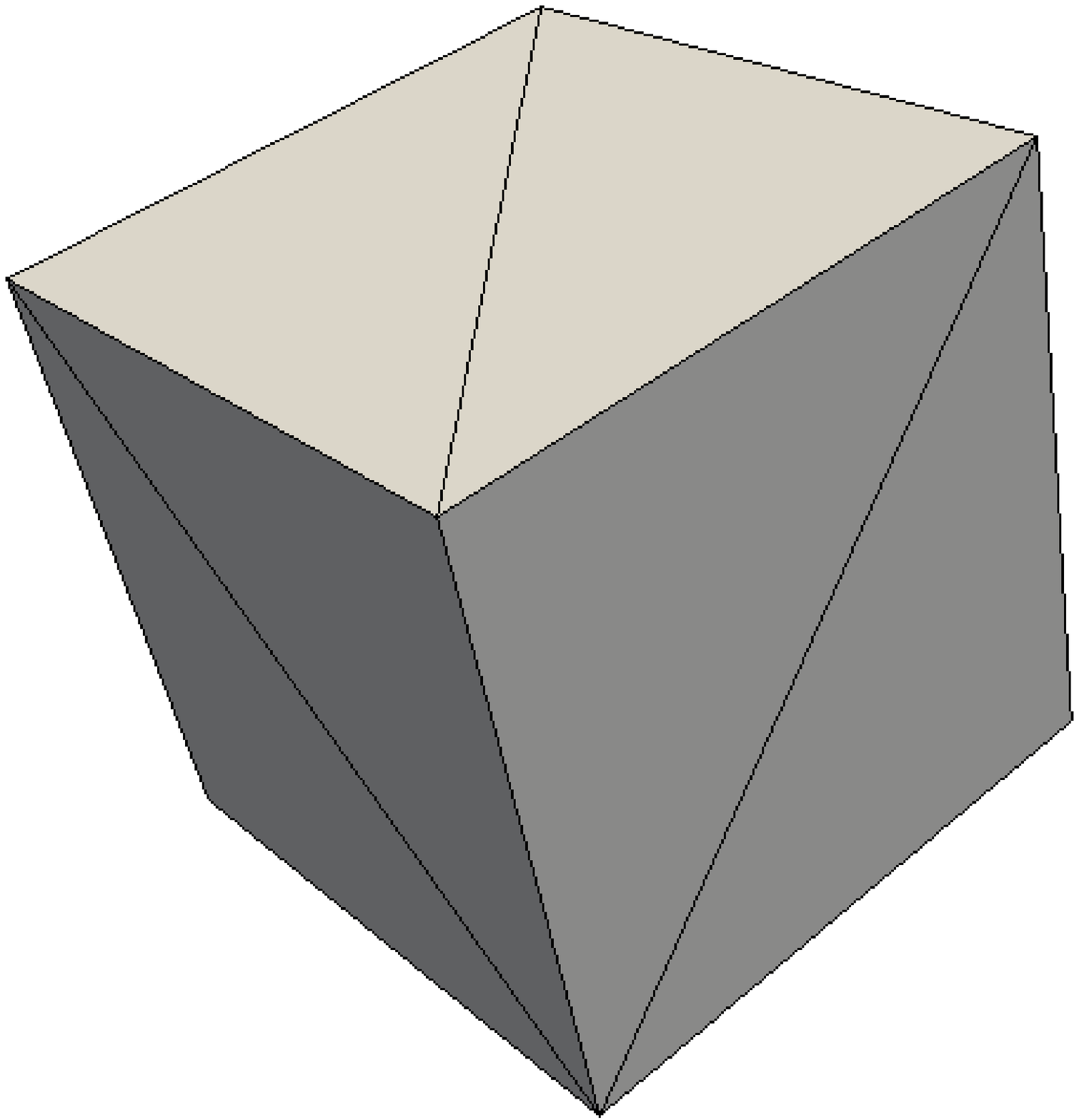}}
\scalebox{0.25}{\includegraphics{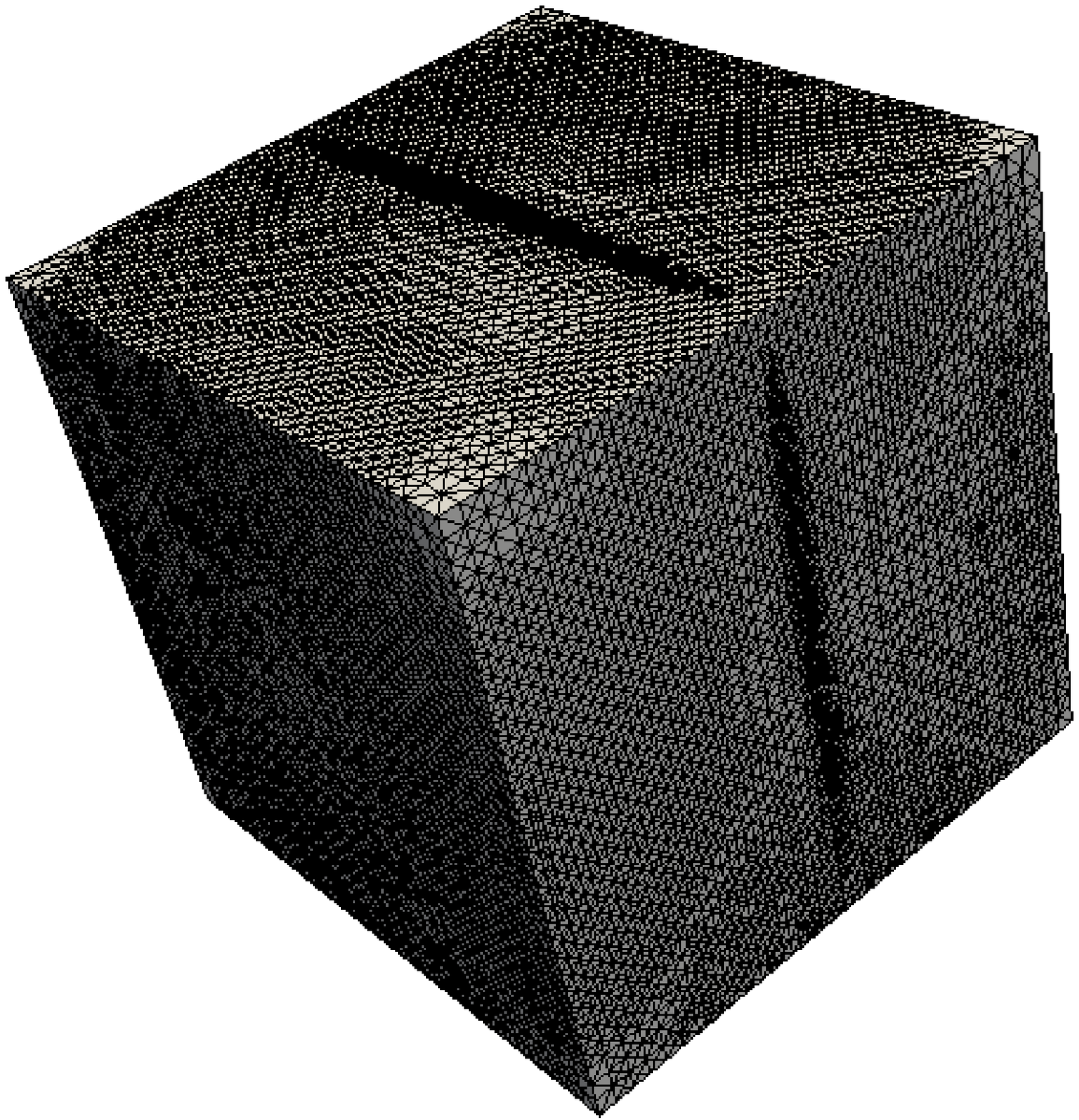}}
\scalebox{0.25}{\includegraphics{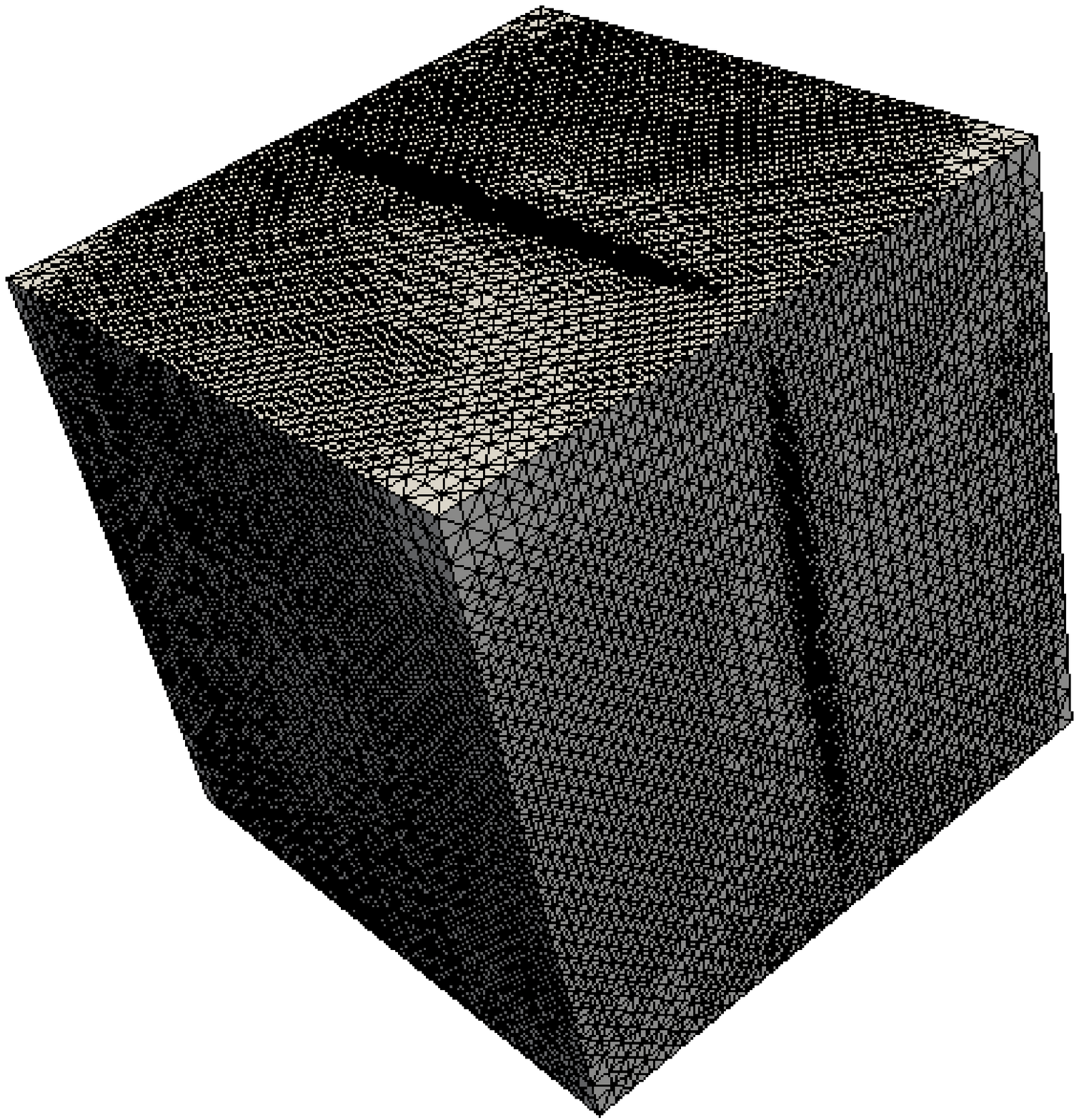}}
\end{center}
\caption{Example 2: The initial mesh (left) and the 25th adaptively refined meshes obtained with $\mathsf{N}=14$ (middle, $\mathsf{Ndof}=15261784$) and $\mathsf{N}=4$ (right, $\mathsf{Ndof}=13511617$).}
\label{Ex2Meshes}
\end{figure}

\section*{Acknowledgment}
The authors are indebted to Gabriel Barrenechea for his comments on an earlier version of this manuscript.
\bibliographystyle{siam}
\bibliography{bibliography}

\end{document}